\documentclass[a4paper,11pt,abstracton]{scrartcl} 

\oddsidemargin \evensidemargin
\title{Geometry and convergence of natural\\ policy gradient methods} 

\author{
\bfseries{Johannes M\"uller} \\
Max Planck Institute for Mathematics in the Sciences, Leipzig, Germany \\
\texttt{jmueller@mis.mpg.de} \\[.1cm]
\bfseries{Guido Mont\'ufar}\\
Departments of Mathematics and Statistics, UCLA, CA, USA\\
Max Planck Institute for Mathematics in the Sciences, Leipzig, Germany\\
\texttt{montufar@math.ucla.edu}\\
}
\setkomafont{author}{\rmfamily}
\date{\large\rmfamily\today}

\usepackage{amsthm}
\usepackage[numbers]{natbib} 
\usepackage{amsmath,amsfonts,bm}
\usepackage{tikz}
\usetikzlibrary{math}
\usetikzlibrary{arrows} 
\usepackage{dsfont}
\usepackage[]{todonotes}
\usepackage{enumitem}
\usepackage[normalem]{ulem}
\usepackage{url}
\usepackage{hyperref}
\usepackage{mathtools}
\usepackage{float}
\usepackage{titletoc}
\usepackage{comment}

\theoremstyle{definition}
\newtheorem{definition}{Definition}[]
\newtheorem{example}[definition]{Example}
\newtheorem{remark}[definition]{Remark}
\newtheorem{setting}[definition]{Setting}
\newtheorem{assumption}[definition]{Assumption}
\newtheorem*{setting*}{Setting}
\newtheorem*{assumption*}{Assumption}
\theoremstyle{plain}
\newtheorem{theorem}[definition]{Theorem}
\newtheorem{corollary}[definition]{Corollary}
\newtheorem{proposition}[definition]{Proposition}
\newtheorem{lemma}[definition]{Lemma}

\definecolor{ao}{rgb}{0.3, .7, 0.0}

\renewcommand{\AA}{\mathcal{A}}
\renewcommand{\SS}{\mathcal{S}}
\newcommand{\OO}{\mathcal{O}}

\tikzset{
>=stealth',
punkt/.style={
           rectangle,
           rounded corners,
           draw=black, very thick,
           text width=6.5em,
           minimum height=2em,
           text centered},
pil/.style={
           ->,
           thick,
           shorten <=2pt,
           shorten >=2pt,}
}

\begin{document}

\maketitle

\begin{abstract}
We study the convergence of several natural policy gradient (NPG) methods in infinite-horizon discounted Markov decision processes with regular policy parametrizations. For a variety of NPGs and reward functions we show that the trajectories in state-action space are solutions of gradient flows with respect to Hessian geometries, based on which we obtain global convergence guarantees and convergence rates. In particular, we show linear convergence for unregularized and regularized NPG flows with the metrics proposed by Kakade and Morimura and co-authors by observing that these arise from the Hessian geometries of conditional entropy and entropy respectively. Further, we obtain sublinear convergence rates for Hessian geometries arising from other convex functions like log-barriers. Finally, we interpret the discrete-time NPG methods with regularized rewards as inexact Newton methods if the NPG is defined with respect to the Hessian geometry of the regularizer. This yields local quadratic convergence rates of these methods for step size equal to the penalization strength. 

\end{abstract}

\section{Introduction}

Markov decision processes (MDPs) are an important model for sequential decision making in interaction with an environment and constitute a theoretical framework for modern reinforcement learning (RL). This framework has been successfully applied in recent years to solve increasingly complex tasks from robotics to board and video games~\cite{silver2016mastering, silver2017mastering, peters2003reinforcement, mnih2013playing, shao2019survey}. 
In MDPs the goal is to identify a \emph{policy} $\pi$, i.e., a procedure to select actions at every time step, which maximizes an expected time-aggregated reward $R(\pi)$. 
We will assume that the set of possible states $\SS$ and the set of possible actions $\AA$ are finite, 
and model the policy $\pi_\theta$ as a differentiably parametrized element in the polytope $\Delta_\AA^\SS$ of conditional probability distributions of actions given states, 
with $\pi_\theta(a|s)$ specifying the probability of selecting action $a\in\AA$ when currently in state $s\in\SS$, for the parameter value $\theta$. 
We will study gradient-based policy optimization methods and more specifically \emph{natural policy gradient (NPG)} methods. 
Inspired by the seminal works of Amari~\cite{amari1998natural, amari1998why}, various NPG methods have been proposed~\cite{kakade2001natural, morimura2008new, moskovitz2021efficient}. 
In general, they take the form 
    \[ 
    \theta_{k+1} = \theta_k + \Delta t G(\theta_k)^+\nabla R(\theta_k), 
    \]
where $G(\theta)^+$ denotes the Moore-Penrose pseudo inverse and $G(\theta)_{ij} = g(dP_\theta e_i, dP_\theta e_j)$ is a Gram matrix defined with respect to some Riemannian metric $g$ and some representation $P(\theta)$ of the parameter. 
Most of our analysis does not actually depend on the specific choice of the pseudo inverse, but in Section~\ref{sec:convergenceDiscrete} we will use the Moore-Penrose pseudo inverse. 
The most traditional natural gradient method is the special case where $P(\theta)$ is a probability distribution and $g$ is the Fisher information in the corresponding space of probability distributions. However, the terminology may be used more generally to refer to a Riemannian gradient method where the metric is in some sense natural. 
Kakade~\cite{kakade2001natural} proposed using $P(\theta) = \pi_\theta$ and taking for $g$ a product of Fisher metrics weighted by the state probabilities resulting from running the Markov process with policy $\pi_\theta$. 
Although this is a natural choice for $P$, the choice of a Riemannian metric on $\Delta^\SS_\AA$ is a non trivial problem. Peters et al.~\cite{peters2003reinforcement} offered reasons to regard Kakade's metric as the true Fisher metric in this case, yet other choices of the weights can be motivated by axiomatic approaches to define a Fisher metric of conditional probabilities \cite{Lebanon2005AxiomaticGO, montufar2014fisher}. 
From our perspective, a main difficulty is that it is not clear how to choose a Riemannian metric on $\Delta^\SS_\AA$ that interacts nicely with the objective function $R(\pi)$, which is a non-convex rational function of $\pi\in\Delta^\SS_\AA$. 
An alternative choice for $P(\theta)$ is the vector of \emph{state-action frequencies} $\eta_\theta$, whose components $\eta_\theta(s,a)$ are the
probabilities of state-action pairs $(s,a)\in\SS\times\AA$ resulting from running the Markov process with policy $\pi_\theta$. 
Morimura et al.~\cite{morimura2008new} proposed using $P(\theta)=\eta_\theta$ and the Fisher information on the state-action probability simplex $\Delta_{\SS\times\AA}$ as a Riemannian metric. 
We will study both approaches and variants from the perspective of Hessian geometry.

\paragraph{Contributions}
We study the natural policy gradient dynamics inside the polytope $\mathcal{N}$ of state-action frequencies, which provides a unified treatment of several existing NPG methods. 
We focus on finite state and action spaces and the expected infinite-horizon discounted reward optimized over the set of memoryless stochastic policies. 
\begin{itemize}
\item 
We show that the dynamics of Kakade's NPG and Morimura's NPG solve a gradient flow in $\mathcal N$ with respect to the Hessian geometries of conditional entropic and entropic regularization of the reward (Sections~\ref{subsubsec:KakadeGeomtry} and~\ref{subsec:morimura} and Proposition~\ref{prop:evolutionStateAction}). 

\item 
Leveraging results on gradient flows in Hessian geometries, we derive linear convergence rates for Kakade's and Morimura's NPG flow for the unregularized reward, which is a linear and hence not strictly concave function in state-action space, and also for regularized reward (Theorems~\ref{thm:linearconvergenceKakade} and \ref{thm:sigmaNPGflowfastrates} and Corollaries \ref{cor:kakadeLinearConvReg} and \ref{cor:sigmaLinearConvReg}). 

\item 
Further, for a class of NPG methods which correspond to $\beta$-divergences and which generalize Morimura's NPG, we show sub-linear convergence in the unregularized case and linear convergence in the regularized case (Theorem~\ref{thm:sigmaNPGflowfastrates} and Corollary~\ref{cor:sigmaLinearConvReg}, respectively). 

\item 
We complement our theoretical analysis with experimental evaluation, which indicates that the established linear and sub-linear rates for unregularized problems are essentially tight. 

\item 
For discrete-time gradient optimization, our ansatz in state-action space yields an interpretation of the regularized NPG method as an inexact Newton iteration if the step size is equal to the regularization strength. 
This yields a relatively short proof for the local quadratic convergence of regularized NPG methods with Newton step sizes (Theorem~\ref{thm:localQuadraticConvergence}). 
This recovers as a special case the local quadratic convergence of Kakade's NPG under state-wise entropy regularization previously shown in~\cite{cen2021fast}. 
\end{itemize} 

\paragraph{Related work}
The application of natural gradients to optimization in MDPs was first proposed by Kakade~\cite{kakade2001natural}, taking as a metric on $\Delta^\SS_\AA$ the product of Fisher metrics on $\Delta^s_\AA$, $s\in\SS$, weighted by the stationary state distribution. 
The relation of this metric to finite-horizon Fisher information matrices was studied by Bagnell and Schneider~\cite{bagnell2003covariant} as well as by Peters et al.~\cite{peters2003reinforcement}. 
Later, Morimura et al.~\cite{morimura2008new} proposed a natural gradient using the Fisher metric on the state-action frequencies, which are probability distributions over states and actions. 

There has been a growing number of works studying the convergence properties of policy gradient methods. 
It is well known that reward optimization in MDPs is a challenging problem, where both the non-convexity of the objective function with respect to the policy and the particular parametrization of the policies can lead to the existence of suboptimal critical points~\cite{bhandari2019global}. Global convergence guarantees of gradient methods require assumptions on the parametrization. Most of the existing results are formulated for tabular softmax policies, but more general sufficient criteria have been given in~\cite{bhandari2019global, zhang2020global, zhang2022convergence}. 

Vanilla PGs have been shown to converge sublinearly at rate $O(t^{-1})$ for the unregularized reward and linearly for entropically regularized reward. 
For unregularized problems, the convergence rate can be improved to a linear rate by normalization~\cite{mei2020global, mei2021leveraging}. 
For continuous state and action spaces, vanilla PG converges linearly for entropic regularization and shallow policy networks in the mean-field regime~\cite{leahy2022convergence}. 

For Kakade's NPG, \cite{agarwal2021theory} established sublinear convergence rate $O(t^{-1})$ for unregularized problems, and the result has been improved to a linear rate of convergence for step sizes found by exact line search~\cite{bhandari2021linear}, constant step sizes~\cite{khodadadian2021linear, alfano2022linear, yuan2022linear}, and for geometrically increasing step sizes~\cite{xiao2022convergence}. 
For regularized problems, the method converges linearly for small step sizes and locally quadratically for Newton-like step size~\cite{cen2021fast}. 
These results have been extended to more general frameworks using state-mixtures of Bregman divergences on the policy polytope~\cite{lan2022policy, zhan2021policy, li2021quasi}, which however do not include NPG methods defined in state-action space such as Morimura’s NPG. 
For projected PGs, \cite{agarwal2021theory} shows sublinear convergence at a rate $O(t^{-1/2})$, and the result has been improved to a sublinear rate $O(t^{-1})$ \cite{xiao2022convergence}, and to a linear rate for step sizes chosen by exact line search~\cite{bhandari2021linear}. 

Apart from the works on convergence rates for policy gradient methods for standard MDPs, a primal-dual NPG method with sublinear global convergence guarantees has been proposed for constrained MDPs~\cite{ding2020natural, ding2022convergence}. 
For partially observable systems, a gradient domination property has been established  in~\cite{azizzadenesheli2018policy}. 
NPG methods with dimension-free global convergence guarantees have been studied for multi-agent MDPs and potential games~\cite{alfano2021dimension}. 
The sample complexity of a Bregman policy gradient arising from a strongly convex function in parameter space has been studied in~\cite{huang2021bregman}. 
For the linear quadratic regulator, global linear convergence guarantees for vanilla, Gauss-Newton and Kakade’s natural policy gradient methods are provided in~\cite{fazel2018global}; 
note that this setting is different to reward optimization in MDPs, where the objective at a fixed time is linear and not quadratic. 
A lower bound of $O(\eta^{-1}\lvert \mathcal S\rvert^{2^{\Omega((1-\gamma)^{-1})}})$ on the iteration complexity for softmax PG method with step size $\eta$ has been established in~\cite{li2021softmax}. 


\paragraph{Notation}
We denote the simplex of probability distributions over a finite set \(\mathcal X\) by \(\Delta_{\mathcal X}\). An element \(\mu\in\Delta_{\mathcal X}\) is a vector with non-negative entries \(\mu_x = \mu(x)\), $x\in\mathcal{X}$ adding to one, $\sum_x \mu_x=1$. We denote the set of Markov kernels from a finite set $\mathcal X$ to another finite set $\mathcal{Y}$ by \(\Delta_{\mathcal Y}^{\mathcal X}\). An element \(Q\in \Delta_{\mathcal Y}^{\mathcal X}\) is a $|\mathcal{X}|\times |\mathcal{Y}|$ row stochastic matrix with entries \(Q_{xy} = Q(y|x)\), $x\in\mathcal{X}$, $y\in\mathcal{Y}$. Given \(Q^{(1)}\in\Delta_{\mathcal Y}^{\mathcal X}\) and \(Q^{(2)}\in\Delta_{\mathcal Z}^{\mathcal Y}\) we denote their composition into a kernel from $\mathcal{X}$ to $\mathcal{Z}$ by \(Q^{(2)}\circ Q^{(1)}\in\Delta_{\mathcal Z}^{\mathcal X}\). Given $p\in \Delta_{\mathcal{X}}$ and $Q\in\Delta^{\mathcal{X}}_{\mathcal{Y}}$ we denote their composition into a joint probability distribution by $p\ast Q \in \Delta_{\mathcal{X}\times\mathcal{Y}}$, $(p\ast Q)(x, y)\coloneqq p(x)Q(y|x)$. The support of a vector $v\in \mathbb{R}^\mathcal{X}$ is the set $\operatorname{supp}(v) = \{x\in\mathcal{X}\colon v_x\neq0\}$. 

For a vector $\mu\in\mathbb R_{\ge0}^\mathcal X$ we denote its \emph{Shannon entropy} by
    \[ H(\mu)\coloneqq -\sum_x \mu(x)\log(\mu(x)), \]
with the usual convention that $0\log(0)\coloneqq 0$. For $\mu\in\mathbb R_{\ge0}^{\mathcal X\times\mathcal Y}$ we denote the $X$-marginal by $\mu_X\in 
\mathbb{R}^\mathcal{X}_{\geq0}$, where $\mu_X(x)\coloneqq \sum_y \mu(x, y)$. Further, we denote the
\emph{conditional entropy} of $\mu$ conditioned on $X$ by
\begin{equation}\label{eq:conditionalEntropy}
    H(\mu|\mu_X) \coloneqq -\sum_{x, y} \mu(x, y) \log\frac{\mu(x,y)}{\mu_X(x)} = H(\mu) - H(\mu_X).
\end{equation}
For any strictly convex function $\phi\colon\Omega\to\mathbb R$ defined on a convex subset $\Omega\subseteq \mathbb R^d$, the associated \emph{Bregman divergence $D_\phi\colon\Omega\times\Omega\to\mathbb R$} is given by $D_\phi(x, y) \coloneqq \phi(x) - \phi(y) - \langle \nabla \phi(y), x-y\rangle$.

Given two smooth manifolds $\mathcal M$ and $\mathcal N$ and a smooth function $f\colon\mathcal M\to\mathcal N$, we denote the differential of $f$ at $p\in\mathcal M$ by $df_p\colon T_p\mathcal M\to T_{f(p)}\mathcal N$. 
In the Euclidean case, we also write $Df(p)$ for the Jacobian matrix with entries $Df(p)_{ij} = \partial_j f_i(p)$.
We denote the gradient of a smooth function $f\colon\mathcal M\to\mathbb R$ defined on a Riemannian manifold $(\mathcal M, g)$ by $\nabla^gf\colon\mathcal M\to T\mathcal M$ and denote the values of the vector field by $\nabla^g f(p)\in T_p\mathcal M$ for $p\in\mathcal M$. When the Riemannian metric is unambiguous we drop the superscript. 

For $A\in\mathbb R^{n\times m}$, we denote its Moore-Penrose inverse by $A^+\in\mathbb R^{m\times n}$. Note that $AA^+$ is the orthogonal (Euclidean) projection onto $\operatorname{range}(A)$ and $A^+A$ is the orthogonal (Euclidean) projection onto $\operatorname{ker}(A)$. 
Finally, for functions $f,g$ we write $f(t) = O(g(t))$ for $t\to t_0$ if there is a constant $c>0$ such that $f(t)\le cg(t)$ for $t\to t_0$, where we allow $t_0=+\infty$.

\section{Markov decision processes}

A \emph{Markov decision process} or shortly \emph{MDP} is a tuple \((\mathcal S, \mathcal A, \alpha, r)\). We assume that \(\mathcal S\) and \(\mathcal A\) are finite sets which we call the \emph{state} and \emph{action space} respectively. 
We fix a Markov kernel \(\alpha\in\Delta_{\mathcal S}^{\mathcal S\times\mathcal A}\) which we call the \emph{transition mechanism}. Further, we consider an \emph{instantaneous reward vector} \(r\in\mathbb R^{\mathcal S\times\mathcal A}\). 
In the case of partially observable MDPs (POMDPs) one also has a fixed kernel \(\beta\in\Delta_{\mathcal O}^{\mathcal S}\) called the \emph{observation mechanism}. The system is \emph{fully observable} if \(\beta=\operatorname{id}\),\footnote{More generally, the system is fully observable if the supports of \(\{\beta(\cdot|s)\}_{s\in\mathcal S}\) are disjoint subsets of $\mathcal{O}$.} in which case the POMDP simplifies to an MDP. 

As \emph{policies} we consider elements \(\pi\in\Delta_{\mathcal A}^{\mathcal S}\). 
More generally, in POMDPs we would consider \emph{effective policies} $\pi = \pi'\circ\beta\in\Delta_{\mathcal A}^{\mathcal S}$ with $\pi'\in\Delta^\OO_\AA$. We will focus on the MDP case, however. 
A policy induces transition kernels \(P_\pi\in\Delta_{\mathcal S\times\mathcal A}^{\mathcal S\times\mathcal A}\) and \(p_\pi\in\Delta_{\mathcal S}^{\mathcal S}\) by
    \begin{equation} 
    P_\pi(s^\prime, a^\prime|s, a) \coloneqq \alpha(s^\prime|s, a) (\pi\circ\beta)(a^\prime|s^\prime) \quad \text{and} \quad p_\pi(s^\prime|s) \coloneqq \sum_{a\in\mathcal A} (\pi\circ \beta)(a|s) \alpha(s^\prime|s, a) . \label{eq:transition-kernel}
    \end{equation}
For any initial state distribution \(\mu\in\Delta_{\mathcal S}\), a policy \(\pi\in\Delta_{\mathcal A}^{\mathcal S}\) defines a Markov process on \(\mathcal S\times \mathcal A\) with transition kernel \(P_\pi\) which we denote by \(\mathbb P^{\pi, \mu}\). 
For a \emph{discount rate} \(\gamma\in(0, 1)\) we define
\begin{align*}
    R(\pi) = R^\mu_\gamma(\pi) & \coloneqq \mathbb E_{\mathbb P^{\pi, \mu}}\bigg[ (1-\gamma) \sum_{t=0}^\infty\gamma^t r(s_t, a_t)\bigg],  
\end{align*}
called the \emph{expected discounted reward}. 
The \emph{expected mean reward} is obtained as the limit with $\gamma\to 1$ when this exists. We will focus on the discounted case, however. 
{The goal is to maximize $R$ over the policy polytope $\Delta_\mathcal A^\mathcal O$.} For a policy \(\pi\) we define the \emph{value function} \(V^\pi =V^\pi_\gamma \in\mathbb R^\mathcal S\) via \(V^\pi(s)\coloneqq R_{\gamma}^{\delta_s}(\pi)\), $s\in\mathcal{S}$, where \(\delta_s\) is the Dirac distribution concentrated at \(s\). 

A short calculation  shows that $R(\pi) = \sum_{s, a} r(s, a) \eta^{\pi}(s, a) = \langle r, \eta^{\pi}\rangle_{\mathcal S\times\mathcal A}$ \citep{zahavy2021reward}, where
\begin{equation}
\eta^{\pi}(s, a) \coloneqq (1-\gamma)\sum_{t=0}^\infty \gamma^t \mathbb P^{\pi, \mu}(s_t = s, a_t = a). 
\end{equation}
The vector $\eta^{\pi}$ is an element of $\Delta_{\mathcal{S}\times \mathcal{A}}$ called the \emph{expected (discounted) state-action frequency} \citep{derman1970finite}, or {(discounted) visitation/occupancy measure}, or on-policy distribution \citep{sutton2018reinforcement}. Denoting the state marginal of \(\eta^{\pi}\) by \(\rho^{\pi}\in\Delta_\mathcal S\) we have $\eta^{\pi}(s, a) = \rho^{\pi}(s) (\pi\circ\beta)(a|s)$. 
We denote the set of all state-action frequencies in the fully and in the partially observable cases by 
    \[ 
    \mathcal N \coloneqq \left\{ \eta^{\pi} : \pi\in\Delta_{\mathcal A}^{\mathcal S} \right\} \subseteq \Delta_{\mathcal S\times\mathcal A}  \quad \text{and}\quad \mathcal N^{\beta} \coloneqq \left\{ \eta^\pi : \pi\in\Delta_{\mathcal A}^{\mathcal O} \right\} \subseteq \Delta_{\mathcal S\times\mathcal A} . 
    \]
Note that the expected cumulative reward function \(R\colon\Delta_\mathcal A^\mathcal O\to\mathbb R\) factorizes according to 
    \[ 
    \Delta_\mathcal A^\mathcal O\to\Delta_\mathcal A^\mathcal S\to\mathcal N^{\mu}\to\mathbb R, \quad \pi\mapsto \pi\circ\beta\mapsto\eta^{\pi}\mapsto \langle r, \eta^{\pi}\rangle_{\mathcal S\times\mathcal A}.
    \]
We recall the following well-known facts. 

\begin{proposition}[State-action polytope of MDPs, \cite{derman1970finite}]
\label{prop: state_action_polytope_fully_observable}
The set $\mathcal N$ of state-action frequencies is a polytope given by $\mathcal N = \Delta_{\SS\times\AA}\cap\mathcal L = \mathbb R_{\ge0}^{\mathcal S\times\mathcal A}\cap\mathcal L$, where
\begin{equation}\label{eq:linearSpace}
     \mathcal L = \left\{ \eta\in\mathbb R^{\SS\times\AA} : \ell_s(\eta) = 0 \text{ for all } s\in\SS, \eta\ge0 \right\},
\end{equation}
and $\ell_s(\eta) \coloneqq \sum_{a} \eta_{sa} - \gamma\sum_{s',a'} \eta_{s'a'}\alpha(s|s', a') - (1-\gamma) \mu_s$. 
\end{proposition}

The state-action polytope for a two-state MDP is shown in Figure~\ref{fig:illustrations}. 
We note that in in the case of partially observable MDPs, the set of state-action frequencies $\mathcal{N}^\beta$ does not form a polytope, but rather a polynomially constrained set involving polynomials of higher degree depending on the properties of the observation kernel~\cite{mueller2022geometry}. 

The result above shows that a (fully observable) Markov decision process can be solved by means of linear programming. 
Indeed, if $\eta^\ast$ is a solution of the linear program $\langle r,\eta\rangle_{\SS\times\AA}$ over $\mathcal{N}$, one can compute the maximizing policy over $\Delta^\SS_\AA$ by conditioning, $\pi^\ast(a|s) = \eta^\ast(s,a)/\sum_{a'}\eta^\ast(s,a)$. 
We propose to study the evolution of natural policy gradient methods in state-action space $\mathcal{N}\subseteq\Delta_{\mathcal{S}\times\mathcal{A}}$. Indeed, we show that the evolution of diverse natural policy gradient algorithms in the state-action polytope solves the gradient flow of a (regularized) linear objective with respect to a Hessian geometry in state-action space. This perspective facilitates relatively short proofs for the global convergence of natural policy gradient methods and can also provide rates. In order to relate Riemannian geometries in the policy space $\Delta_\AA^\SS$ to Riemannian geometries in the state-action polytope $\mathcal N$ we need the following assumption. 

\begin{assumption}[Positivity]\label{ass:positivity}
For every $s\in\mathcal S$ and $\pi\in\Delta_\AA^\OO$, we assume that $\sum_{a} \eta_{sa}^\pi>0$. 
\end{assumption}

Assumption~\ref{ass:positivity} holds in particular if either \(\alpha>0\) and $\gamma>0$ or $\gamma<1$ and \(\mu>0\) entrywise~\cite{mueller2022geometry}. This assumption is standard in linear programming approaches and necessary for the convergence of policy gradient methods in MDPs~\citep{kallenberg1994survey,mei2020global}. 
With this assumption in place we have the following. 

\begin{proposition}[Inverse of state-action map,
\cite{mueller2022geometry}]\label{prop:birational}
Under Assumption~\ref{ass:positivity}, the mapping $\Delta_\AA^\SS\to\mathcal N, \omega\mapsto \eta$ is rational and bijective with rational inverse given by conditioning
$\mathcal N \to \Delta_\AA^\SS, \eta \mapsto \omega$, 
where $\omega_{as} = \frac{\eta_{sa}}{\sum_{a'}\eta_{sa'}}$.
\end{proposition}

This result shows that the (interior of the) set of policies and the (interior of the) state-action polytope are diffeomorphic. Hence, we can port the Riemannian geometry on any of the two sets to the other by using the pull back along $\pi\mapsto \eta$ or the conditioning map $\eta\mapsto \pi$. 

\section{Natural gradients}

In this section we provide some background on the notion of natural gradients. 

\subsection{Definition and general properties of natural gradients}

In many applications, one aims to optimize a model parameter $\theta$ with respect to an objective function $\ell$ that is defined on a model space $\mathcal{M}$, as illustrated in Figure~\ref{fig:parametricModels}. 
\begin{figure} 
    \centering
    \begin{tikzpicture}[node distance=1cm, auto,]
    \node[punkt] (parameter) {Parameter space \(\Theta\subseteq\mathbb R^p\)};
    \node[punkt, inner sep=5pt,right=2cm of parameter] (model) {Model space \(\mathcal M\)}
    edge[pil, <-] node[auto] {\(P\)} (parameter.east);
    \node[punkt, inner sep=5pt,below=2cm of model] (formidler) {Reals \(\mathbb R\)}
    edge[pil, <-] node[auto] {\(\ell\)}  (model.south)
    edge[pil, <-] node[auto] {\(L\)}  (parameter.south);
\end{tikzpicture}
\caption{Schematic drawing of parametric models with an objective function \(\ell\) and resulting parameter objective function \(L\); note that neither the choice of geometry in the model space nor the factorization or the model space itself is uniquely determined by the objective function $L$.}
\label{fig:parametricModels}
\end{figure}
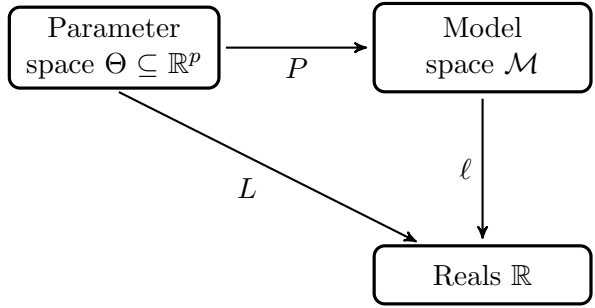
This general setup, with an objective function that factorizes as $L(\theta)=\ell(P(\theta))$, covers several usual parameter estimation and supervised learning cases, and also problems such as the numerical solution of PDEs with neural networks or policy optimization in MDPs and reinforcement learning. 
Naively, the optimization problem can be approached with first order methods, computing the gradients in parameter space with respect to the Euclidean geometry. 
However, this neglects the geometry of the parametrized model $\mathcal M_\Theta = P(\Theta)$, which is often seen as a disadvantage since it may lead to parametrization-dependent plateaus in the optimization landscape. 
At the same time, the biases that particular parametrizations can introduce into the optimization can be favorable in some cases. This is an active topic of investigation particularly in deep learning, where $P$ is often a highly non-linear function of $\theta$. 
At any rate, there is a good motivation to study of the effects of the parametrization and the possible advantages from incorporating the geometry of model space into the optimization procedure in parameter space. 

The \emph{natural gradient} as introduced in~\cite{amari1998natural} is a way to incorporate the geometry of the model space into the optimization procedure and to formulate iterative update directions that are invariant under reparametrizations. 
Although it has been most commonly applied in the context of parameter estimation under the maximum likelihood criterion, the concept of natural gradient has been formulated for general parametric optimization problems and in combination with arbitrary geometries. 
In particular, natural gradients have been applied to neural network training~\cite{park2000adaptive, martens2015optimizing, desjardins2015natural, izadi2020optimization}, policy optimization~\citep{kakade2001natural, peters2003reinforcement, morimura2008new} and inverse problems~\cite{nurbekyan2022efficient}. 
Especially in the latter case, different notions of natural gradients have been introduced. A version that incorporates the geometry of the \emph{sample space} are natural gradients based on an optimal transport geometry in model space~\cite{Li2018natural,Malago2018Wasserstein,Arbel2020Kernelized}. 
We shall discuss natural gradients in a way that emphasizes that even for a specific problem there may not be a unique natural gradient. This is because both the factorization $L(\theta)=\ell(P(\theta))$ of the objective as well as what should be considered a natural geometry in model space may not be unique. 

But what is it that makes a gradient or update direction \emph{natural}? The general consensus is that it should be invariant under reparametrization to prevent artificial plateaus and provide consistent stopping criteria, and it should (approximately) correspond to a gradient update with respect to the geometry in the model space. 
We now give the formal definition of the natural gradient with respect to a given factorization and a geometry in model space that we adopt in this work, which can be shown to satisfy the desired properties. 

\begin{definition}[General natural gradient]
\label{def:ng}
Consider the problem of optimizing an objective $L\colon\Theta\to\mathbb R$, where the \emph{parameter space} $\Theta\subseteq\mathbb R^p$ is an open subset. 
Further, assume that the objective factorizes as $L = \ell\circ P$, where $P\colon \Theta\to\mathcal{M}$ is a \emph{model parametrization} with $\mathcal M$ a Riemannian manifold with Riemannian metric $g$, and $\ell\colon \mathcal{M}\to\mathbb{R}$ is a \emph{loss in model space}, as shown in Figure~\ref{fig:parametricModels}. 
For $\theta\in\Theta$ we define the Gram matrix $G(\theta)_{ij}\coloneqq g_{P(\theta)}(dP_\theta e_i, dP_\theta e_j)$ and call $\nabla^N L(\theta) \coloneqq G(\theta)^+ \nabla L(\theta)$ the \emph{natural gradient (NG) of $L$ at $\theta$ with respect to the factorization $L = \ell\circ P$ and the metric $g$}. 
\end{definition}

\paragraph{Natural gradient as best improvement direction} 
Consider a parametrization $P\colon\Theta\to\mathcal M$ with image $\mathcal M_\Theta = P(\Theta)$, where $\mathcal M$ is a Riemannian manifold with metric $g$. 
Let us fix a parameter \(\theta\in\Theta\) and set $p\coloneqq P(\theta)$. 
Moving in the direction \(v\in T_\theta\Theta\) in parameter space results in moving in the direction \(w = dP_\theta v \in T_{p}\mathcal M\) in model space. The space of all directions that can result in this way is the generalized tangent space \(T_\theta\mathcal M_\Theta\coloneqq\operatorname{range}(d_\theta P) \subseteq T_p\mathcal M\). 
Hence, the best direction one can take on \(\mathcal M_\Theta\) by infinitesimally varying the parameter $\theta$ is given by 
    \[ 
    \underset{w\in T_\theta\mathcal M_\Theta, g_p(w, w) = 1}{\arg\max} \partial_w \ell(p),  
    \]
which is equal (up to normalization) to the projection $\Pi_{T_\theta \mathcal M_\Theta}(\nabla^g \ell(p))$ of the Riemannian gradient \(\nabla^g \ell(p)\) onto \(T_\theta\mathcal M_\Theta\). 
Moving in the direction of the natural gradient in parameter space results in the optimal update direction over the generalized tangent space $T_\theta\mathcal M_\Theta$ in model space. 

\begin{theorem}[Natural gradient leads to steepest descent in model space]\label{thm:NGinModelSpace}
Consider the settings of Definition~\ref{def:ng}, where $\mathcal M$ is a Riemannian manifold with metric $g$. 
Let $\nabla^NL(\theta)\coloneqq G(\theta)^+\nabla_\theta L(\theta)$ denote the natural gradient with respect to this factorization.
Then it holds that 
\[ 
dP_\theta (\nabla^NL(\theta)) = \Pi_{T_\theta \mathcal M_\Theta}( \nabla^g \ell(P(\theta))). 
\]
\end{theorem}
For invertible Gram matrices $G(\theta)$ this result is well known \cite[Subsection 12.1.2]{amari2016information}; for singular Gram matrices we refer to~\cite[Theorem 1]{van2022invariance}. 

\subsection{Choice of a geometry in model space}

\paragraph{Invariance axiomatic geometries} 
A celebrated theorem by Chentsov~\cite{cencov1982} characterizes the Fisher metric of statistical manifolds with finite sample spaces as the unique metric (up to multiplicative constants) that is invariant with respect to congruent embeddings by Markov mappings. 
A generalization of Chentsov's result for arbitrary sample spaces was given by Ay et al.~\cite{ay2017information}. 

Given two Riemannian manifolds $(\mathcal{E},g)$, $(\mathcal{E}',g')$ and an embedding $f\colon \mathcal{E}\to\mathcal{E}'$, the metric is said to be invariant if $f$ is an isometry, meaning that 
$$
g_p(u,v) = (f^\ast g')_p(u,v):= g'_{f(p)}(f_\ast u, f_\ast v) , \quad \text{for all $p\in\mathcal{E}$ and $u,v\in T_p\mathcal{E}$},  
$$
where $f_\ast\colon T_p\mathcal{E}\to T_{f(p)}\mathcal{E}'$ is the pushforward of $f$. 
And a congruent Markov mapping is in simple terms a linear map $p\mapsto  M^T p$, where $M$ is a row stochastic partition matrix, i.e., a matrix of non-negative entries with a single non-zero entry per column and entries of each row adding to one. Such a mapping has the natural interpretation of splitting each elementary event into several possible outcomes with fixed conditional probabilities. 
By Chentsov's theorem, requiring invariance with respect to any such mapping results in a single possible choice for the metric (up to multiplicative constants). We recall that on the interior of the probability simplex $\Delta_\SS$ the Fisher metric is given by 
$$
g_p(u,v) = \sum_{s\in\SS} \frac{u_s v_s}{p_s}, \quad \text{for all $u,v\in T_p\Delta_\SS$.}
$$

Based on this approach, Campbell~\cite{Campbell} characterized the set of invariant metrics on the set of non-normalized positive measures with respect to congruent embeddings by Markov mappings. This results in a family of metrics which restrict to the Fisher metric (up to a multiplicative constant) over the probability simplex. 
Following this line of ideas, Lebanon~\cite{Lebanon2005AxiomaticGO} characterized a class of invariant metrics of positive matrices that restrict to products of Fisher metrics over stochastic matrices.\footnote{For Riemannian manifolds $(\mathcal{M}_1,g_1)$ and $(\mathcal{M}_2,g_2)$, the product metric on $\mathcal{M}_1\times\mathcal{M}_2$ is defined by $g(u_1+u_2, v_1+v_2)=g_1(u_1,v_1)+g_2(u_2,v_2)$.} 
The maps considered by Lebanon do not map stochastic matrices to stochastic matrices, which motivated \cite{montufar2014fisher} to investigate a natural class of mappings between conditional probabilities. 
They showed that requiring invariance with respect to their proposed mappings singles out a family of metrics that correspond to products of Fisher metrics on the interior of the conditional probability polytope, 
$$
g_\pi(u,v) = \sum_{s\in\SS} \frac{1}{|\SS|}\sum_{a\in\AA} \frac{u_{sa} v_{sa}}{\pi_{sa}}, \quad \text{for all $u,v\in T_\pi \Delta^\SS_\AA$} , 
$$
up to multiplicative constants. This work also offered a discussion of metrics on general polytopes and weighted products of Fisher metrics, which correspond to the Fisher metric when the conditional probability polytope is embedded in the joint probability simplex by way of providing a marginal distribution.

\paragraph{Hessian geometries}

Instead of characterizing the geometry of model space $\mathcal M$ via an invariance axiomatic, one can select a metric based on the optimization problem at hand. 
For example, it is well known that the Fisher metric is the local Riemannian metric induced by the Hessian of the KL-divergence in the probability simplex. 
Hence, if the objective function is a KL-divergence, choosing the Fisher metric yields preconditioners that recover the inverse of the Hessian at the optimum, which can yield locally quadratic convergence rates. 
More generally, if the objective $\ell\colon\mathcal M\to\mathbb R$ has a positive definite Hessian at every point, it induces a Riemannian metric via 
$$
g_p(v,w) = v^\top \nabla^2\ell(p) w
$$ 
in local coordinates, which we call the \emph{Hessian geometry} induced by $\ell$ on $\mathcal{M}$; see~\cite{amari2010information, shima2007geometry}. 

\begin{example}[Hessian geometries]
The following Riemannian geometries are induced by strictly convex functions.
\begin{enumerate}
    \item \emph{Euclidean geometry:} The Euclidean geometry on $\mathbb R^d$ is induced by the 
    squared Euclidean norm $x\mapsto \sum_i x_i^2$. 
    \item \emph{Fisher geometry:} The Fisher metric on $\mathbb R_{>0}^{d}$ is induced by the negative entropy $x\mapsto \sum_i x_i \log(x_i)$. 
    \item \emph{Itakura-Saito:} The logarithmic barrier function $x\mapsto \sum_i\log(x_i)$ of the positive cone $\mathbb R_{>0}^d$ yields the Itakura-Saito metric (see the next item). 
    \item \emph{$\sigma$-geometries: } All of the above examples can be interpreted as special cases of a parametric family of Hessian metrics. More precisely, if we let 
    \begin{equation}\label{eq:sigmaFunctions}
    \phi_\sigma(x)\coloneqq \begin{cases}
\sum_i x_i\log(x_i) \quad & \text{if } \sigma = 1 \\
-\sum_i \log(x_i) \quad & \text{if } \sigma = 2 \\
\frac{1}{(2-\sigma)(1-\sigma)}\sum x_i^{2-\sigma} \quad & \text{otherwise}, 
\end{cases}
\end{equation}
then the resulting Riemannian metric on $\mathbb R^{d}$ for $\sigma\in(-\infty, 0]$ and on $\mathbb R_{>0}^d$ for $\sigma\in(0, \infty)$ is given by
\begin{equation}\label{eq:sigmaMetric}
    g^\sigma_x(v, w) = \sum_{i} \frac{v_i w_i}{x_i^\sigma}.
\end{equation}
This recovers the Euclidean geometry for $\sigma=0$, 
the Fisher metric for $\sigma=1$, and the Itakura-Saito metric for $\sigma=2$. 
Note that these geometries are closely related to the so-called $\beta$-divergences~\cite{amari2010information}, which are the Bregman divergences of the functions $\phi_\sigma$ for $\beta = 1-\sigma$. 
We use $\sigma$ instead of $\beta$ in order to avoid confusion with our notation for the observation kernel $\beta$ in a POMDP. 
    
    \item \emph{Conditional entropy: } Given two finite sets $\mathcal X, \mathcal Y$ and a probability distribution $\mu$ in $\Delta_{\mathcal{X}\times\mathcal{Y}}$ we can consider the conditional entropy 
    $\phi_C(\mu) \coloneqq H(\mu|\mu_X) = H(\mu) - H(\mu_X)$ from~\eqref{eq:conditionalEntropy}. 
    This is a convex function on the simplex $\Delta_{\mathcal X\times\mathcal Y}$~\cite{neu2017unified}. 
    The Hessian of the conditional entropy is given by 
    \begin{equation}
       \partial_{(s,a)}\partial_{(s',a')} \phi_C(\mu) = \delta_{xx'}\left(\delta_{yy'} \mu(x,y)^{-1} - \mu_X(x)^{-1} \right),  \label{eq:conditional-entropy}
    \end{equation}
    as can be verified by explicit computation or the chain rule for Hessian matrices (see also proof of Theorem~\ref{thm:kakadeAsPullback}). 
    This Hessian does not induce a Riemannian geometry on the entire simplex since is not positive definite on the tangent space $T\Delta_{\mathcal X\times\mathcal Y}$, as can be seen by considering the specific choice $\mathcal X = \mathcal Y = \{1,2\}, \mu_{ij} = 1/4$ for all $i,j=1, 2$ and the tangent vector $v\in T_\mu \Delta_{\mathcal X\times\mathcal Y}$ given by $v_{ij} = (-1)^{i}$. 
    However, when fixing a marginal distribution $\nu\in\Delta_\mathcal X, \nu>0$, then the conditional entropy $\phi_C$ induces a Riemannian metric on the interior of $ P = \{\mu\in\Delta_{\mathcal X\times\mathcal Y} : \mu_X = \nu \}$. To see this we consider the diffeomorphism given by conditioning $\operatorname{int}(P)\to\operatorname{int}(\Delta_{\mathcal Y}^\mathcal X), \mu \mapsto \mu_{Y|X}$. It can be shown by explicit computation (analogous to the proof of Theorem~\ref{thm:kakadeAsPullback}) that the Hessian $\nabla^2\phi_C(\mu)$ is the metric tensor of the pull back of the Riemmanian metric 
\[ 
g\colon T\Delta_\mathcal Y^\mathcal X\times T\Delta_\mathcal Y^\mathcal X \to\mathbb R, \quad g_{\mu(\cdot|\cdot)}(v,w) \coloneqq \sum_{x} \nu(x) \sum_{y} \frac{v(x,y) w(x,y)}{\mu(y|x)}. 
\]
This argument can be adapted to sets $\{\mu\in\Delta_{\mathcal X\times\mathcal X} : \mu_X = \nu(\mu_{Y|X}) \}$, where $\nu\colon \operatorname{int}(\Delta_\mathcal Y^\mathcal X)\to\operatorname{int}(\Delta_\mathcal X)$ depends smoothly on the conditional $\mu_{Y|X}\in\Delta_\mathcal Y^\mathcal X$. 

We note that the Bregman divergence induced by the conditional entropy is the conditional relative entropy~\cite{neu2017unified}, 
        \begin{align*}
        D_{\phi_C}(\mu^{(1)}, \mu^{(2)}) 
        & =  D_{KL}(\mu^{(1)}, \mu^{(2)}) - D_{KL}(\mu^{(1)}_X, \mu^{(2)}_X) \\ 
        & = \sum_{x} \mu^{(1)}_X(x) D_{KL}(\mu^{(1)}(\cdot|x), \mu^{(2)}(\cdot|x)). 
        \end{align*}
\end{enumerate}
\end{example}

\paragraph{Local Hessian of Bregman divergences}
Let $\phi$ be a twice differentiable strictly convex function and denote its Bregman divergence with $D_\phi(x,y) = \phi(x) - \phi(y) - \langle \nabla \phi(y), x-y\rangle$. 
Then it holds that 
\begin{equation}\label{eq:hessianBregman}
    \nabla_y^2 D_\phi(x,y)|_{y=x} = \nabla^2_y D_\phi(y, x)|_{y=x} = \nabla^2 \phi(x). 
\end{equation}
To see this, we set $f(y) \coloneqq D_\phi(x,y)$. Then it is straight forward to see that $\nabla^2f(y) = \nabla^2\phi(y)$. 
Further, one can compute 
\begin{align*}
    \partial_{y_j} f(y) & = \partial_{y_j} \left(\phi(x) - \phi(y) - \sum_k \partial_{y_k}\phi(y) (x_k-y_k)\right) \\
    & =  - \partial_{y_j}\phi(y) + \sum_k \partial_{y_j}\partial_{y_k}\phi(y) (y_k-x_k) +  \partial_{y_j}\phi(y). 
\end{align*} 
Hence, we obtain 
\begin{align*}
    \partial_{y_i}\partial_{y_j} f(y) = - \partial_{y_i}\partial_{y_j}\phi(y) + \sum_k \partial_{y_i}\partial_{y_j}\partial_{y_k}\phi(y) (y_k-x_k) + \partial_{y_i}\partial_{y_j}\phi(y) + \partial_{y_i}\partial_{y_j}\phi(y), 
\end{align*}
and hence $\nabla^2f(x) = \nabla^2\phi(x)$. 

\paragraph{Connection to Gauss-Newton method}
Let $\phi$ be a twice differentiable strictly convex function.  
Then the Gram matrix of the Hessian geometry is given by 
    \[ 
    G(\theta) = DP(\theta)^\top \nabla^2\phi(P(\theta))DP(\theta). 
    \]
Hence $G^{-1}(\theta)$ can be interpreted as a Gauss-Newton preconditioner of the objective function $\phi\circ P$~\cite{martens2020new}. 
In particular, for the square loss we have $\phi(x) = \|x\|_2^2$, in which case $G(\theta)^{-1} = (DP(\theta)^\top DP(\theta))^{-1}$ is the usual nonlinear least squares Gauss-Newton preconditioner. 

\section{Natural policy gradient methods}

In this section we give a brief overview of different notions of policy gradient methods that have been proposed in the literature and study their associated geometries in state-action space. 
Policy gradient methods~\cite{williams1992simple, konda1999actor, sutton1999policy, Marbach2001Simulation, baxter2000reinforcement} offer a flexible approach to reward optimization. They have been used in robotics~\cite{peters2003reinforcement} and have been combined with deep neural networks~\cite{silver2016mastering, silver2017mastering,shao2019survey}. 
In the context of MDPs there are multiple notions of natural policy gradients. 
For instance, one may choose to use an optimal transport geometry in model space resulting in Wasserstein natural policy gradients~\cite{moskovitz2021efficient}. 
Most important to our discussion, there are different possible choices for the model space. 
One obvious candidate is the policy space $\Delta_\AA^\SS$, which was used by Kakade~\cite{kakade2001natural}. 
However the objective function $R(\pi)$ is a rational non-convex function over this space an thus requires a delicate analysis. 
A second candidate, which was proposed by Morimura et al.~\cite{morimura2008new}, is the state-action space $\mathcal N\subseteq\Delta_{\SS\times\AA}$, for which the objective becomes a rather simple, linear function. 
By Proposition~\ref{prop:birational} the two model spaces $\Delta^\SS_\AA$ and $\mathcal{N}$ are diffeomorphic under mild assumptions, which allows us to study any NPG method defined with respect to the policy space in state-action space. 
Because of the simplicity of the objective function in state-action space, we propose to study the evolution of NPG methods in this space. 
As we will see, this has the added benefit that it allows us to interpret several of the existing NPG methods as being induced by Hessian geometries. Based on this observation we can conduct a relatively simple convergence analysis for these methods. 
Finally, we propose a class of policy gradients closely related to $\beta$-divergences that interpolate between NPG arising from logarithmic barriers, entropic regularization and the Euclidean geometry. 

\subsection{Policy gradients}

Throughout the section, we consider parametric policy models $P\colon\Theta\to\Delta_\mathcal A^\mathcal S$ and write $\pi_\theta = P(\theta)\in\Delta_\mathcal A^\mathcal S$ for the policy arising from the parameter $\theta$. We denote the corresponding state-action and state frequencies by $\eta_\theta$ and $\rho_\theta$. Finally, in slight abuse of notation we write $R(\theta)$ for the expected infinite-horizon discounted reward obtained by the policy $\pi_\theta$. 
The \emph{vanilla policy gradient (vanilla PG)} method is given by the iteration
\begin{equation}
    \theta_{k+1}\coloneqq \theta_k + \Delta t \nabla R(\theta_k),
\end{equation}
where $\Delta t>0$ is the step size. 

For $\gamma\in(0,1)$, the reward function $\pi\mapsto R(\pi)$ is a rational function. 
Hence, in principle it can be differentiated using any automatic differentiation method. One can use the celebrated policy gradient theorem and use matrix inversion to compute the parameter update. 

\begin{theorem}[Policy gradient theorem]
Consider an MDP $(\SS, \AA, \alpha, r), \gamma\in[0,1)$ and a parametrized policy class. It holds that
\begin{align*}
    \partial_{\theta_i} R(\theta) = \sum_{s} \rho_\theta(s) \sum_{a}\partial_{\theta_i} \pi_\theta(a|s) Q^{\pi_\theta}(s, a) = \sum_{s, a}\eta_\theta(s, a) \partial_{\theta_i} \log(\pi_\theta(a|s)) Q^{\pi_\theta}(s, a),
\end{align*}
where $Q^\pi\coloneqq(I-\gamma P_\pi)^{-1}r\in\mathbb R^{\mathcal S\times\mathcal A}$ is the state-action value function.
\end{theorem}

In a reinforcement learning setup, one does not have direct access to the transition $\alpha$ and hence to $P_\pi$ \eqref{eq:transition-kernel} nor $Q^\pi$, and sometimes even $\SS$ is not known a priori. 
In this case, one has to estimate the gradient from interactions with the environment~\cite{baxter2000reinforcement, baxter2001infinite, morimura2010derivatives, sutton2018reinforcement}.
In this work, however, we study the planning problem in MDPs, i.e., we assume that we have access to exact gradient evaluations. 

\paragraph{Policy parametrizations}
Many results on the convergence of policy gradient methods have been providede for \emph{tabular softmax policies}. The tabular softmax parametrization is given by 
\begin{equation}  
    \pi_\theta(a|s) \coloneqq \frac{e^{\theta_{sa}}}{\sum_{a'}e^{\theta_{sa'}}} \quad \text{for all } a\in\mathcal A, s\in\SS , 
    \label{eq:soft-max}
\end{equation}
for $\theta\in\mathbb R^{\SS\times\AA}$.
One benefit of tabular softmax policies is that they parametrize the interior of the policy polytope $\Delta_\AA^\SS$ in a regular way, i.e., such that the Jacobian has full rank everywhere, and the parameter is unconstrained in an affine space. 

\begin{definition}[Regular policy parametrization]
We call a policy parametrization $\mathbb R^p\to\operatorname{int}(\Delta_{\mathcal A}^\mathcal S), \theta \mapsto \pi_\theta$ \emph{regular} if it is differentiable and satisfies 
\begin{equation}\label{eq:regularparametrization}
    \operatorname{span}\{\partial_{\theta_{i}}\pi_\theta : i = 1, \dots, p\} = T_{\pi_\theta} \Delta_\mathcal A^\mathcal S = \prod_{s\in\mathcal S} T_{\pi_\theta(\cdot|s)}\Delta_\mathcal A \quad \text{for every } \theta\in \mathbb{R}^p.
\end{equation}
\end{definition}
We will focus on regular policy parametrizations. Nonetheless, we observe that policy optimization with constrained search variables can also be an attractive option and refer to \cite{muller2022solving} for a discussion in context of POMDPs.

\paragraph{Regularization in MDPs} 
In practice, the reward function is often regularized as $R_\lambda = R - \lambda\psi$. This is often motivated to encourage exploration~\cite{williams1992simple} and has also been shown to lead to fast convergence for strictly convex regularizers $\psi$~\cite{mei2020global, cen2021fast}. 
One popular regularizer is the conditional entropy in state-action space, see~\cite{neu2017unified, mei2020global, cen2021fast}, 
\begin{equation}\label{eq:conditionalEntropySA}
    \psi_C(\theta) = \sum_{s} \rho_\theta(s) \sum_{a}\pi_\theta(a|s) \log(\pi_\theta(a|s)) = H(\eta_\theta) - H(\rho_\theta),
\end{equation}
which has also been used to successfully design trust region and proximal methods for reward optimization~\cite{schulman2015trust, schulman2017proximal}. 
It is also possible to take the functions $\phi_\sigma$ defined in~\eqref{eq:sigmaFunctions} as regularizers. 
This includes the entropy function, which is studied in state-action space in~\cite{neu2017unified} and logarithmic barriers, which are studied in policy space in~\cite{agarwal2021theory}. 
Introducing a regularizer changes the optimization problem and usually also the optimizer. 
The difference introduced by this regularization can be estimated in terms of the regularization strength $\lambda$. 
For logarithmic barriers in state-action space, this follows from standard estimates for interior point methods~\cite{boyd2004convex}. For entropic regularization in state-action space, this is elaborated in~\cite{weed2018explicit}, and for 
the conditional entropy this is done in~\cite{mei2020global, cen2021fast}. We will see later that several of these regularizers lead to Hessian geometries in state-action space that correspond to different natural gradients that have been proposed in the context of policy optimization. 

\paragraph{Partially observable systems}
Although we will only consider parametric policies in fully observable MDPs, our discussion covers the case of POMDPs in the following way. 
Any parametric family of observation-based policies $\{\pi_\theta:\theta\in\Theta\} \subseteq \Delta_\AA^\OO$ induces a parametric family of state-based policies $\{\pi_\theta\circ\beta : \theta\in\Theta\}$. 
Hence, the policy gradient theorem as well as the definitions of natural policy gradients directly extend to the case of partially observable systems. 
However, the global convergence guarantees in Section~\ref{sec:convergenceFlows} and Section~\ref{sec:convergenceDiscrete} do not carry over to POMDPs since they assume tabular softmax (state) policies. 

\paragraph{Projected policy gradients}
An alternative to using parametrizations with the property that any unconstrained choice of the parameter leads to a policy, 
is to use constrained parametrizations and projected gradient methods. 
For instance, one can parametrize policies in $\Delta^\SS_\AA$ by their constrained entries and use the iteration 
    \[ 
    \pi_{k+1} \coloneqq \Pi_{\Delta_{\AA}^\SS}(\pi_k + \Delta t G(\pi_k)^+ \nabla R(\pi)), 
    \]
where $\Pi_{\Delta_\AA^\SS}$ is the (Euclidean) projection to $\Delta_{\AA}^\SS$. 
We will not study projected policy gradient methods and refer to~\cite{agarwal2021theory, xiao2022convergence} for convergence rates of these methods. 

\subsection{Kakade's natural policy gradient}
\label{sec:kakades-metric} 

Kakade~\cite{kakade2001natural} proposed natural policy gradient based on a Riemannian geometry in the policy polytope $\Delta_\mathcal A^\mathcal S$. We will see that Kakade's NPG can be interpreted as the NPG induced by the Hessian geometry in state-action space arising from conditional entropy regularization of the linear program associated to MDPs. 
Kakade's idea was to mix the Fisher information matrices of the policy over the individual states according to the state frequencies, i.e., to use the following Gram matrix:  
\begin{align}\label{eq:defKakadeGram}
    \begin{split}
         G_K(\theta)_{ij} & = \sum_s \rho_\theta(s)\sum_{a} \pi_\theta(a|s) \partial_{\theta_i} \log(\pi_\theta(a|s)) \partial_{\theta_j}\log(\pi_\theta(a|s))  \\
         & = \sum_{s, a} \eta_\theta(s, a) \partial_{\theta_i}\log(\pi_\theta(a|s)) \partial_{\theta_j}\log(\pi_\theta(a|s)) \\
         & = \sum_s \rho_\theta(s)\sum_{a} \frac{\partial_{\theta_i} \pi_\theta(a|s) \partial_{\theta_j}\pi_\theta(a|s)}{\pi_\theta(a|s)}.
    \end{split}
\end{align}

\begin{definition}[Kakade's NPG and geometry in policy space]
We refer to the natural gradient $\nabla^K R(\theta) \coloneqq G_K(\theta)^+ \nabla_\theta R(\pi_\theta)$ as \emph{Kakade's natural policy gradient (K-NPG)}, where $G_K$ is defined in~\eqref{eq:defKakadeGram}. Hence, Kakade's NPG is the NPG induced by the factorization $\theta\mapsto \pi_\theta\mapsto R(\theta)$ and the Riemannian metric on $\operatorname{int}(\Delta_\mathcal A^\mathcal S)$ given by
\begin{align}\label{eq:kakadeGeometry}
    g^K_\pi(v, w) \coloneqq \sum_{s} \rho^\pi(s)\sum_{a} \frac{v(s, a) w(s, a)}{\pi(a|s)} \quad \text{for all } v, w\in T_\pi\Delta_\mathcal A^\mathcal S.
\end{align}
\end{definition}
Due to its popularity, this method is often referred to simply as \emph{the} natural policy gradient. We will call it Kakade's NPG in order to distinguish it from other NPGs. 

\begin{remark}
In \cite{kakade2001natural} the definition of $G_K$ was heuristically motivated by the fact that the reward is also a mix of instantaneous rewards according to the state frequencies, $R(\pi)=\sum_{s} \rho^\pi(s) \sum_a \pi(a|s) r(s,a)$. 
The invariance axiomatic approaches discussed in \cite{Lebanon2005AxiomaticGO,montufar2014fisher} also yield mixtures of Fisher metrics over individual states, which however do not fully recover Kakade's metric, since this would require a way to account for the particular process that gives rise to the stationary state distribution $\rho^\pi$. 
The works \cite{peters2003reinforcement,bagnell2003covariant,nagaoka2017exponential} argued that the Gram matrix $G_K$ corresponds to the limit of the Fisher information matrices of finite-path probability measures as the path length tends to infinity. 
\end{remark} 

\paragraph{Interpration as Hessian geometry of conditional entropy regularization}\label{subsubsec:KakadeGeomtry}
The metric $g^K$ on the conditional probability polytope $\Delta_{\mathcal A}^\mathcal S$ has been studied in terms of its invariances and its connection to the Fisher metric on finite-horizon path space~\cite{bagnell2003covariant, peters2003reinforcement, montufar2014fisher}. 
We offer a different interpretation of Kakade's geometry by studying its counterpart in state-action space, which we show to be the Hessian geometry induced by the conditional entropy. 

\begin{theorem}[Kakade's geometry as conditional entropy Hessian geometry]\label{thm:kakadeAsPullback}
Consider an MDP $(\mathcal S, \mathcal A, \alpha)$ and fix $\mu\in\Delta_\mathcal S$ and $\gamma\in(0,1)$ such that Assumption~\ref{ass:positivity} holds. Then, Kakade's geometry on $\Delta_\AA^\SS$ is the pull back of the Hessian geometry induced by the conditional entropy on the state-action polytope $\mathcal N\subseteq\Delta_{\mathcal S\times\mathcal A}$ along $\pi\mapsto \eta^\pi$. In particular, Kakade's natural policy gradient is the natural policy gradient induced by the factorization $\theta\mapsto\eta_\theta\mapsto R(\theta)$ with respect to the conditional entropy Hessian geometry, i.e., 
\begin{align}\label{eq:comparisonKakadeMorimura}
\begin{split}
    G_K(\theta)_{ij}  &= \sum_{s,a}\frac{\partial_{\theta_{i}} \eta_\theta(s,a)\partial_{\theta_{j}} \eta_\theta(s,a)}{\eta_\theta(s,a)} - \sum_{s}\frac{\partial_{\theta_{i}} \rho_\theta(s)\partial_{\theta_{j}} \rho_\theta(s)}{\rho_\theta(s)} \\
    &= \sum_{s,a}\partial_{\theta_{i}} \log(\eta_\theta(s,a))\partial_{\theta_{j}}\log( \eta_\theta(s,a))\eta_\theta(s,a) \\
     &- \sum_{s}\partial_{\theta_{i}} \log(\rho_\theta(s))\partial_{\theta_{j}} \log(\rho_\theta(s))\rho_\theta(s).
    \end{split}
\end{align}
\end{theorem}
\begin{proof}
We can pull back the Riemannian metric on the policy polytope proposed by Kakade along the conditioning map 
to define a corresponding geometry in state-action space. 
The metric tensor in state-action space is given by 
\begin{align}
    \begin{split}
        G(\eta)_{(s, a), (s', a')} & = g^K_\pi(\partial_{(s,a)} \eta(\cdot|\cdot), \partial_{(s',a')} \eta(\cdot|\cdot)) \\ & =  \sum_{\tilde s, \tilde a} \rho(\tilde s) \frac{\partial_{(s,a)} \eta(\tilde a|\tilde s)\partial_{(s',a')} \eta(\tilde a|\tilde s)}{\eta(\tilde a|\tilde s)} \\
        & = \sum_{\tilde s, \tilde a} \rho(\tilde s)^2 \frac{\partial_{(s,a)} \eta(\tilde a|\tilde s)\partial_{(s',a')} \eta(\tilde a|\tilde s)}{\eta(\tilde s, \tilde a)}.
    \end{split}
\end{align}
Using $\partial_{(s,a)} \eta(\tilde a|\tilde s) = \partial_{(s,a)}( \eta(\tilde s, \tilde a) \rho(\tilde s)^{-1}) = \delta_{s\tilde s}(\delta_{a\tilde a} \rho(\tilde s)^{-1}- \eta(\tilde s, \tilde a) \rho(\tilde s)^{-2})$ we obtain 
\begin{align}\label{eq:metricTensorKakade}
    \begin{split}
        G(\eta)_{(s, a), (s', a')} & = 
        \delta_{ss'}\left(\delta_{aa'} \eta(s, a)^{-1} - \rho(s)^{-1} \right).
    \end{split}
\end{align}
We aim to show that $G(\eta) = \nabla^2\phi_C(\eta)$, where $\phi_C(\eta) = H(\eta) - H(\rho)$, where $\rho(s) = \sum_{a}\eta(s,a)$ denotes the state-marginal. Note that $\nabla^2 H(\eta) = \operatorname{diag}(\eta)$, which is the first term appearing in~\eqref{eq:metricTensorKakade}. 
For linear maps $g_A(x) = Ax$ the chain rule yields the expression  
\[ 
\partial_i\partial_j (f\circ g_A)(x) = \sum_{k,l} A_{ki} \partial_k\partial_l f(g_A(x)) A_{lj}. 
\]
Noting that $\rho$ is a linear function of $\eta$ we obtain
\[ \partial_{(s,a)}\partial_{(s', a')} H(\rho) = \sum_{\tilde s, \hat s} \delta_{\tilde s, s} \partial_{\tilde s}\partial_{\hat s} H(\rho) \delta_{\hat s, s'} = \delta_{ss'} \rho(s)^{-1}, \]
which is the second term in~\eqref{eq:metricTensorKakade}. Overall this implies $G(\eta) = \nabla^2\phi_C(\eta)$. 
\end{proof}

The Bregman divergence of the conditional entropy is the conditional relative entropy and has been studied as a regularizer for the linear program associated to MDPs in~\cite{neu2017unified}.

\begin{remark}\label{rem:kakadePolicyspace}
Kakade's NPG is known to converge at a locally quadratic rate under conditional entropy regularization~\cite{cen2021fast}, a regularizer which in policy space takes the form
    \[ 
    \psi(\pi) = \sum_{s} \rho^\pi(s) \sum_a \pi(a|s) \log(\pi(a|s)) = \sum_{s} \rho^\pi(s) H(\pi(\cdot|s)). 
    \]
Note however, by direct calculation, that Kakade's geometry in policy space $g^K$ defined in~\eqref{eq:kakadeGeometry} is not the Hessian geometry induced by $\psi$ in policy space, which would take the form
\begin{align*}
    \nabla^2\psi(\pi) 
    &= \sum_{s} \rho^\pi(s) \nabla^2 H(\pi(\cdot|s)) 
    + \sum_{s} (\nabla H(\cdot | s)^\top\nabla\rho^\pi(s) + \nabla H(\cdot | s)\nabla\rho^\pi(s)^\top) \\
     &
    + \sum_{s} H(\pi(\cdot|s)) \nabla^2\rho^\pi(s).
\end{align*}
Instead, the metric proposed by Kakade only considers the contribution of the first term, see~\eqref{eq:kakadeGeometry}. 
As we will see in Sections~\ref{sec:convergenceFlows} and~\ref{sec:convergenceDiscrete}, the interpretation of Kakade's NPG as a Hessian natural gradient induced by the conditional entropic regularization in state-action space allows for a great simplification of its convergence analysis. 
\end{remark}

\subsection{Morimura's natural policy gradient}\label{subsec:morimura}

In contrast to Kakade's approach, who proposed a mixture of Fisher metrics to obtain a metric on the conditional probability polytope $\Delta_\AA^\SS$, Morimura and co-authors~\cite{morimura2008new} proposed to work with the Fisher metric in state-action space $\Delta_{\SS\times\AA}$ to define a natural gradient for reward optimization. 
The resulting Gram matrix is given by the Fisher information matrix induced by the state-action distributions, that is $P(\theta)=\eta_\theta$ and 
\begin{equation}\label{eq:definitionMorimuraMatrix}
    G_M(\theta)_{ij} = \sum_{s, a} \partial_{\theta_i}\log(\eta_\theta(s, a)) \partial_{\theta_j}\log(\eta_\theta(s, a)) \eta_\theta(s, a).
\end{equation}

\begin{definition}[Morimura's NPG]
We refer to the natural gradient $\nabla^M R(\theta) \coloneqq G_M(\theta)^+ \nabla_\theta R(\pi_\theta)$ as \emph{Morimura's natural policy gradient (M-NPG)}, where $G_M$ is defined in~\eqref{eq:definitionMorimuraMatrix}. Hence, Morimura's NPG is the NPG induced by the factorization $\theta\mapsto \eta_\theta\mapsto R(\theta)$ and the Fisher metric on $\operatorname{int}(\Delta_{\mathcal S\times\mathcal A})$. 
\end{definition}
By~\eqref{eq:comparisonKakadeMorimura} the Gram matrix proposed by Morimura and co-authors and the Gram matrix proposed by Kakade are related to each other by 
    \[ 
    G_K(\theta) = G_{M}(\theta) - F_\rho(\theta), 
    \]
where $F_{\rho}(\theta)_{ij} = \sum_{s}\rho_\theta(s) \partial_{\theta_{i}} \log(\rho_\theta(s))\partial_{\theta_{j}} \log(\rho_\theta(s))$ denotes the Fisher information matrix of the state distributions. This relation is reminiscent of the chain rule for the conditional entropy and can be verified by direct computation; see~\cite{morimura2008new}. 
Where we have seen that Kakade's geometry in state-action space is the Hessian geometry of conditional entropy, the Fisher metric is known to be the Hessian metric of the entropy function~\cite{amari2010information}. 
Hence, we can interpret the Fisher metric as the Hessian geometry of the entropy regularized reward $\eta \mapsto \langle r, \eta\rangle - H(\eta)$. 

\subsection{General Hessian natural policy gradient} 
Generalizing the above definitions, we define general state-action space Hessian NPGs as follows.
Consider a twice differentiable function $\phi\colon\mathbb R_{>0}^{\SS\times\AA}\to\mathbb R$ such that  $\nabla^2\phi(\eta)$ is positive definite on $T_\eta\mathcal N = T\mathcal L\subseteq\mathbb R^{\mathcal S\times\mathcal A}$ for every $\eta\in \operatorname{int}(\mathcal N)$.
Then we set
\[ G_\phi(\theta)_{ij} \coloneqq \sum_{s, s',a, a'} \partial_{\theta_i} \eta_\theta(s,a) \partial_{(s,a)} \partial_{(s',a')}\phi(\eta_\theta) \partial_{\theta_j} \eta_\theta(s',a'), \]
which is the Gram matrix with respect to the Hessian geometry in $\mathbb R_{>0}^{\SS\times\AA}$. 

\begin{definition}[Hessian NPG]
We refer to the natural gradient $\nabla^\phi R(\theta) \coloneqq G_\phi(\theta)^+ \nabla_\theta R(\pi_\theta)$ as \emph{Hessian natural policy gradient with respect to $\phi$} or shortly \emph{$\phi$-natural policy gradient ($\phi$-NPG)}. 
\end{definition}

Leveraging results on gradient flows in Hessian geometries we will later provide global convergence guarantees including convergence rates for a large class of Hessian NPG flows covering Kakade's and Morimura's natural gradients as special cases. 
Further, we consider the family $\phi_\sigma$ of strictly convex functions defined in~\eqref{eq:sigmaFunctions}. 
With $G_\sigma(\theta)$ we denote the Gram matrix associated with the Riemannian metric $g^\sigma$, i.e.,
    \[ G_\sigma(\theta)_{ij} = \sum_{s,a} \frac{\partial_{\theta_i} \eta_\theta(s,a)\partial_{\theta_j} \eta_{\theta}(s,a)}{\eta_\theta(s,a)^\sigma}. \]

\begin{definition}[$\sigma$-NPG]
We refer to the natural gradient $\nabla^\sigma R(\theta) \coloneqq G_\sigma(\theta)^+ \nabla_\theta R(\pi_\theta)$ as the \emph{$\sigma$-natural policy gradient ($\sigma$-NPG)}. Hence, the $\sigma$-NPG is the NPG induced by the factorization $\theta\mapsto \eta_\theta\mapsto R(\theta)$ and the metric $g^\sigma$ on $\operatorname{int}(\Delta_{\mathcal S\times\mathcal A})$ defined in~\eqref{eq:sigmaMetric}. 
\end{definition}

For $\sigma=1$ we recover the Fisher geometry and hence Morimura's NPG; for $\sigma=2$ we obtain the Itakura-Saito metric; and for $\sigma=0$ we recover the Euclidean geometry. 
Later, we show that the Hessian gradient flows exist globally for $\sigma\in[1, \infty)$ and provide convergence rates depending on $\sigma$. 

\section{Convergence of natural policy gradient flows}\label{sec:convergenceFlows}

In this section we study the convergence properties of natural policy gradient flows arising from Hessian geometries in state-action space for fully observable systems and tabular softmax policies. Although we focus on this case, we observe that our results directly extend to regular parametrizations of the interior of the policy polytope $\Delta_\mathcal A^\mathcal S$. 
Leveraging tools from the theory of gradient flows in Hessian geometries established in~\cite{alvarez2004hessian} we show $O(t^{-1})$ convergence of the objective value for a large class of Hessian geometries and unregularized reward. 
We strengthen this general result and establish linear convergence for Kakade's and Morimura's NPG flows and $O(t^{-1/(\sigma-1)})$ convergence for $\sigma$-NPG flows for $\sigma\in(1,2)$. 
We provide empirical evidence that these rates are tight and that the rate $O(t^{-1/(\sigma-1)})$ also holds for $\sigma\ge2$. 
Under strongly convex penalization, we obtain linear convergence for a large class of Hessian geometries. 

\paragraph{Reduction to state-action space}
For a solution $\theta(t)$ of the natural policy gradient flow, the corresponding state-action frequencies $\eta(t)$ solve the gradient flow with respect to the Riemannian metric. 
This is made precise in the following result, which shows that it suffices to study Riemannian gradient flows in state-action space in order to study natural policy gradient flows for tabular softmax policies. 

\begin{proposition}[Evolution in state-action space]\label{prop:evolutionStateAction}
Consider an MDP $(\mathcal S, \mathcal A, \alpha)$, a Riemannian metric $g$ on $\operatorname{int}(\mathcal{N}) = \mathbb R_{>0}^{\mathcal S\times\mathcal A}$ and an differentiable objective function $\mathfrak R\colon\operatorname{int}(\Delta_{\mathcal S\times\mathcal A}) \to \mathbb R$. Consider a regular policy parametrization and the objective $R(\theta)\coloneqq \mathfrak R(\eta_\theta)$ and a solution $\theta\colon[0,T]\to \Theta = \mathbb R^{\mathcal S\times\mathcal A}$ of the natural policy gradient flow
\begin{equation}\label{eq:NPGflow}
    \partial_t \theta(t) = \nabla^N R(\theta(t)) = G(\theta(t))^+\nabla R(\theta(t)),
\end{equation}
where $G(\theta)_{ij} = g_p(\partial_{\theta_i} \eta_\theta, \partial_{\theta_j} \eta_\theta)$ and $G(\theta)^+$ denotes some pseudo inverse of $G(\theta)$. 
Then, setting $\eta(t)\coloneqq \eta_{\theta(t)}$ we have that $\eta\colon[0, T] \to \Delta_{\mathcal S\times\mathcal A}$ is the gradient flow with respect to the metric $g|_{\mathcal N}$ and the objective $\mathfrak R$, i.e., solves
\begin{equation}\label{eq:NPGflowSA}
    \partial_t \eta(t) = \nabla^{g|\mathcal N}\mathfrak R(\eta(t)) = \Pi_{T\mathcal L}(\nabla^g \mathfrak R(\eta(t))),
\end{equation}
where $\Pi_{T\mathcal L}^g$ is the $g$-orthogonal projection onto the tangent space $T\mathcal L$ with $\mathcal L$ defined in~\eqref{eq:linearSpace}. 
\end{proposition}
\begin{proof}
This is a direct consequence of Theorem~\ref{thm:NGinModelSpace}.
\end{proof}

The preceding result covers the commonly studied tabular softmax parametrization. 
For general parametrizations, the result does not hold. However, if for any two parameters $\theta, \theta'$ with $\eta_\theta = \eta_{\theta'}$ it holds that 
        \[ \operatorname{span}\{\partial_{\theta_{i}}\pi_\theta : i = 1, \dots, p\} = \operatorname{span}\{\partial_{\theta_{i}}\pi_{\theta'} : i = 1, \dots, p\}, 
        \]
then a similar result can be established. 

By Proposition~\ref{prop:evolutionStateAction} it suffices to study solutions $\eta\colon[0, T]\to\mathcal N$ of the gradient flow in state-action space. 
We have seen before that a large class of natural policy gradients arise from Hessian geometries in state-action space. In particular, this covers the natural policy gradients proposed by Kakade~\cite{kakade2001natural} and Morimura et al.~\cite{morimura2008new}. We study the evolution of these flows in state-action space and leverage results on Hessian gradient flows of convex problems in~\cite{alvarez2004hessian} to obtain global convergence rates for different NPG methods. 

\subsection{Convergence of unregularized Hessian natural policy gradient flows}

First, we study the case of unregularized reward, i.e., where the state-action objective is linear and given by $\mathfrak R(\eta) = \langle r, \eta\rangle$. In this case we obtain global convergence guarantees including rates. In particular, our general result covers the $\sigma$-NPGs and thus Morimura's NPGs as well as Kakade's NPGs. For the remainder of this section we work under the following assumptions.

\begin{setting}\label{set:MDPconvergence}
Let $(\mathcal S, \mathcal A, \alpha)$ be an MDP, $\mu\in\Delta_\mathcal S$ and $r\in\mathbb R^{\mathcal S\times\mathcal A}$ and let the positivity Assumption~\ref{ass:positivity} hold. 
We denote the state-action polytope by $\mathcal N = \mathbb R_{\ge0}^{\SS\times\AA} \cap\mathcal L$, see Proposition~\ref{prop: state_action_polytope_fully_observable}, and its (relative) interior and boundary by $\operatorname{int}(\mathcal N) = \mathbb R_{>0}^{\SS\times\AA} \cap\mathcal L$ and $\partial\mathcal N = \partial\mathbb R_{\ge0}^{\SS\times\AA}\cap\mathcal L$ respectively.
We consider an objective function $\mathfrak R\colon\mathbb R^{\SS\times\AA}\to\mathbb R\cup\{-\infty\}$ that is finite, differentiable and concave on $\mathbb R_{>0}^{\mathcal S\times\mathcal A}$ and continuous on its domain $\operatorname{dom}(\mathfrak R) =  \{\eta\in\mathbb R^{\mathcal S\times\mathcal A}:\mathfrak R(\eta)\in\mathbb R\}$. Further, we consider a real-valued function $\phi\colon\mathbb R^{\SS\times\AA}\to\mathbb R\cup\{+\infty\}$, which we assume to be finite and twice continuously differentiable on $\mathbb R_{>0}^{\mathcal S\times\mathcal A}$ and such that $\nabla^2\phi(\eta)$ is positive definite on $T_\eta\mathcal N = T\mathcal L\subseteq\mathbb R^{\mathcal S\times\mathcal A}$ for every $\eta\in \operatorname{int}(\mathcal N)$. 
Further, with $\eta\colon[0,T)\to\mathcal N$ we denote a solution of the Hessian gradient flow 
\begin{equation}\label{eq:hessianGFst}
    \partial_t \eta(t) = \Pi_{T\mathcal L}(\nabla^2\phi(\eta(t))^{-1} \nabla\mathfrak R(\eta(t))),
\end{equation}
which is the gradient flow with respect to the Hessian geometry induced by $\phi$ on $\mathcal N$. 
We denote\footnote{Note that $\mathfrak R$ is bounded over the bounded set $\mathcal N$ as a concave function.} $R^\ast \coloneqq \sup_{\eta\in\mathcal N} \mathfrak R(\eta)<\infty$
and by $\eta^\ast\in\mathcal N$, we denote a maximizer -- if one exists -- of $\mathfrak R$ over $\mathcal N$. 
We denote the policies corresponding to $\eta_0$ and $\eta^\ast$ by $\pi_0$ and $\pi^\ast$, see Proposition~\ref{prop:birational}.
\end{setting}
We observe that the Hessian of the conditional entropy only defines a Riemannian metric on $\operatorname{int}(\mathcal N)$, even if not over all of $\Delta_{\SS\times\AA}$. 
Note that in general $\eta^\ast$ might lie on the boundary and for linear $\mathfrak R$ corresponding to unregularized reward it necessarily lies on the boundary. 

\paragraph{Sublinear rates for general case} 

We begin by providing a sublinear rate of convergence for general NPG flows, which we the specialize to Kakade and $\sigma$-NPGs.

\begin{lemma}[Convergence of Hessian natural policy gradient flows]\label{prop:convergenceGeneral} 
Consider Setting~\ref{set:MDPconvergence} and assume that there exists a solution $\eta\colon[0,T)\to\operatorname{int}(\mathcal N)$ of the NPG flow~\eqref{eq:hessianGFst} with initial condition $\eta(0) = \eta_0$. Then for any $\eta'\in\mathcal N$ and $t\in[0, T)$ it holds that 
\begin{equation}
    \mathfrak R(\eta') - \mathfrak R(\eta(t)) \le D_\phi(\eta', \eta_0)t^{-1},
\end{equation}
where $D_\phi$ denotes the Bregman divergence of $\phi$. In particular it holds that $\mathfrak R(\eta(t))\to R^\ast$ as $T\to\infty$. Further, this convergence happens at a rate $O(t^{-1})$ if there is a maximizer $\eta^\ast\in\mathcal N$ of $\mathfrak R$ with $\phi(\eta^\ast)<\infty$. 
\end{lemma}
\begin{proof}
This is precisely the statement of Proposition 4.4 in~\cite{alvarez2004hessian}; note however, that they assume a globally defined objective $\mathfrak R\colon\mathbb R^{\SS\times\AA}\to\mathbb R$ and hence for completeness we provide a quick argument. Note that 
\begin{equation}\label{eq:derivativeu}
    \partial_t D_\phi(\eta, \eta(t)) = \langle\nabla \mathfrak R(\eta(t)), \eta - \eta(t) \rangle \le \mathfrak R(\eta(t)) - \mathfrak R(\eta), 
\end{equation}
which can either be seen by inspecting the proof of equation (4.4) in~\cite{alvarez2004hessian} and noting that the proof does not require the stronger assumption made there or by explicit computation. 
Integration and the monotonicity of $t\mapsto \mathfrak R(\eta(t))$ yields the claim. 
\end{proof} 

The previous result is very general and reduces the problem of showing convergence of the natural gradient flow to the problem of well posedness. However, well posedness is not always given, such as for example in the case of an unregularized reward and the Euclidean geometry in state-action space. In this case, the gradient flow in state-action space will reach the boundary of the state-action polytope $\mathcal N$ in finite time at which point the gradient is not classically defined anymore and the softmax parameters blow up; see Figure~\ref{fig:illustrations}. An important class of Hessian geometries that prevent a finite hitting time of the boundary are induced by the class of Legendre-type functions, which curve up towards the boundary. 

\begin{definition}[Legendre type functions]\label{def:LegendreType}
We call $\phi\colon\mathbb R^{\mathcal S\times\mathcal A}\to\mathbb R\cup\{+\infty\}$ a \emph{Legendre type function} if it satisfies the following properties:
\begin{enumerate}
    \item \emph{Domain: } It holds that  $\mathbb R_{>0}^{\mathcal S\times\mathcal A}\subseteq \operatorname{dom}(\phi) \subseteq \mathbb R_{\ge0}^{\mathcal S\times\mathcal A} $, where $\operatorname{dom}(\phi) = \{\eta\in\mathbb R^{\SS\times\AA}:\phi(\eta)<\infty\}$. 
    \item \emph{Smoothness and convexity:} We assume $\phi$ to be continuous on $\operatorname{dom}(\phi)$ and twice continuous differentiable on $\mathbb R_{>0}^{\mathcal S\times\mathcal A}$ and such that $\nabla^2\phi(\eta)$ is positiv definite on $T_\eta\mathcal N = T\mathcal L \subseteq\mathbb R^{\mathcal S\times\mathcal A}$ for every $\eta\in \operatorname{int}(\mathcal N)$.
    \item \emph{Gradient blowup at boundary:}
    For any $(\eta_k)\subseteq\operatorname{int}(\mathcal N)$ with $\eta_k\to \eta\in\partial\mathcal N$  we have $\lVert\nabla \phi(\eta_k)\rVert \to \infty$. 
\end{enumerate}
\end{definition}

We note that the above definition differs from~\cite{alvarez2004hessian}, who consider Legendre functions on arbitrary open sets but work with more restrictive assumptions. More precisely, they require the gradient blowup on the boundary of the entire cone $\mathbb R_{\ge0}^{\SS\times\AA}$ and not only on the boundary of the feasible set $\mathcal N$ of the optimization problem. However, this relaxation is required to cover the case of the conditional entropy, which corresponds to Kakade's NPG, as we see now. 

\begin{example}
The class of Legendre type functions covers the functions inducing Kakade's and Morimura's NPG via their Hessian geometries. More precisely, the following Legendre type functions will be of great interest in the remainder: 
\begin{enumerate}
    \item The functions $\phi_\sigma$ defined in~\eqref{eq:sigmaFunctions} used to define the $\sigma$-NPG are Legendre type functions for $\sigma\in[1, \infty)$. Note that this includes the Fisher geometry, corresponding to Morimura's NPG for $\sigma=1$ but excludes the Euclidean geometry, which corresponds to $\sigma=0$.
    \item The conditional entropy $\phi_C$ defined in~\eqref{eq:conditionalEntropySA} is a Legendre type function. The Hessian geometry of this function induces Kakade's NPG. 
    Note that in this case the gradient blowup holds on the boundary $\mathcal N$ but not on the boundary of $\Delta_{\SS\times\AA}$ or even $\mathbb R_{\ge0}^{\SS\times\AA}$.
\end{enumerate}
\end{example}

The definition of a Legendre function with the gradient blowing up at the boundary of the feasible set prevents the gradient flow from reaching the boundary in finite time and thus ensures the global existence of the gradient flow. 

Let us now turn towards Kakade's natural policy gradient, which is the Hessian NPG induced by the conditional entropy $\phi_C$ defined in~\eqref{eq:conditionalEntropy}. 
The Bregman divergence of the conditional entropy (see~\citep{polyanskiy2014lecture}) is given by 
\begin{align*}
    D_{\phi_C}(\eta_1, \eta_2) & = \sum_{s,a} \eta_1(s, a)\log\left(\frac{\eta_1(s,a)}{\eta_2(s,a)} \right) - \sum_{s,a} \eta_1(s, a)\log\left(\frac{\sum_{a'}\eta_1(s,a')}{\sum_{a'}\eta_2(s,a')} \right) \\
& = D_{KL}(\eta_1, \eta_2) - D_{KL}(\rho_1, \rho_2) = \sum_{s}\rho_1(s)D_{KL}(\eta_1(\cdot|s), \eta_2(\cdot|s)),
\end{align*}
which has been studied in the context of mirror descent algorithms of the linear programming formulation of MDPs in~\cite{neu2017unified}.

\begin{theorem}[Convergence of Kakade's NPG flow for unregularized reward]\label{thm:convergenceKNPGFlow}
Consider Setting~\ref{set:MDPconvergence} with $\phi=\phi_C$ being the conditional entropy defined in~\eqref{eq:conditionalEntropySA} and let $\mathfrak R(\eta) = \langle r, \eta\rangle$ denote the unregularized reward and fix an element $\eta_0\in\operatorname{int}(\mathcal N)$. Then there exists a unique global solution $\eta\colon[0, \infty)\to \operatorname{int}(\mathcal N)$ of Kakade's NPG flow with initial condition $\eta(0) = \eta_0$, i.e., of~\eqref{eq:hessianGFst} with $\phi = \phi_C$, and it holds that 
\[ 
R^\ast-\mathfrak R(\eta(t)) \le t^{-1} D_{\phi_C}(\eta^\ast, \eta_{0})  = t^{-1}\sum_{s} \rho^\ast(s) D_{KL}(\pi^\ast(\cdot|s), \pi_{0}(\cdot|s)),
\] 
where $D_{\phi_C}$ denotes the conditional relative entropy. In particular, we have $\operatorname{dist}(\eta(t), S)\in O(t^{-1})$, where $S = \{\eta\in\mathcal N:\langle r, \eta\rangle = R^\ast\}$ denotes the solution set and $\operatorname{dist}$ denotes the Euclidean distance.
\end{theorem}
\begin{proof}
The well posedness follows by a similar reasoning as in~\cite[Theorem~4.1]{alvarez2004hessian}. Now the result follows directly from Lemma~\ref{prop:convergenceGeneral}. 
\end{proof}

Now we elaborate the consequences of the general convergence result Lemma~\ref{prop:convergenceGeneral} for the case of $\sigma$-NPG flows. 
Here, the study is more delicate since for $\sigma>2$ we typically have $\phi_\sigma(\eta^\ast) = \infty$ since the maximizer $\eta^\ast$ lies at the boundary unless the reward is constant. 

\begin{theorem}[Convergence of $\sigma$-NPG flow for unregularized reward]\label{cor:convSigmaNPGFlow}
Consider Setting~\ref{set:MDPconvergence} with $\phi=\phi_\sigma$ for some $\sigma\in[1, \infty)$ being defined in~\eqref{eq:sigmaFunctions}. Denote the unregularized reward by $\mathfrak R(\eta) = \langle r, \eta\rangle$ and fix an element $\eta_0\in\operatorname{int}(\mathcal N)$.
Then there exists a unique global solution $\eta\colon[0, \infty)\to \operatorname{int}(\mathcal N)$ of the Hessian NPG flow~\eqref{eq:hessianGFst} with inital condition $\eta(0) = \eta_0$ and it holds that $R^\ast - \mathfrak R(\eta(t)) = O(f_\sigma(t))$ as $t\to\infty$, where 
    \[ 
    f_\sigma(t) \coloneqq \begin{cases}
    t^{-1} \quad & \text{for } \sigma \in[1, 2) \\
    \log(t)t^{-1} & \text{for } \sigma = 2 \\
    t^{\sigma-3} & \text{for } \sigma\in (2, \infty).
    \end{cases}
    \]
In particular, we have $\operatorname{dist}(\eta(t), S)\in O(f_\sigma(t))$, where $S = \{\eta\in\mathcal N:\langle r, \eta\rangle = R^\ast\}$ denotes the solution set and $\operatorname{dist}$ denotes the Euclidean distance.
This result covers Morimura's NPG flow as the special case with $\sigma=1$. 
\end{theorem}
\begin{proof}
By the preceding Lemma~\ref{prop:convergenceGeneral}
it suffices to show the well posedness of the $\sigma$-NPG flow. The result \cite[Theorem~4.1]{alvarez2004hessian} guarantees the well posedness for Hessian gradient flows for smooth Legendre type functions. Note however that they work with slightly stronger assumptions, which are that the gradient blowup of the Legendre type functions occurs not only on the boundary of $\mathcal N$ but on the boundary of $\mathbb R^{\SS\times\AA}_{\ge0}$ and that the objective $\mathfrak R$ is globally defined. 
Consolidating the proof in~\cite{alvarez2004hessian} reveals that both of these relaxations do not change the validity or proof of the statement. 

It is easy to see that for $\sigma\ge1$ the functions $\phi_\sigma$ are of Legendre type and smooth and hence we can apply the preceding Lemma~\ref{prop:convergenceGeneral}. Let $\eta^\ast$ be a maximizer, which necessarily lies at the boundary of $\mathcal N$ (except for constant reward) and therefore has at least one zero entry.
For $\sigma\in[1, 2)$ we have that $\phi_\sigma(\eta^\ast)<\infty$ and hence we obtain $R^\ast-\mathfrak R(\eta(t))\in D_{\phi_\sigma}(\eta^\ast, \eta_{0})t^{-1}$. 
Consider now the case $\sigma=2$ and pick $v\in\mathbb R^{\mathcal S\times\mathcal A}$ such that $\eta_\delta\coloneqq\eta^\ast+\delta v\in\operatorname{int}(\mathcal N)$ for small $\delta>0$. Then it holds that 
\begin{align*}
    R^\ast - \mathfrak R(\eta(t)) & = R^\ast - \langle r, \eta_\delta\rangle + \langle r, \eta_\delta\rangle - \mathfrak R(\eta(t)) = O(\delta) + D_{\phi_\sigma}(\eta_\delta, \eta_{0})t^{-1} \\
    & = O(\delta) + \left(\phi_\sigma(\eta_\delta) -  \phi_\sigma(\eta_{0}) - \langle \nabla \phi_\sigma(\eta_{0}), \eta_\delta-\eta_{0}\rangle\right)t^{-1} \\
    & = O(\delta) + O(\log(\delta) + 1) t^{-1}. 
\end{align*}
Setting $\delta = t^{-1}$ we obtain $R^\ast - \mathfrak R(\eta(t)) = O(t^{-1}) + O((\log(t^{-1}) + 1) t^{-1}) = O(\log(t)t^{-1})$ for $t\to\infty$. For $\sigma\in(2, \infty)$ the calculation follows in analogue fashion.
Noting that $\operatorname{dist}(\eta(t), S) \sim R^\ast-\mathfrak R(\eta(t))$ finishes the proof. 
\end{proof}

\begin{remark}
Theorem~\ref{cor:convSigmaNPGFlow} and Theorem~\ref{thm:convergenceKNPGFlow} show global convergence of $\sigma$-NPG and Kakade's NPG flows to a maximizer of the unregularized problem. 
Note that the reason why this is possible is that 
one does not work with a regularized objective but rather with a geometry arising from a regularization but with the original linear objective. 
For $\sigma<1$ the flow may reach a face of the feasible set in finite time; see Figure~\ref{fig:illustrations}. 
For $\sigma \ge3$ Theorem~\ref{cor:convSigmaNPGFlow} is uninformative since $\mathfrak R(\eta(t))$ is non increasing. 
However, in our experiments we observed that (discretizations of) $\sigma$-NPG flows still converge for $\sigma\ge3$, although the plateau problem becomes more pronounced, as can be seen in Figure~\ref{fig:illustrations}. 

Furthermore, one can show that the trajectory converges towards the maximizer that is closest to the initial point $\eta_0$ with respect to the Bregman divergence~\cite{alvarez2004hessian}. 
\end{remark}

\paragraph{Faster rates for $\sigma\in[1,2)$ and Kakade's NPG} 

Now we obtain improved and even linear convergence rates for Kakade's and Morimura's NPG flow for unregularized problems. 
To this end, we first formulate the following general result. 

\begin{lemma}[Convergence rates for gradient flow trajectories]\label{prop:ratesTrajectories} 
Consider Setting~\ref{set:MDPconvergence} and assume that there is a global solution $\eta\colon[0, \infty)\to\operatorname{int}(\mathcal N)$ of the Hessian gradient flow~\eqref{eq:hessianGFst}. 
Assume that there is $\eta^\ast\in\mathcal N$ such that $\phi(\eta^\ast)<+\infty$ as well as a neighborhood $N$ of $\eta^\ast$ in $\mathcal N$ and $\omega\in(0, \infty)$ and $\tau\in[1, \infty)$ such that
\begin{equation}\label{eq:lowerBound}
    \mathfrak R(\eta^\ast) - \mathfrak R(\eta) \ge \omega D_\phi(\eta^\ast, \eta)^\tau \quad \text{for all } \eta\in N.
\end{equation}
Then there is a constant $c>0$ such that
\begin{enumerate}
    \item if $\tau=1$, then $D_\phi(\eta^\ast, \eta(t))\le c e^{-\omega t}$,
    \item if $\tau>1$, then $D_\phi(\eta^\ast, \eta(t))\le c t^{-1/(\tau - 1)}$. 
\end{enumerate} 
\end{lemma}
The lower bound \eqref{eq:lowerBound} 
can be interpreted as a form of strong convexity under which the objective value controls the Bregman divergence and hence convergence in objective value implies convergence of the state-action trajectories in the sense of the Bregman divergence. 
\begin{proof}
The statement of this result can be found in  \cite[Proposition 4.9]{alvarez2004hessian}, where however stronger assumptions are made and hence we provide a short proof. 
First, note that our assumption implies that $\eta^\ast$ is the unique global maximizer of $\mathfrak R$ over $\mathcal N$. 
By~\eqref{eq:derivativeu} it holds that $u(t) \coloneqq \partial_t D_\phi(\eta^\ast, \eta(t))$ is strictly decreasing as long as $\eta(t)\ne\eta^\ast$. Note that if $\eta(t) = \eta^\ast$ for some $t\in[0, \infty)$, we have $u(t') = 0$ for all $t'\ge t$ and hence the statement becomes trivial. Therefore, we can assume $u(t)>0$ for all $t>0$. 
By Lemma~\ref{prop:convergenceGeneral} it holds that $\mathfrak R(\eta(t))\to\mathfrak R(\eta^\ast)$ and hence $\eta(t)\to\eta^\ast$; this is due to the compactness of $\mathcal N$ and because the continuity of $\mathfrak R$ implies that every accumulation point of $\eta(t)$ is a maximizer and thus equal to $\eta^\ast$. 
Hence, $\eta(t)\in N$ for $t\ge t_0$. 
For the statement about the asymptotic behavior we may therefore assume without loss of generality that $\eta(t) \in N$ for all $t\ge0$. 
Combining~\eqref{eq:derivativeu} and~\eqref{eq:lowerBound} we obtain $u'(t) \le -\omega u(t)^\tau$. Dividing by the right hand side and integrating the inequality we obtain $u(t)\le u(0)e^{-\omega t}$ for $\tau=1$ and $u(t)\le \omega^{1/(1-\tau)}(\tau-1)^{1/(1-\tau)}t^{1/(1-\tau)}$. 
\end{proof} 

\begin{theorem}[Linear convergence of unregularized Kakade's NPG flow]\label{thm:linearconvergenceKakade}
Consider Setting~\ref{set:MDPconvergence}, where $\phi = \phi_C$ is the conditional entropy defined in~\eqref{eq:conditionalEntropySA} 
and assume that there is a unique maximizer $\eta^\ast$ of the unregularized reward $\mathfrak R$.
Then $R^\ast - \mathfrak R(\eta(t)) = O(e^{-ct})$ for some $c>0$. 
\end{theorem}
\begin{proof}
Let $\phi_C$ denote the conditional entropy, so that $D_{\phi_C}(\eta^\ast, \eta) = D_{KL}(\eta^\ast, \eta) - D_{KL}(\rho^\ast, \rho) \le D_{KL}(\eta^\ast, \eta)$. Hence, we obtain just like in the case of $\sigma$-NPG flows for $\sigma=1$ that $D_{\phi_C}(\eta^\ast, \eta) = O(\mathfrak R(\eta^\ast) - \mathfrak R(\eta))$ for $\eta\to\eta^\ast$ and hence $D_{\phi_C}(\eta^\ast, \eta(t)) = O(e^{-ct})$ for some $c>0$ by Lemma~\ref{prop:ratesTrajectories}. 
Hence, it remains to estimate $\mathfrak R(\eta^\ast) - \mathfrak R(\eta)  = O(\lVert \eta^\ast - \eta \rVert_1)$ by the conditional relative entropy $D_{\phi_C}(\eta^\ast, \eta)$.
Note that $\pi^\ast$ is a deterministic policy and hence we can write $\pi^\ast(a^\ast_s|s) = 1$ and estimate
\begin{align*}
    D_{\phi_C}(\eta^\ast, \eta)  & = \sum_s \rho^\ast(s) D_{KL}(\pi^\ast(\cdot|s), \pi^\ast(\cdot|s)) = - \sum_s \rho^\ast(s) \log(\pi(a_s^\ast|s)) \\& \ge \sum_s \rho^\ast(s) (1-\pi(a_s^\ast|s)) = 2^{-1}\sum_s \rho^\ast(s) \lVert \pi^\ast(\cdot|s) - \pi(\cdot|s) \rVert_1  \\ & \ge 2^{-1}\left(\min_s\rho^\ast(s) \right) \cdot \lVert \pi^\ast - \pi \rVert_1.
\end{align*}
Here, we have used $\log(t) \le t-1$ as well as
\begin{align*}
    \lVert \pi^\ast(\cdot|s) - \pi(\cdot|s) \rVert_1 & = \sum_{a\ne a^\ast_s} \lvert \pi^\ast(a|s) - \pi(a|s) \rvert +  \lvert \pi^\ast(a|s) - \pi(a|s) \rvert \\
    & = \sum_{a\ne a_s^\ast} \pi(a|s) + (1-\pi(a_s^\ast|s)) = 2(1-\pi(a_s^\ast|s)).
\end{align*} 
Now we observe that the mapping $\pi\mapsto \eta$ is $L$-Lipschitz with constant $L=O((1-\gamma)^{-1})$. The fact that $L=O((1-\gamma)^{-1})$ follows from the policy gradient theorem as $\partial_{\pi(a|s)} \eta^\pi = \rho^\pi(s) (I-\gamma P_\pi^\top)^{-1}e_{(s,a)}$, see also~\cite[Proposition 48]{mueller2022geometry}. 
In turn, it holds that 
\[ 
\lVert \eta^\ast - \eta(t) \rVert_1\le L \lVert \pi^\ast - \pi(t) \rVert_1 \le \frac{2L}{ \min_s\rho^\ast(s)} \cdot D_{\phi_C}(\eta^\ast, \eta(t)) = O(e^{-ct}). 
\]
Altogether this implies $R^\ast - \mathfrak R(\eta(t)) = O(e^{-ct})$, which concludes the proof. 
The $O$ notation hides constants that scale with the norm of the instantaneous reward vector $r$, inversely with the minimum state probability, and inversely with $(1-\gamma)$ where $\gamma$ is the discount rate. 
\end{proof}

\begin{theorem}[Improved convergence rates for $\sigma$-NPG flow]\label{thm:sigmaNPGflowfastrates}
Consider Setting~\ref{set:MDPconvergence}, where $\phi = \phi_\sigma$ for some $\sigma\in [1, 2)$ as defined in \eqref{eq:sigmaFunctions}, and assume that there is a unique maximizer $\eta^\ast$ of the unregularized reward $\mathfrak R$. Then $R^\ast - \mathfrak R(\eta(t)) \in O(g_\sigma(t))$, where 
    \[ 
    g_\sigma(t) = \begin{cases}
        e^{-ct} \quad & \text{if } \sigma = 1 \\
        t^{-1/(\sigma-1)} & \text{if } \sigma\in(1,2),
    \end{cases}
    \]
for some $c>0$, where $\eta\colon[0, \infty)\to\operatorname{int}(\mathcal N)$ denotes the solution of the $\sigma$-NPG flow. 
\end{theorem}
\begin{proof}
The key is to show that~\eqref{eq:lowerBound} holds for $\tau = (2-\sigma)^{-1}\ge1$. 
To see that this holds, we first consider the case $\sigma \in(1, 2)$, where we obtain
\[ D_\sigma(\eta^\ast, \eta) = \sum_{s, a}\frac{\eta^\ast(s,a)^{2-\sigma}}{(1-\sigma)(2-\sigma)} - \sum_{s, a}\frac{\eta(s,a)^{2-\sigma}}{(1-\sigma)(2-\sigma)} - \sum_{s, a} \frac{\eta(s,a)^{1-\sigma}(\eta^\ast(s,a)-\eta(s,a))}{1-\sigma}. \]
We can bound every summand by $O(\lvert \eta^\ast(s,a) - \eta(s,a) \rvert)$ if $\eta^\ast(s,a)>0$ and $O(\lvert \eta^\ast(s,a) - \eta(s,a) \rvert^{2-\sigma})$ if $\eta^\ast(s,a)=0$ for $\eta\to\eta^\ast$ respectively. Overall, this shows that
\begin{align*}
    D_\sigma(\eta^\ast, \eta) = O(\lVert \eta^\ast - \eta \rVert^{2-\sigma}) = O((\mathfrak R(\eta^\ast) - \mathfrak R(\eta))^{2-\sigma}) \quad \text{for } \eta\to\eta^\ast,
\end{align*}
where the last estimate holds since $\eta^\ast$ is the unique minimizer of the linear function $\mathfrak R$ over the polytope $\mathcal N$. By Lemma~\ref{prop:ratesTrajectories} we obtain $D_\sigma(\eta^\ast, \eta(t))=O(t^{-1/(\tau-1)}) = O(t^{-(2-\sigma)/(\sigma-1)})$. It remains to estimate the value of $\mathfrak R$ by means of the Bregman divergence $D_\sigma$. For this, we note that $\mathfrak R(\eta^\ast) - \mathfrak R(\eta) = O(\lVert \eta^\ast-\eta\rVert_1)$ and estimate the individual terms. First, note that for $x\to y$ (with $x, y\ge0$) it holds that 
\begin{align*}
    \lvert x - y \rvert & = O\left(\left(\frac{y^{2-\sigma}}{(1-\sigma)(2-\sigma)} - \frac{x^{2-\sigma}}{(1-\sigma)(2-\sigma)} - \frac{x^{1-\sigma}(y-x)}{1-\sigma}\right)^{1/(2-\sigma)}\right).
\end{align*}
For $y=0$ this is immediate and for $y>0$ the local strong convexity of $x\mapsto x^{2-\sigma}$ around $y$ implies 
    \begin{align*}
    \lvert x-y\rvert 
    &= O\left(\left(y^{2-\sigma} - x^{2-\sigma} - (2-\sigma)x^{1-\sigma}(y - x)\right)^{1/2}\right) \\
    &= O\left(\left(y^{2-\sigma} - x^{2-\sigma} - (2-\sigma)x^{1-\sigma}(y - x)\right)^{1/(2-\sigma)}\right) 
    \end{align*}
for $x\to y$. Now, Jensen's inequality yields
\[ 
\lVert \eta^\ast - \eta \rVert_1 = O(D_\sigma(\eta^\ast, \eta)^{1/(2-\sigma)}). 
\]
Overall, we obtain
\[ 
\mathfrak R(\eta^\ast) - \mathfrak R(\eta(t)) = O(\lVert \eta^\ast- \eta(t) \rVert_1^{1/(2-\sigma)}) =  O(t^{-1/(1-\sigma)}).  
\]

The case $\sigma=1$ can be treated similarly, where one obtains $D_\sigma(\eta^\ast, \eta)=O(\lVert \eta^\ast-\eta\rVert)=O(\mathfrak R(\eta^\ast) - \mathfrak R(\eta))$ for $\eta\to\eta^\ast$.
To relate the $L^1$-norm to the Bregman divergence one can employ Pinsker's inequality $\lVert \eta^\ast-\eta\rVert_1\le \sqrt{2 D_{KL}(\eta^\ast, \eta)} = \sqrt{2 D_{\sigma}(\eta^\ast, \eta)}$. 
\end{proof}

Compared to Theorem~\ref{cor:convSigmaNPGFlow} the above Theorem~\ref{thm:sigmaNPGflowfastrates} improves the $O(t^{-1})$ rates for $\sigma\in[1,2)$. Later, we conduct numerical experiments that indicate that the rates $O(t^{-1/(\sigma-1)})$ also hold for $\sigma\ge2$ and are tight. 

\paragraph{Numerical examples} 

We use the following example proposed by Kakade~\cite{kakade2001natural} and which 
was also used in~\cite{bagnell2003covariant, morimura2008new}. 
We consider an MDP with two states $s_1, s_2$ and two actions $a_1, a_2$, with the transitions and instantaneous rewards shown in Figure~\ref{fig:kakadeExample}.  
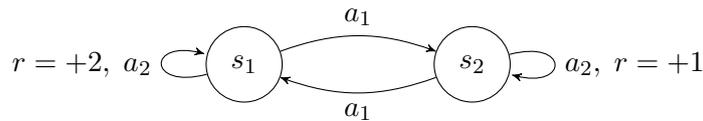
\begin{figure}[ht]
    \centering
    \begin{tikzpicture}
    \node[shape=circle, draw=black, minimum size=1cm] (A) at (0,0) {\(s_1\)};
    \node[shape=circle,draw=black,minimum size=1cm] (B) at (3,0) {\(s_2\)};
    \path [->] (B) [bend left=20] edge node[below] {\(a_1\)} (A);
    \path [->] (A) [bend left=20] edge node[above] {\(a_1\)} (B);
    \path [->] (A) [loop left] edge node[left] {\(r=+2,\; a_2\)} (A);
    \path [->] (B) [loop right] edge node[right] {\(a_2,\; r=+1\)} (B);
\end{tikzpicture}
\caption{Transition graph and reward of the MDP example.} 
\label{fig:kakadeExample}
\end{figure}

We adopt the initial distribution $\mu(s_1)=0.2, \mu(s_2) = 0.8$ and work with a discount factor of $\gamma = 0.9$, whereas Kakade studied the mean reward case. Note however that the experiments can be performed for arbirtrarily large discount factor, where we chose a smaller factor since the correspondence between the policy polytope and the state-action polytope is clearer to see in the illustrations. 
We consider tabular softmax policies and plot the trajectories of vanilla PG, Kakade's NPG, and $\sigma$-NPG for the values $\sigma\in\{-0.5, 0, 0.5, 1, 1.5, 2, 3, 4\}$ for $30$ random (but the same for every method) initializations. 
We plot the trajectories in the state-action space (Figure~\ref{fig:illustrations}) and in the policy polytope (Figure~\ref{fig:policyTrajectories}). 
In order to put the convergence results from this section into perspective, we plot the evolution of the optimality gap $R^\ast-R(\theta(t))$ (Figure~\ref{fig:convergenceRates}). 
We use an adaptive step size $\Delta t_k$, which 
prevents the blowup of the parameters for $\sigma<1$, and hence we do not consider the number of iterations but rather the sum of the step sizes as a measure for the time, $t_n = \sum_{k=1}^n \Delta t_k$. 
For vanilla PG and $\sigma\in(1, 2)$ we expect a decay at rate $O(t^{-1})$~\cite{mei2020global} and $O(t^{-1/(\sigma-1)})$ by Theorem~\ref{thm:sigmaNPGflowfastrates}. 
Therefore we use a logarithmic (on both scales) plot for vanilla PG and $\sigma>1$ and also indicate the predicted rate using a dashed gray line. For Kakade's and Morimuras NPG we expect linear convergence by Theorem~\ref{thm:linearconvergenceKakade} and \ref{thm:sigmaNPGflowfastrates} respectively and hence use a semi-logarithmic plot. 

\begin{figure}[ht]
\centering
\resizebox{1\textwidth}{!}{
\begin{tikzpicture}[scale=1.3]
\node[inner sep=0pt] (r1) at (0,0)
    {\includegraphics[width=3cm, clip=true, trim=1.5cm 1.5cm 1.5cm 0cm]{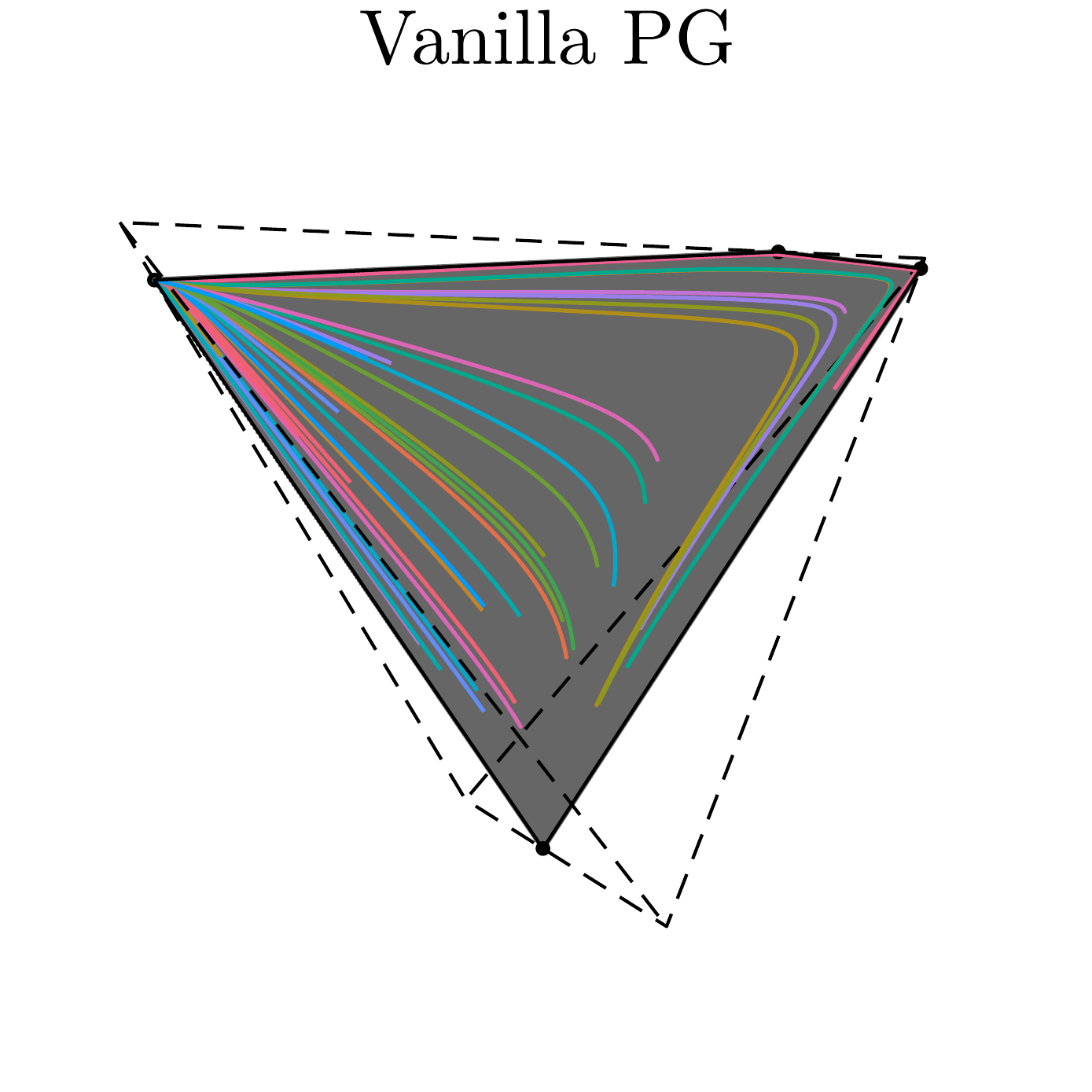}};
\node[inner sep=0pt] (r2) at (2.4,0)
    {\includegraphics[width=3cm, clip=true, trim=1.5cm 1.5cm 1.5cm 0cm]{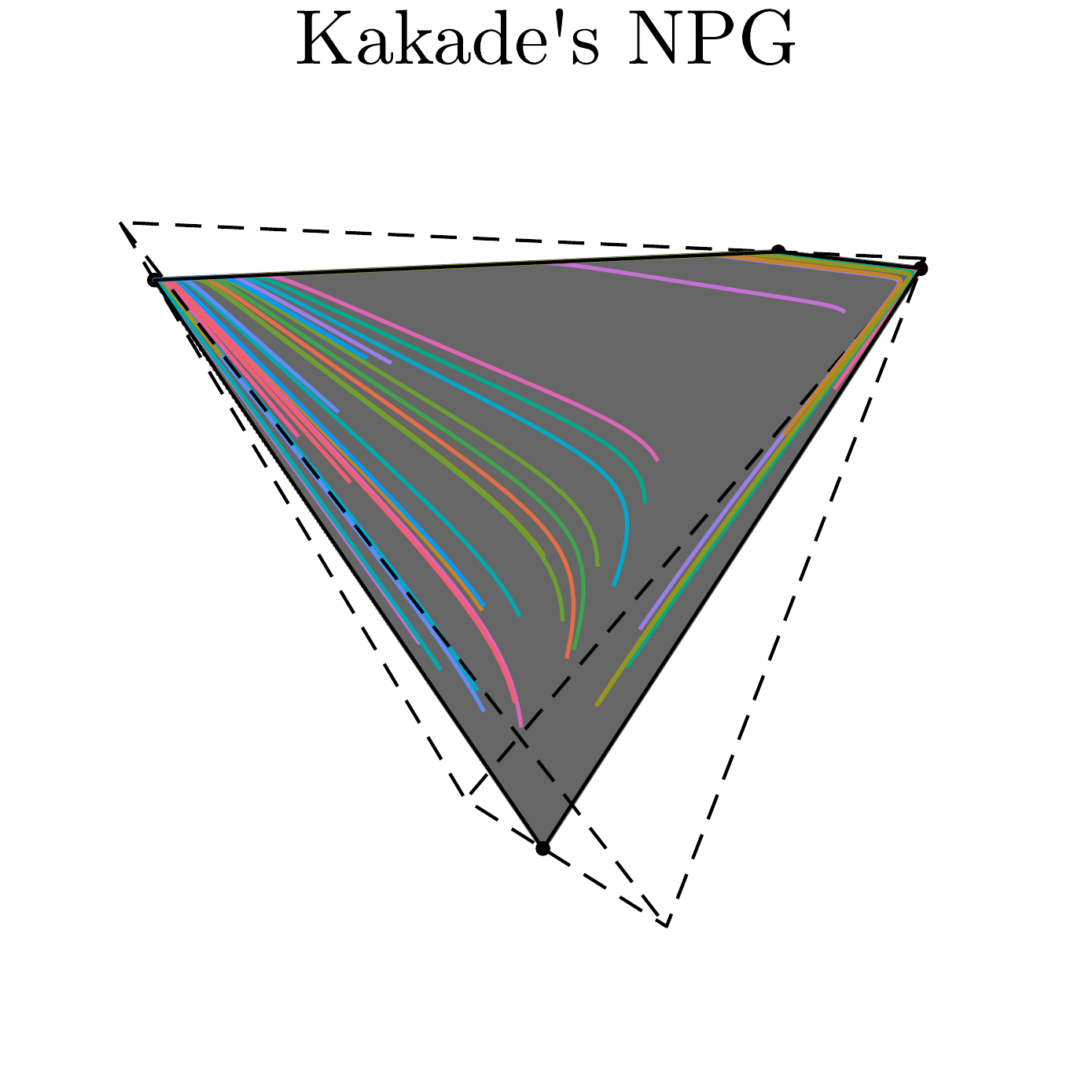}};
\node[inner sep=0pt] (r3) at (4.8,0)
    {\includegraphics[width=3cm, clip=true, trim=1.5cm 1.5cm 1.5cm 0cm]{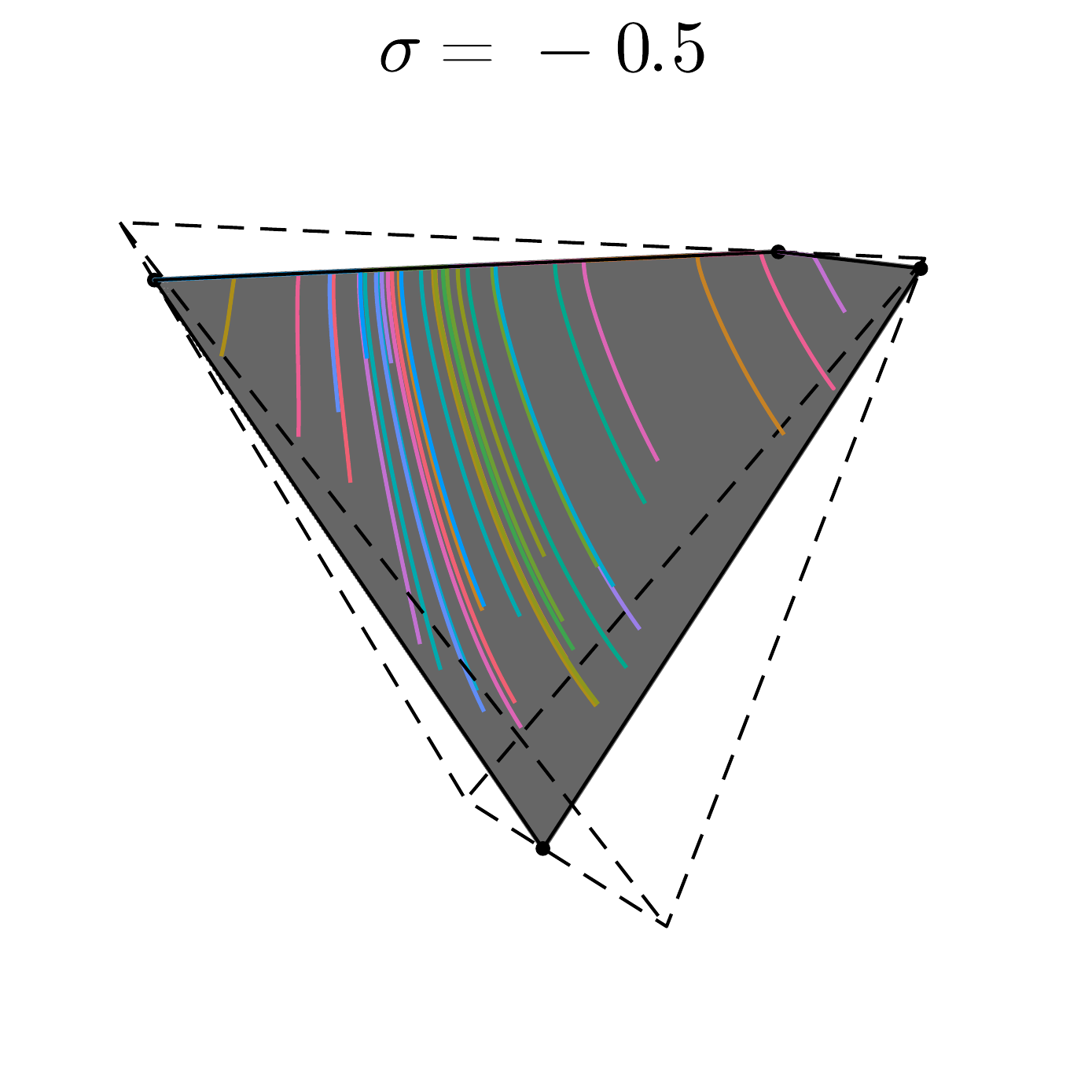}};
\node[inner sep=0pt] (r4) at (7.2,0)
    {\includegraphics[width=3cm, clip=true, trim=1.5cm 1.5cm 1.5cm 0cm]{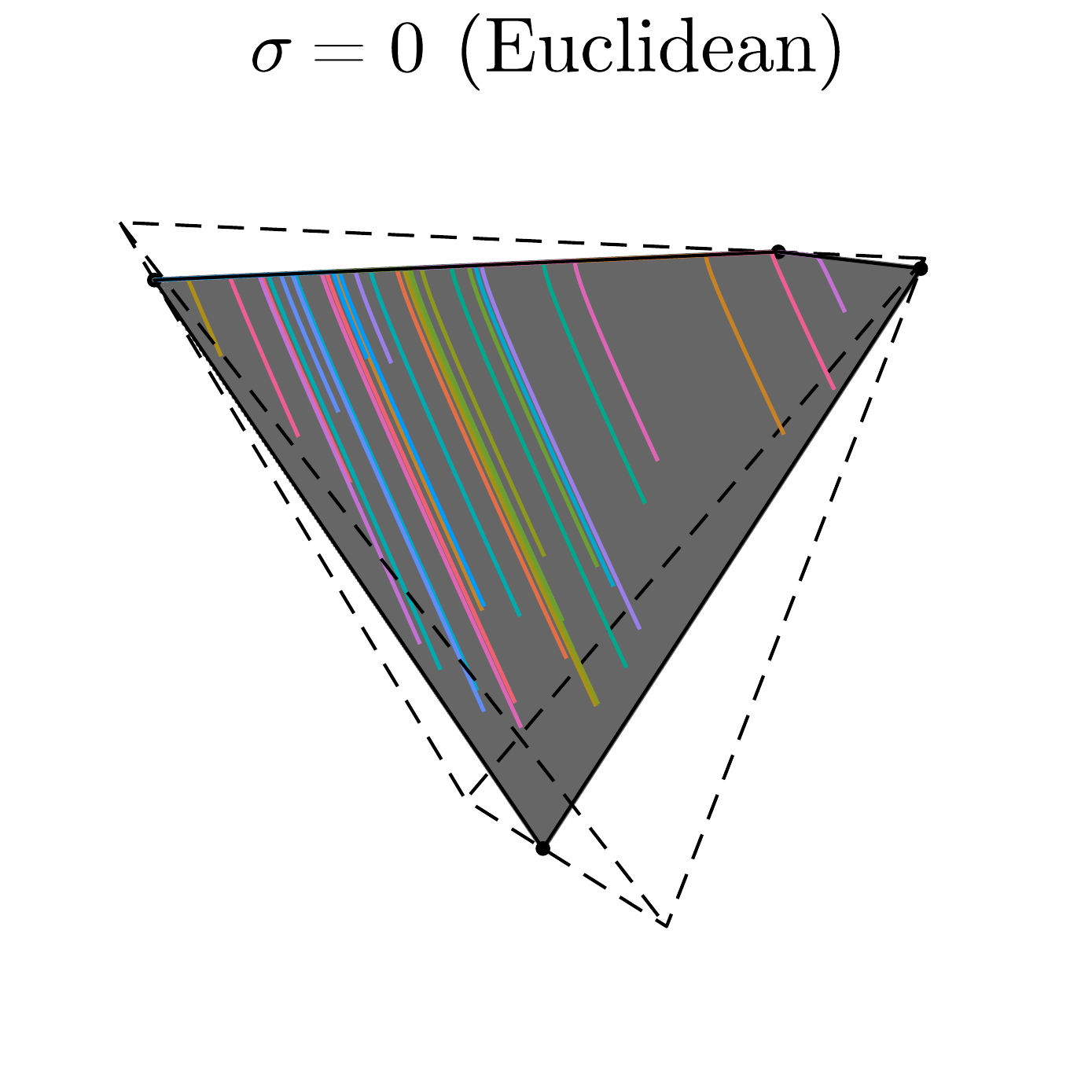}};
\node[inner sep=0pt] (r5) at (9.6,0)
    {\includegraphics[width=3cm, clip=true, trim=1.5cm 1.5cm 1.5cm 0cm]{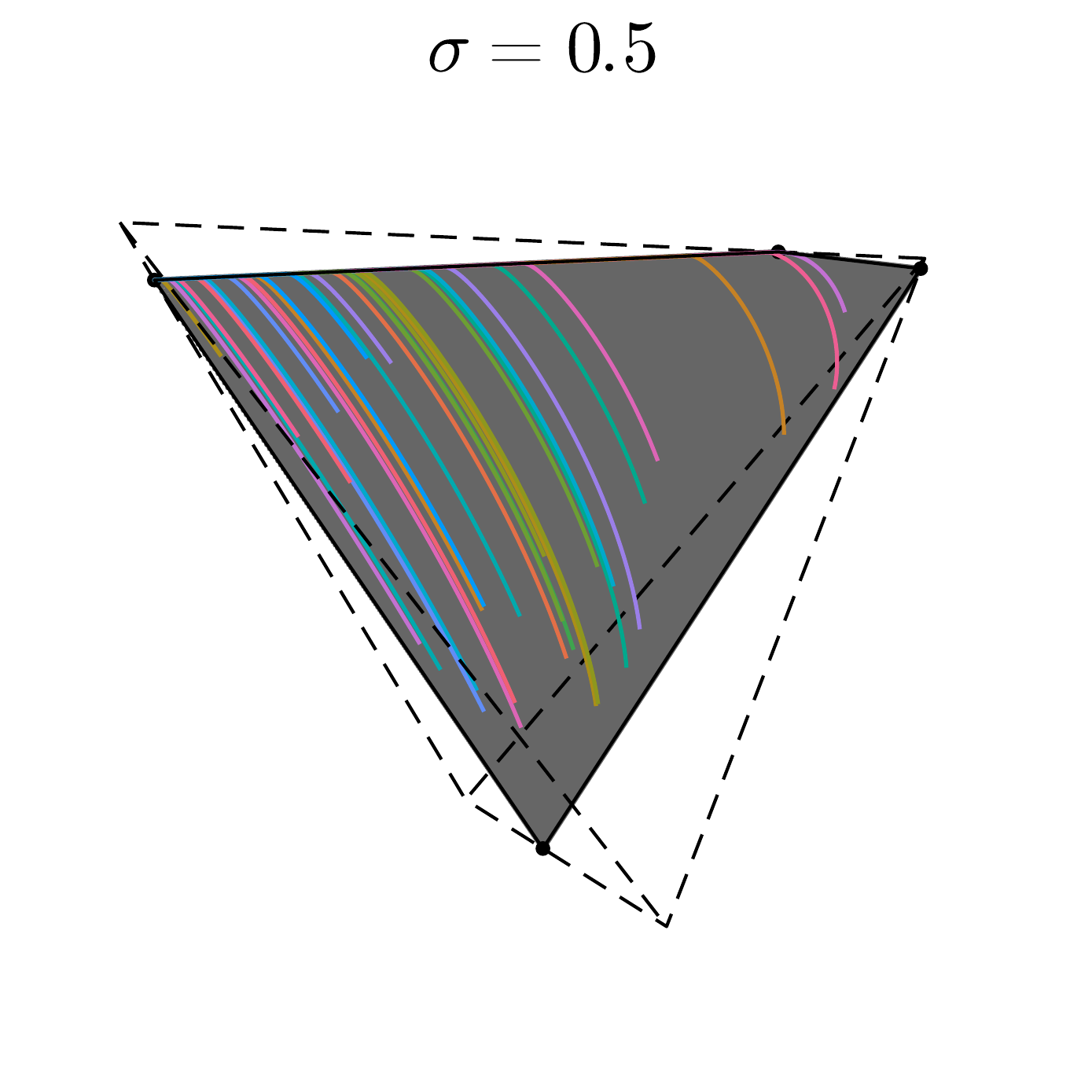}};
\node[inner sep=0pt] (r6) at (0,-2.6)
    {\includegraphics[width=3cm, clip=true, trim=1.5cm 1.5cm 1.5cm 0cm]{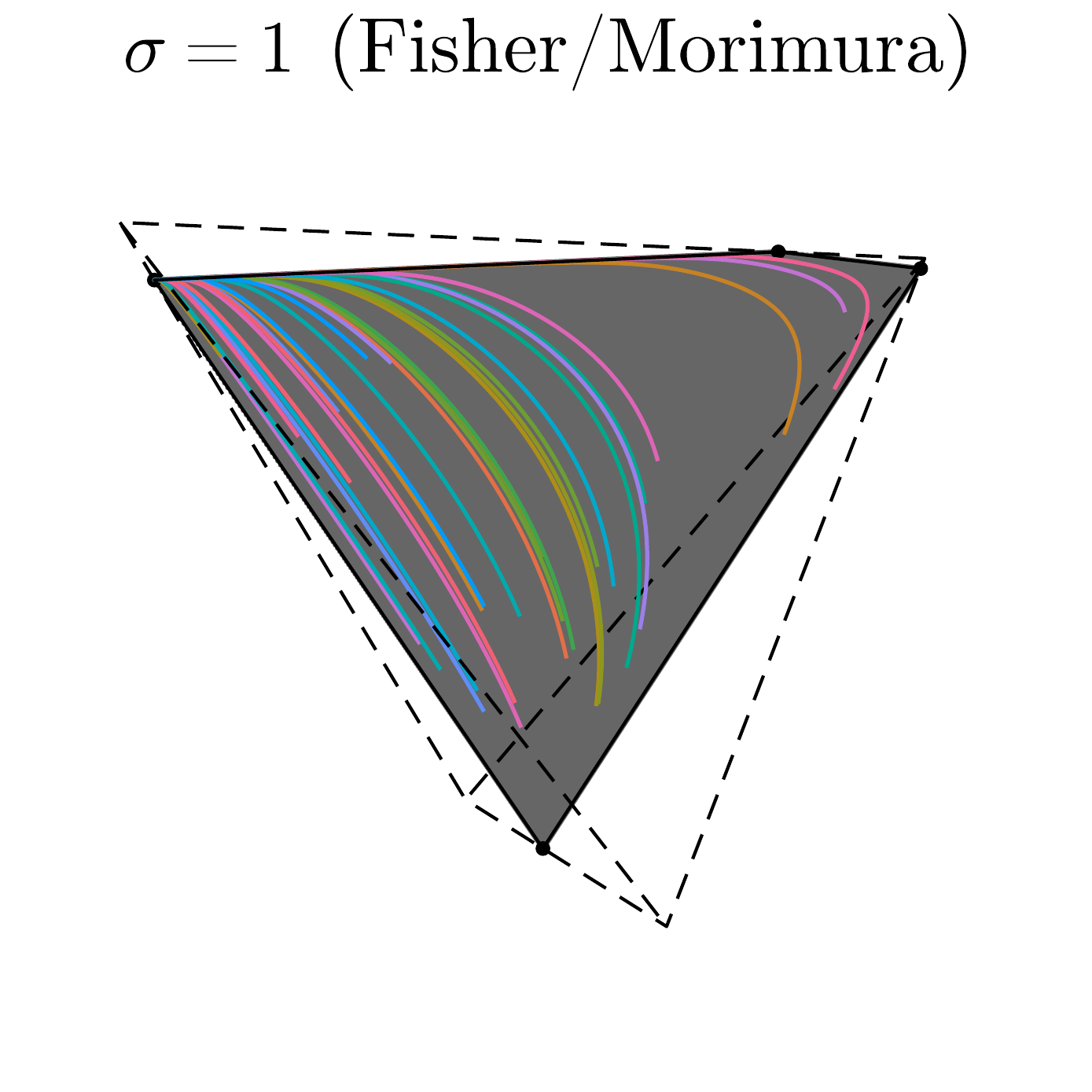}};
\node[inner sep=0pt] (r7) at (2.4,-2.6)
    {\includegraphics[width=3cm, clip=true, trim=1.5cm 1.5cm 1.5cm 0cm]{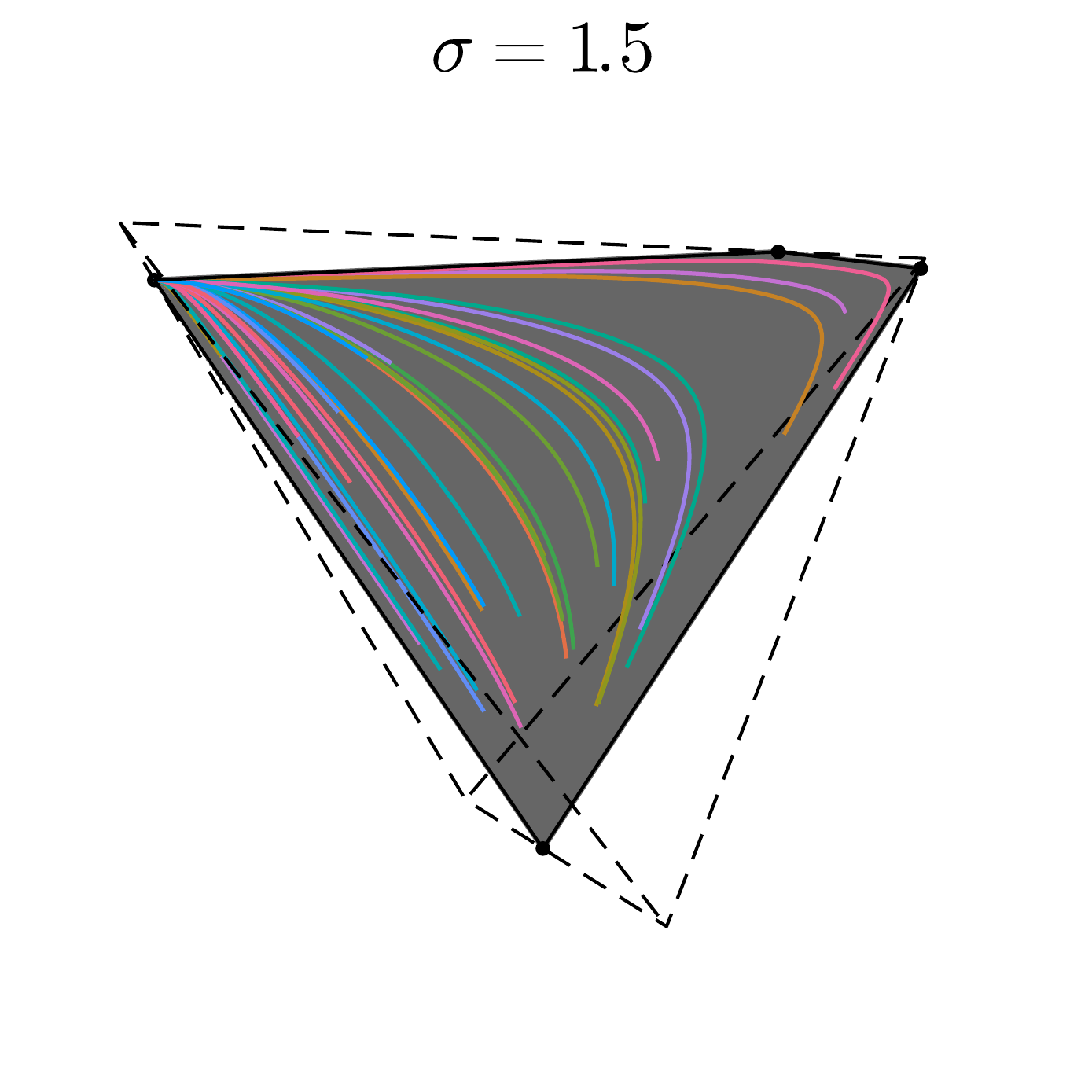}};
\node[inner sep=0pt] (r8) at (4.8,-2.6)
    {\includegraphics[width=3cm, clip=true, trim=1.5cm 1.5cm 1.5cm 0cm]{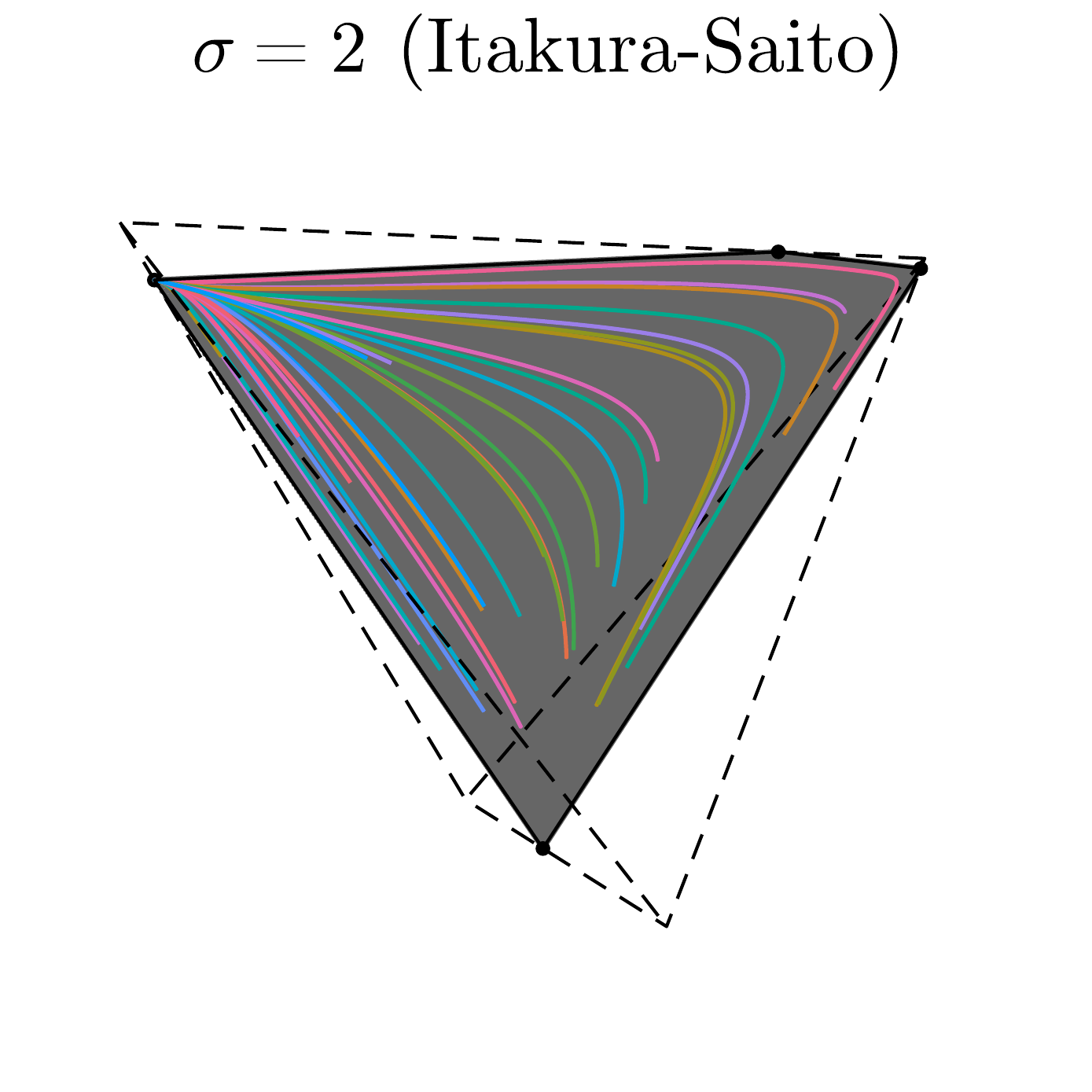}};
\node[inner sep=0pt] (r9) at (7.2,-2.6)
    {\includegraphics[width=3cm, clip=true, trim=1.5cm 1.5cm 1.5cm 0cm]{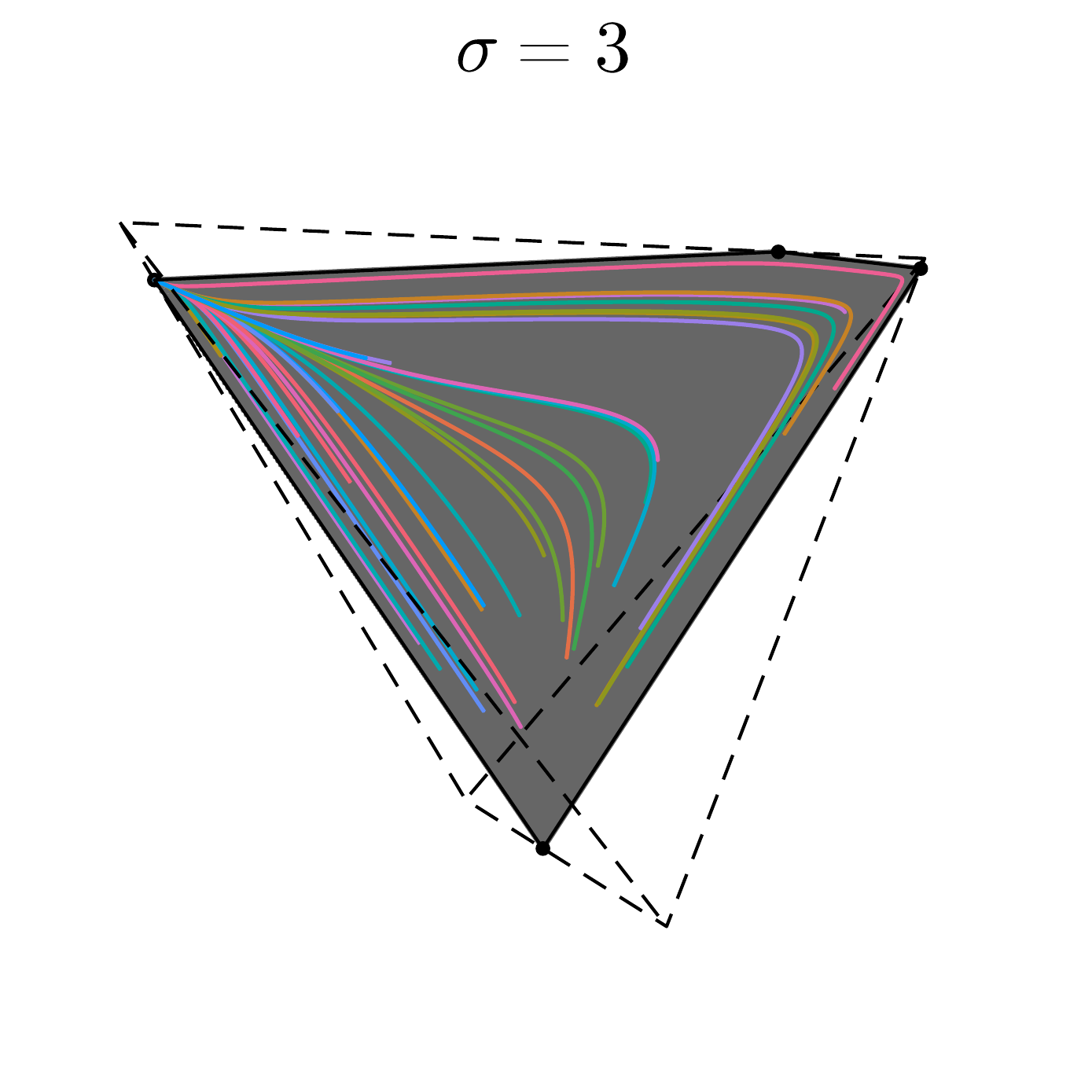}};
\node[inner sep=0pt] (r10) at (9.6,-2.6)
    {\includegraphics[width=3cm, clip=true, trim=1.5cm 1.5cm 1.5cm 0cm]{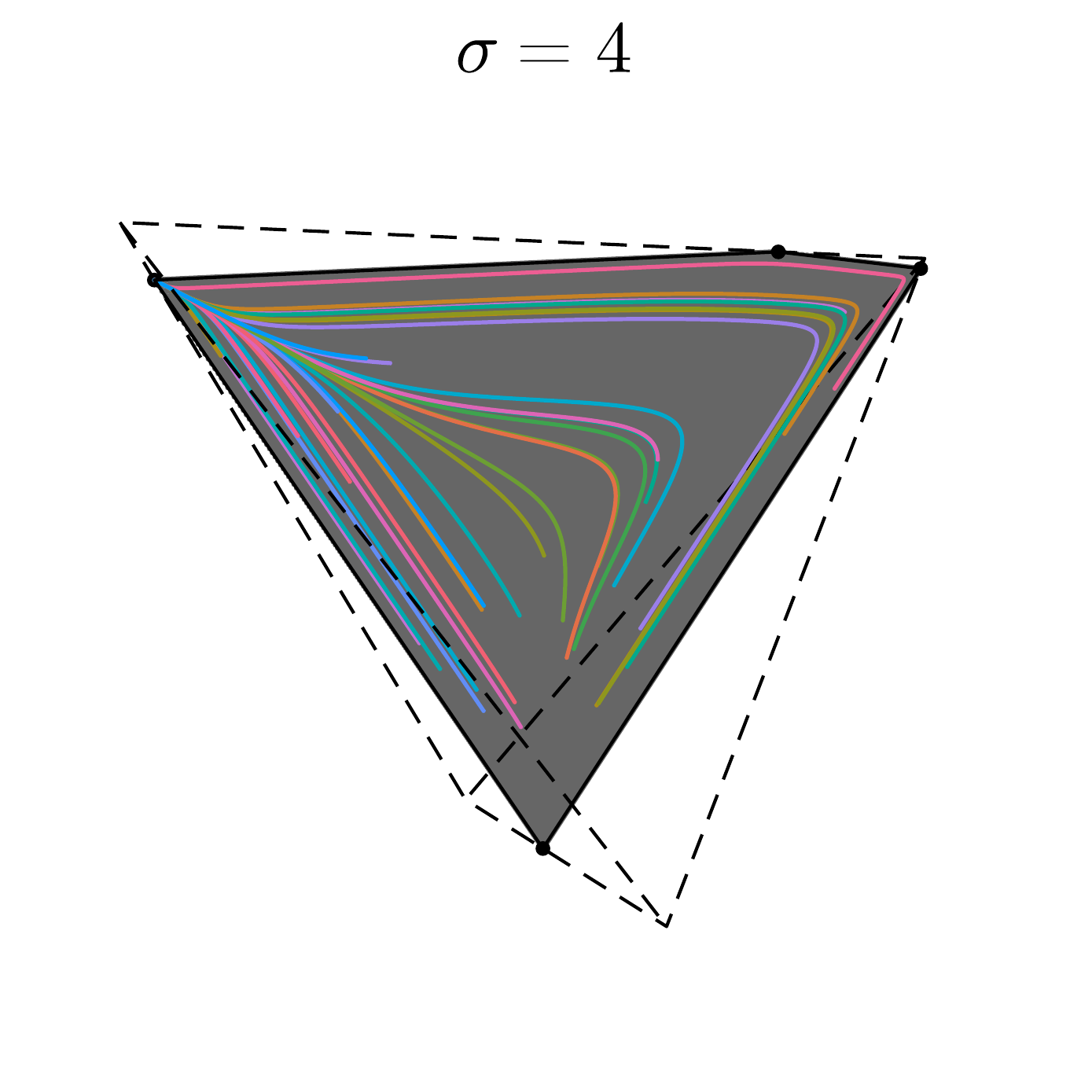}};
\end{tikzpicture}
}
\caption{State-action trajectories for different PG methods, which are vanilla PG, Kakade's NPG and $\sigma$-NPG, where Morimura's NPG corresponds to $\sigma=1$; the state-action polytope is shown in gray inside a three dimensional projection of the the simplex $\Delta_{\SS\times\AA}$;  shown are trajectories with the same random $30$ initial values for every method; the maximizer $\eta^\ast$ is located at the upper left corner of the state-action polytope. }
\label{fig:illustrations}
\end{figure}
First, we note that for $\sigma \in\{-0.5, 0, 0.5\}$ the trajectories of $\sigma$-NPG flow hit the boundary of the state-action polytope $\mathcal N$, which is depicted in gray inside the simplex $\Delta_{\mathcal S\times\AA}$. This is consistent with our analysis, since the functions $\phi_\sigma$ are Legendre type functions only for $\sigma\in[1, \infty)$ and hence only in this case is the NPG flow is guaranteed to exhibit long time solutions. 
However, we observe finite-time convergence of the trajectories towards the global optimum (see Figure~\ref{fig:convergenceRates}), which we suspect to be due to the discretization error. 

For the other methods, namely vanilla PG, Kakade's NPG and $\sigma$-NPG with $\sigma\in[1,\infty)$, Theorem~\ref{cor:convSigmaNPGFlow} and Theorem~\ref{thm:convergenceKNPGFlow} show the global convergence of the gradient flow trajectories, which we also observe both in state-action space and in policy space (see Figures~\ref{fig:illustrations} and \ref{fig:policyTrajectories} respectively). 
When considering the convergence in objective value 
we observe that both Kakade's and Morimura's NPG exhibit a linear rate of convergence as asserted by Theorem~\ref{thm:linearconvergenceKakade} and Theorem~\ref{thm:sigmaNPGflowfastrates}, 
whereby Kakade's NPG appears to have more severe plateaus in some examples. 
For vanilla PG and $\sigma >1$ we observe a sublinear convergence rate of $O(t^{-1})$ and $O(t^{-1/(\sigma-1)})$ respectively, which are shown via dashed gray lines in each case. This confirms the convergence rate $O(t^{-1})$ for vanilla PG~\cite{mei2020global} and indicates that the rate $O(t^{-1/(\sigma-1)})$ shown for $\sigma\in(1,2)$ is also valid in the regime $\sigma>2$. 
Finally, we observe that larger $\sigma$ appears to lead to more severe plateaus, which is apparent in the convergence in objective and also from the evolution in policy space and in state-action space. 

\begin{figure}[ht]
\centering
\resizebox{1\textwidth}{!}{
\begin{tikzpicture}
\node[inner sep=0pt] (r1) at (0,0)
    {\includegraphics[width=3cm]{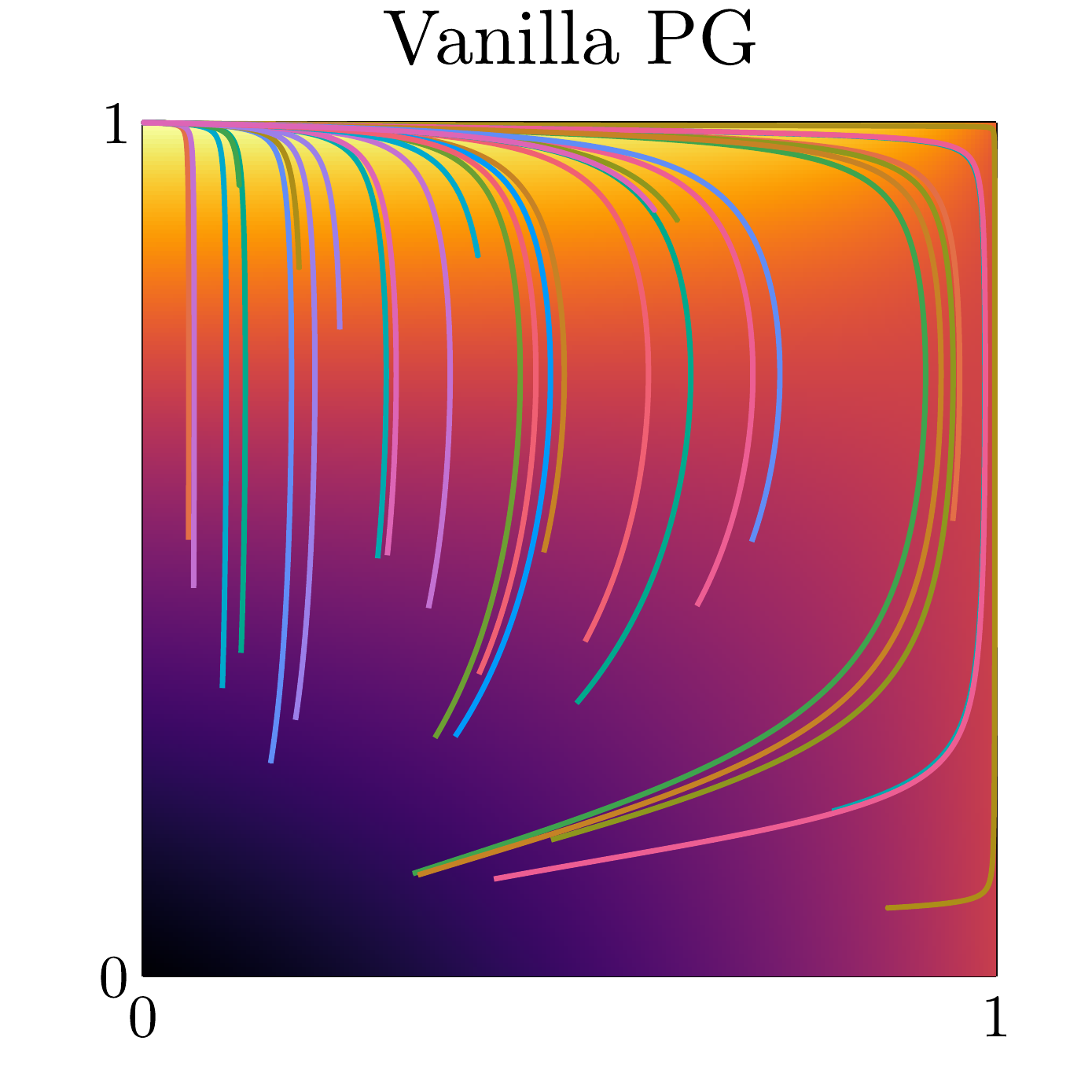}};
\node[inner sep=0pt] (r2) at (3.1,0)
    {\includegraphics[width=3cm]{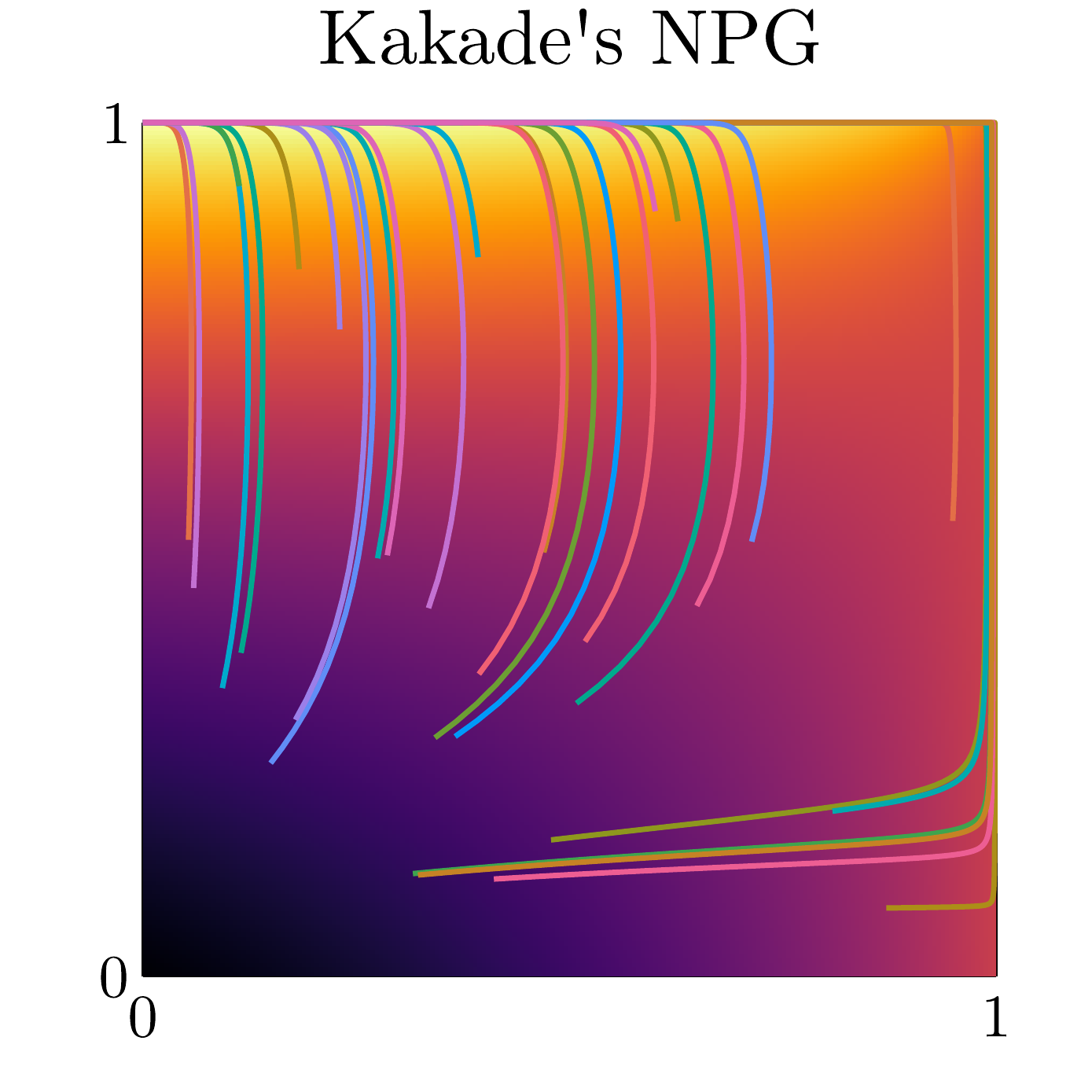}};
\node[inner sep=0pt] (r3) at (6.2,0)
    {\includegraphics[width=3cm]{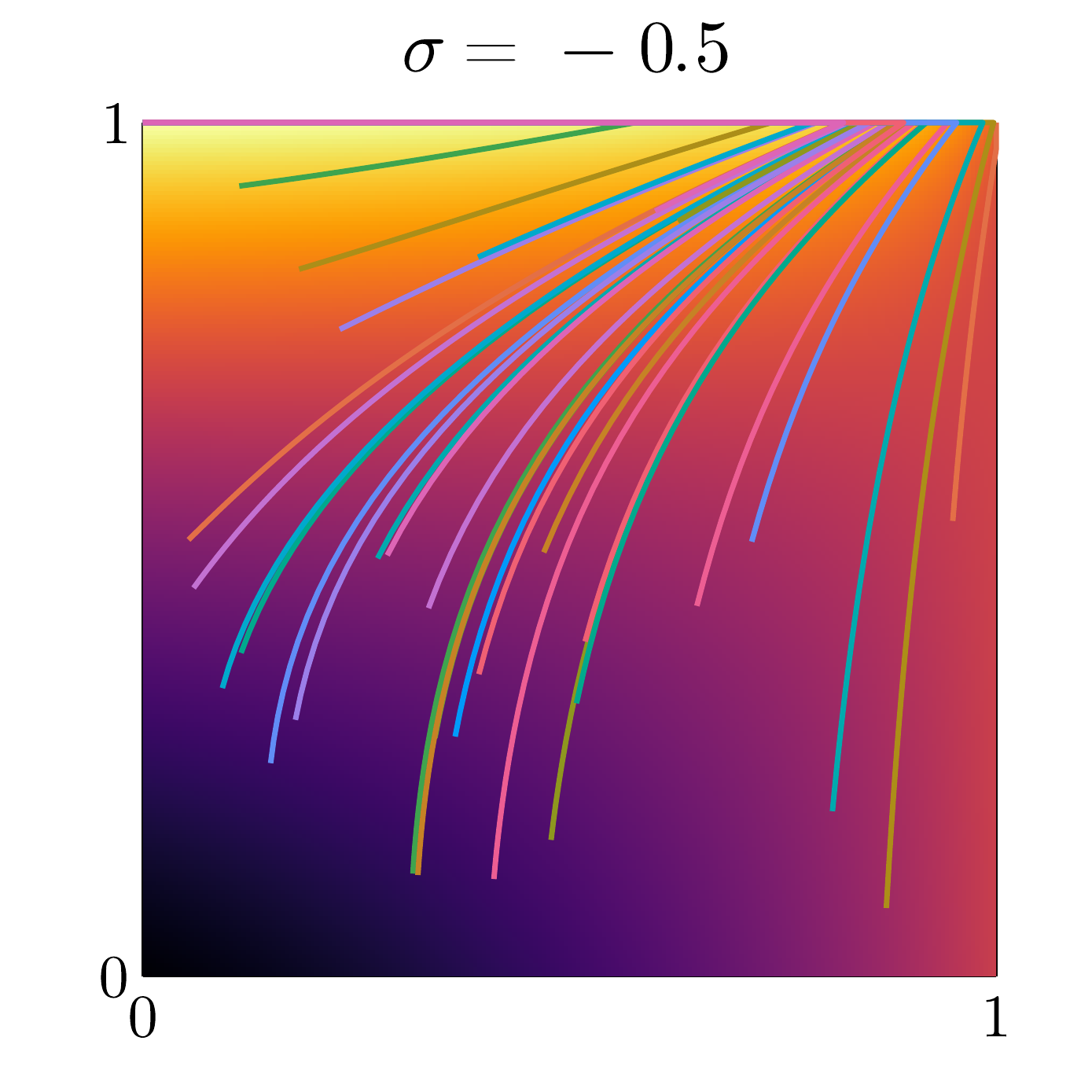}};
\node[inner sep=0pt] (r4) at (9.3,0)
    {\includegraphics[width=3cm]{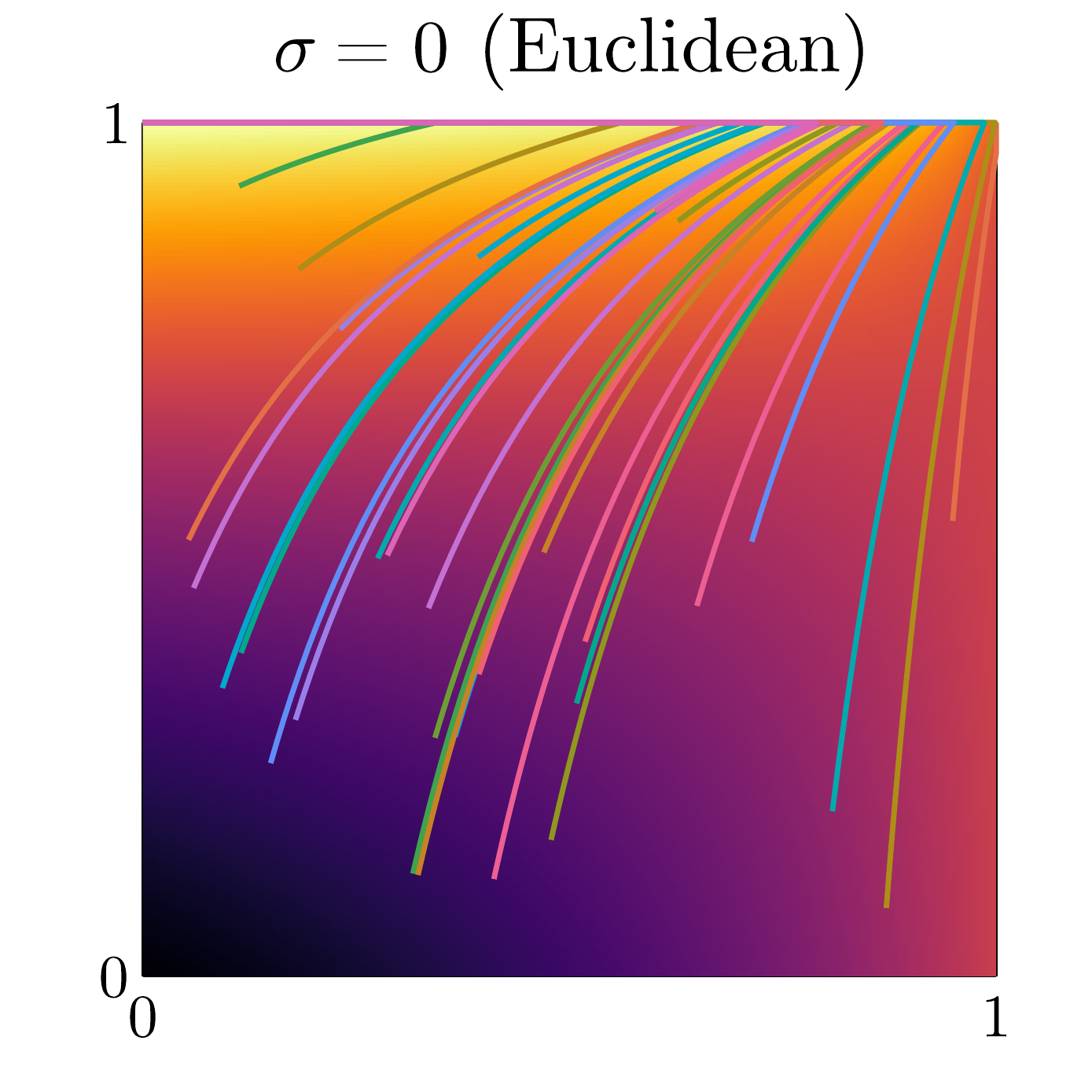}};
\node[inner sep=0pt] (r5) at (12.4,0)
    {\includegraphics[width=3cm]{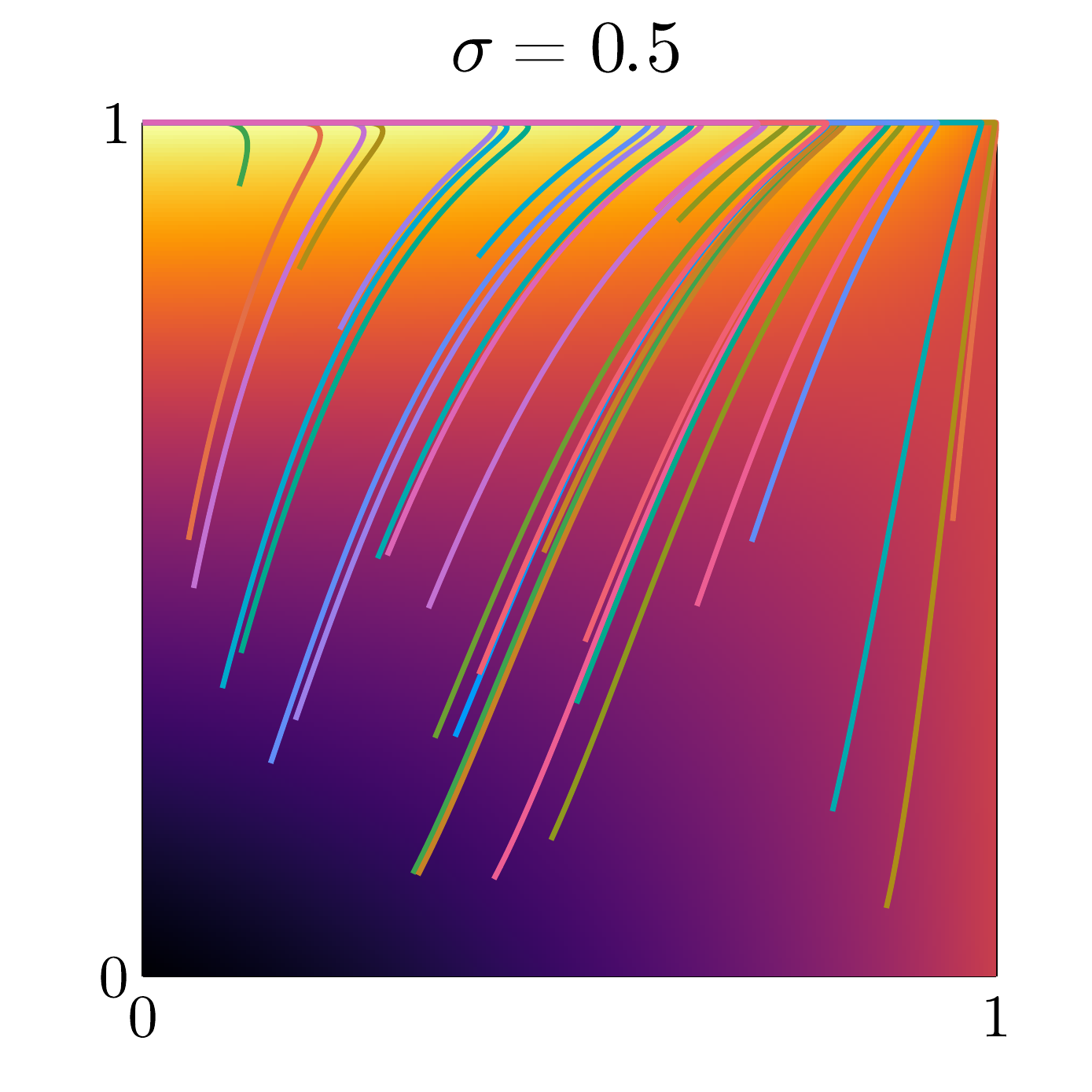}};
\node[inner sep=0pt] (r6) at (0,-3.3)
    {\includegraphics[width=3cm]{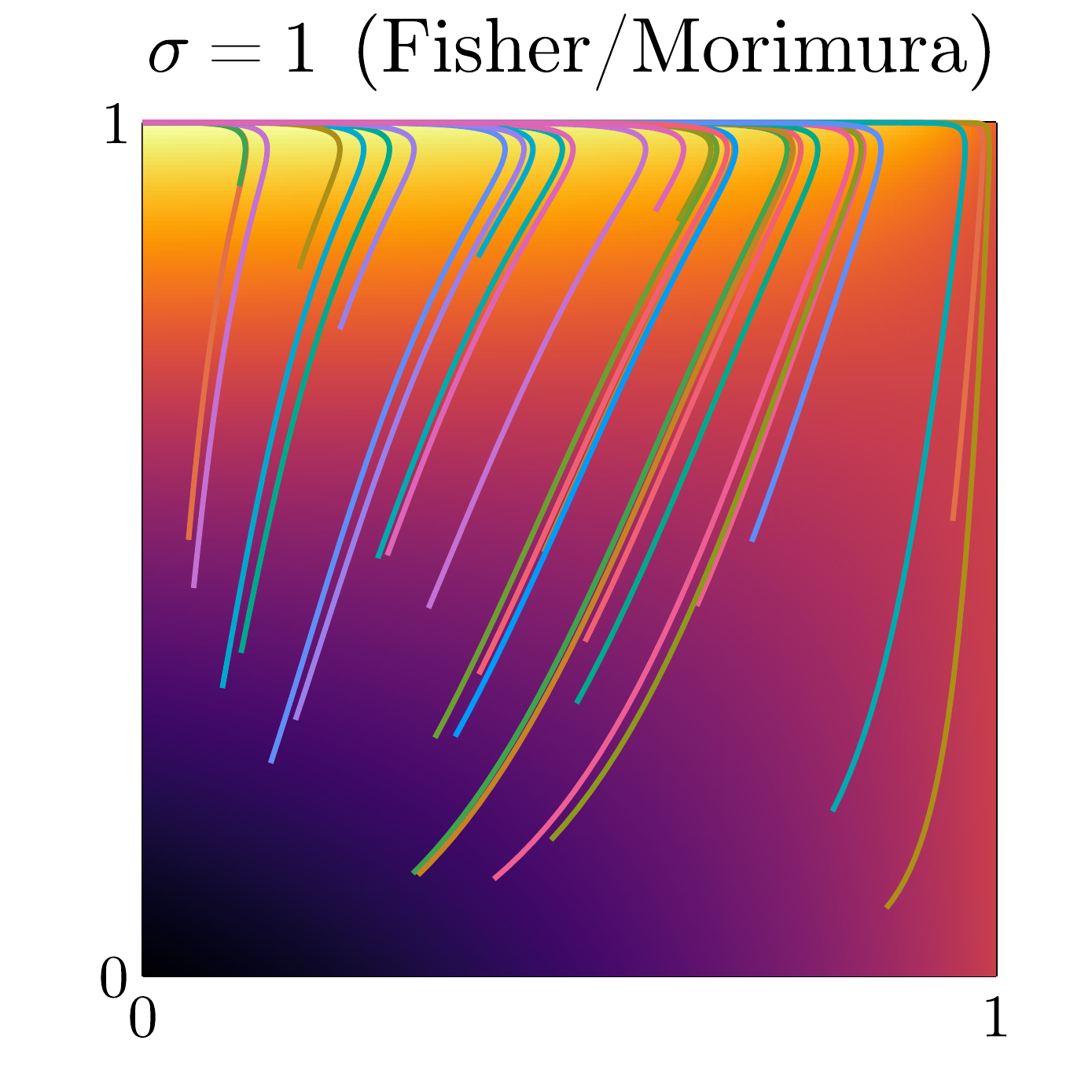}};
\node[inner sep=0pt] (r7) at (3.1,-3.3)
    {\includegraphics[width=3cm]{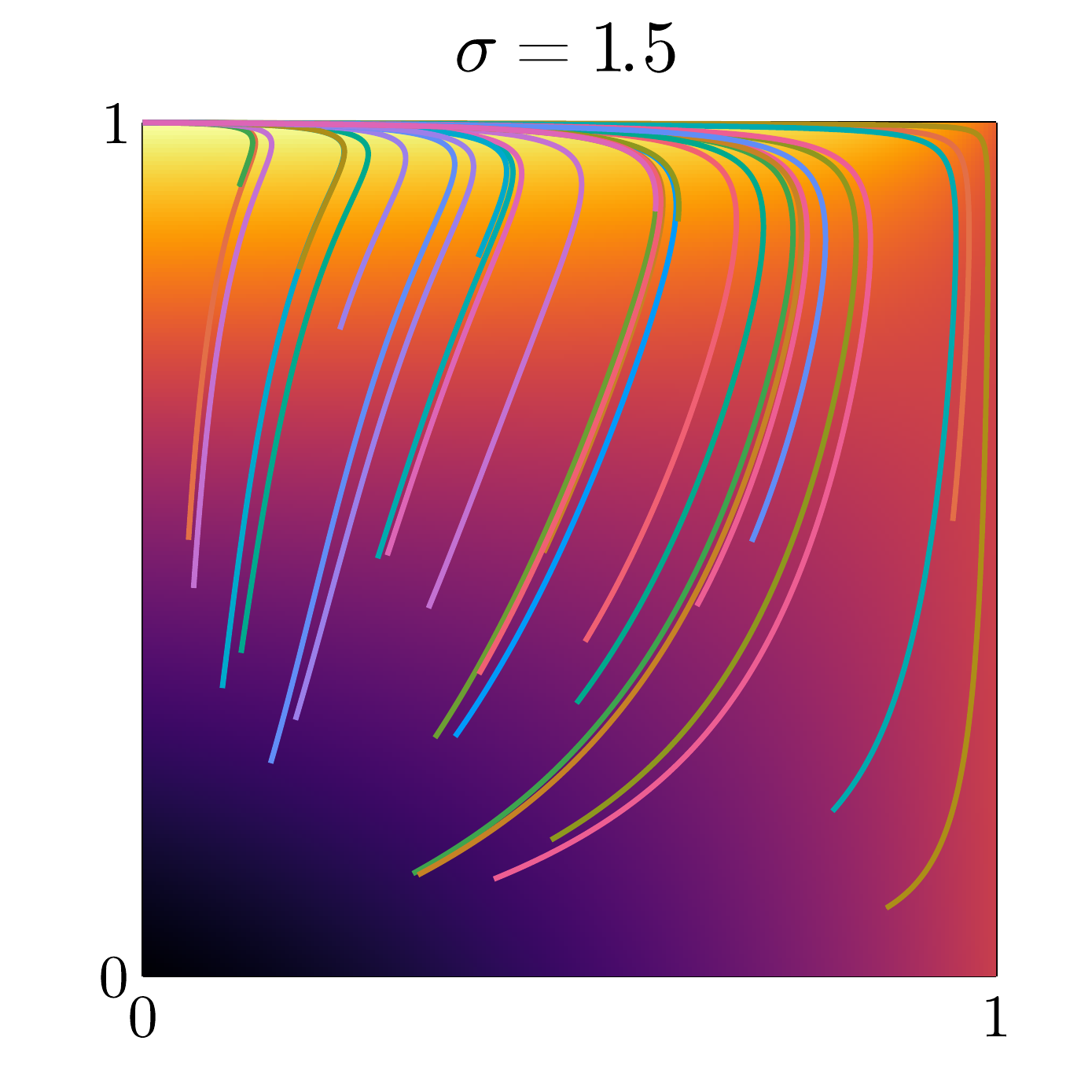}};
\node[inner sep=0pt] (r8) at (6.2,-3.3)
    {\includegraphics[width=3cm]{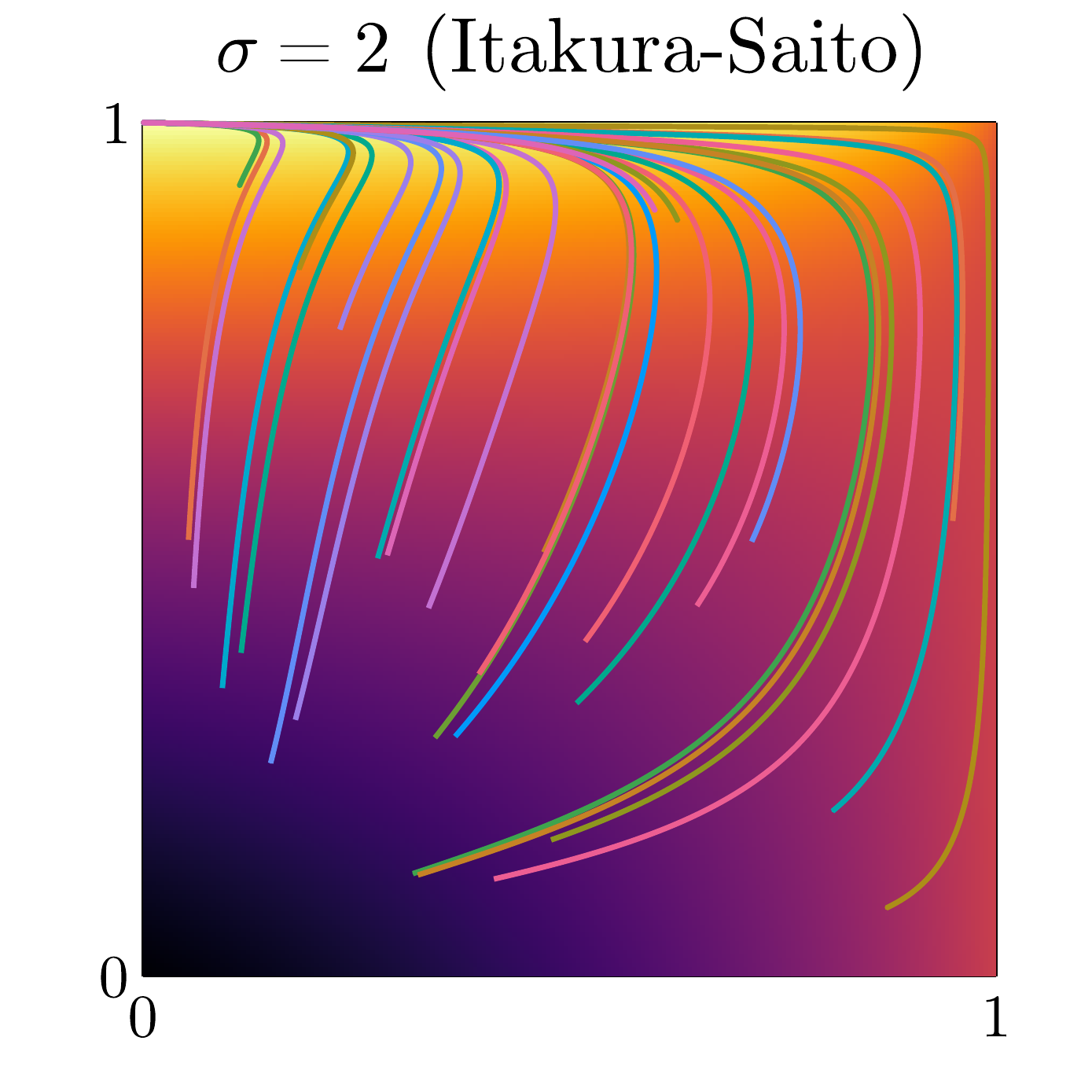}};
\node[inner sep=0pt] (r9) at (9.3,-3.3)
    {\includegraphics[width=3cm]{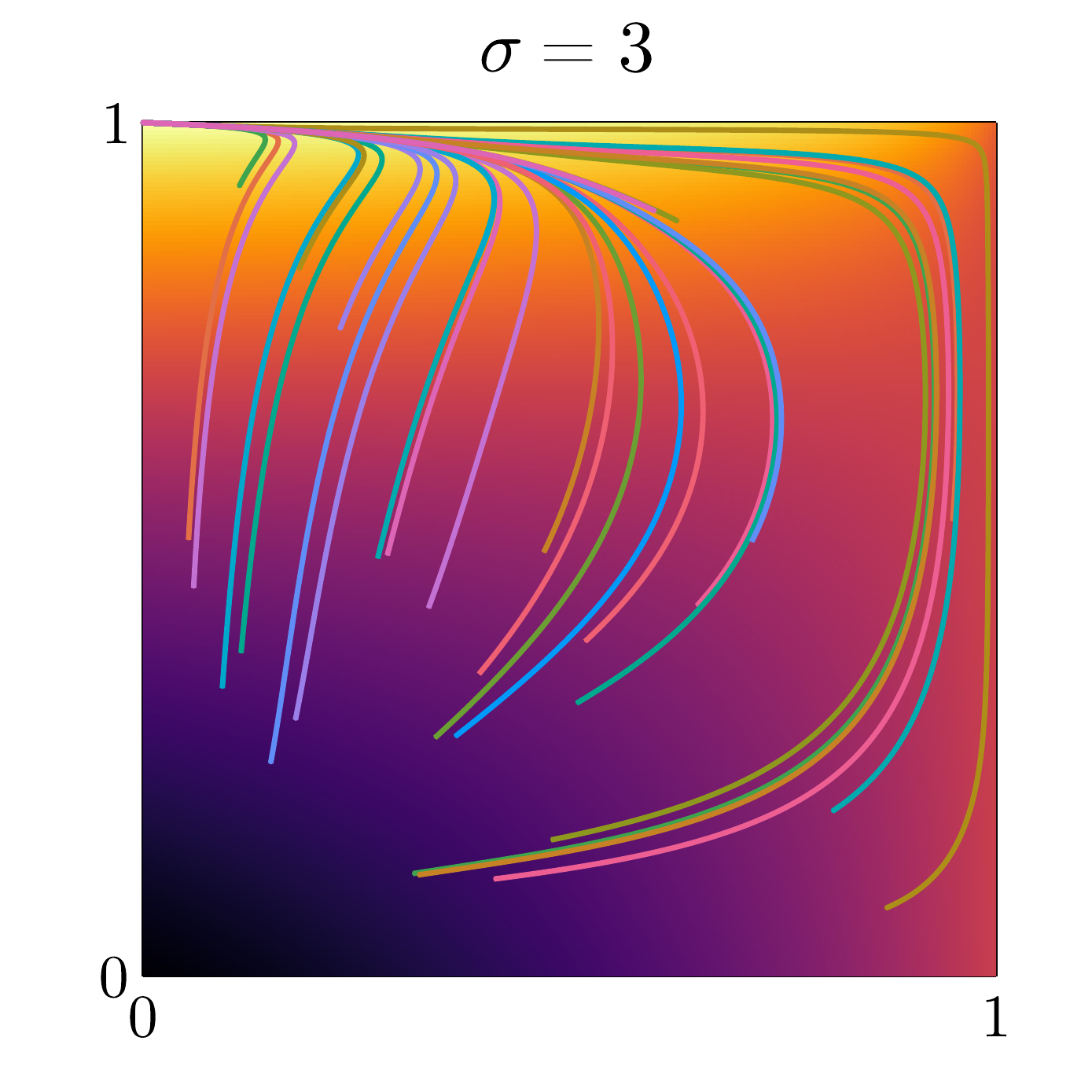}};
\node[inner sep=0pt] (r10) at (12.4,-3.3)
    {\includegraphics[width=3cm]{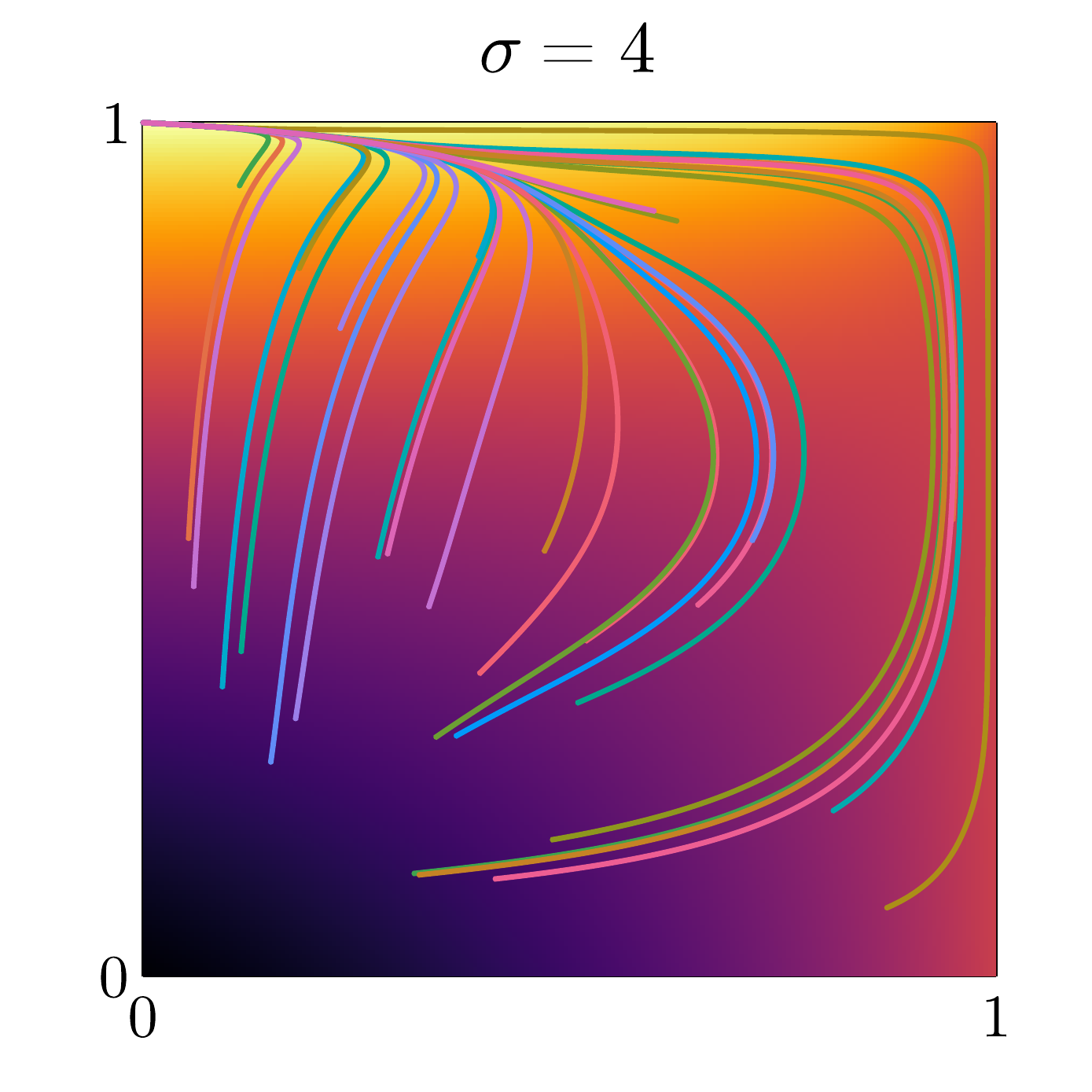}};
\node[inner sep=0pt] (colorbar) at (14.4,-1.7)
    {\includegraphics[height=5.6cm]{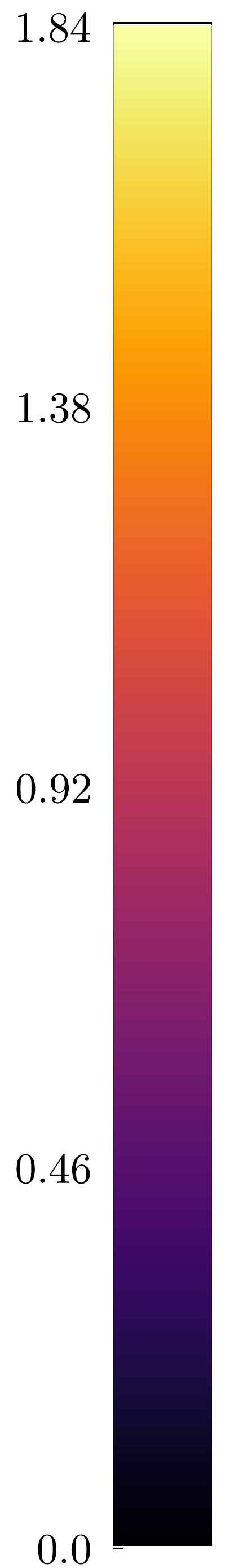}};
\node[inner sep=0pt] (label1) at (-1.5,0) {\rotatebox{90}{\tiny$\pi(a_1|s_2)$}};
\node[inner sep=0pt] (label2) at (-1.5,-3.3) {\rotatebox{90}{\tiny$\pi(a_1|s_2)$}};
\node[inner sep=0pt] (label3) at (0,-4.8) {\tiny$\pi(a_1|s_1)$};
\node[inner sep=0pt] (label3) at (3.1,-4.8) {\tiny$\pi(a_1|s_1)$};
\node[inner sep=0pt] (label3) at (6.2,-4.8) {\tiny$\pi(a_1|s_1)$};
\node[inner sep=0pt] (label3) at (9.3,-4.8) {\tiny$\pi(a_1|s_1)$};
\node[inner sep=0pt] (label3) at (12.4,-4.8) {\tiny$\pi(a_1|s_1)$};
\end{tikzpicture}
}
\caption{Plots of the trajectories of the individual methods inside the policy polytope $\Delta_\AA^\SS\cong[0,1]^2$; additionally, a heatmap of the reward function $\pi\mapsto R(\pi)$ is shown; 
the maximizer $\pi^\ast$ is located at the upper left corner of the policy polytope. 
}
\label{fig:policyTrajectories}
\end{figure}
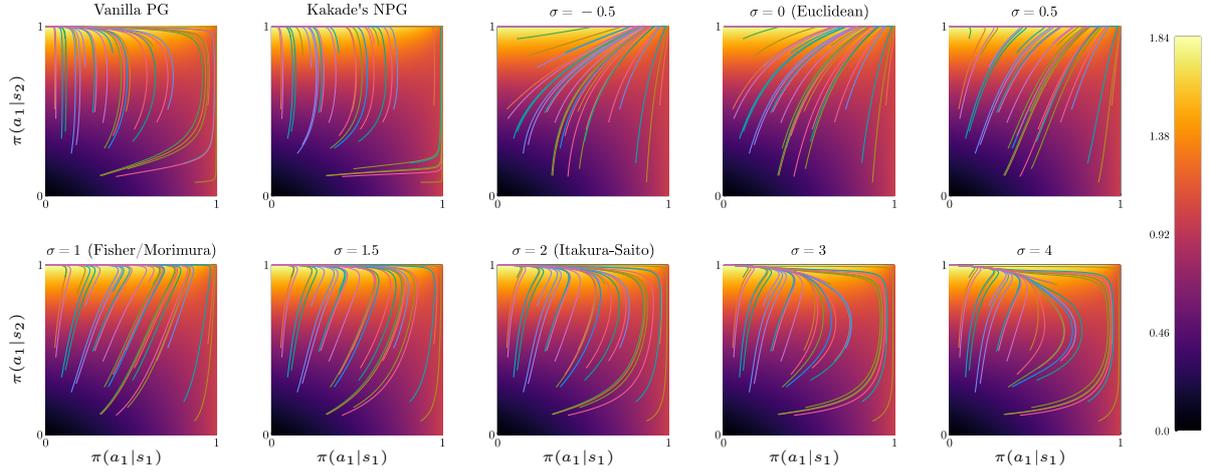

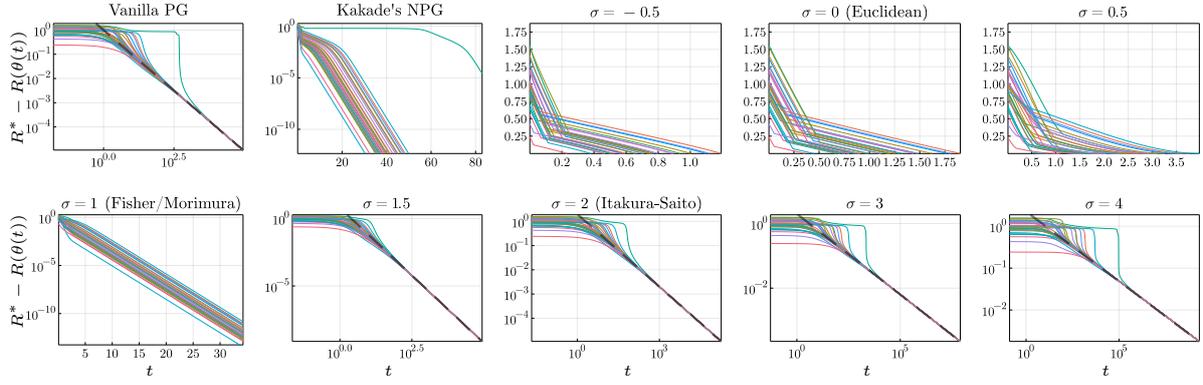
\begin{figure}[ht]
\centering
\resizebox{1\textwidth}{!}{
\begin{tikzpicture}[]
\node[inner sep=0pt] (r1) at (0,0)
    {\includegraphics[width=3cm]{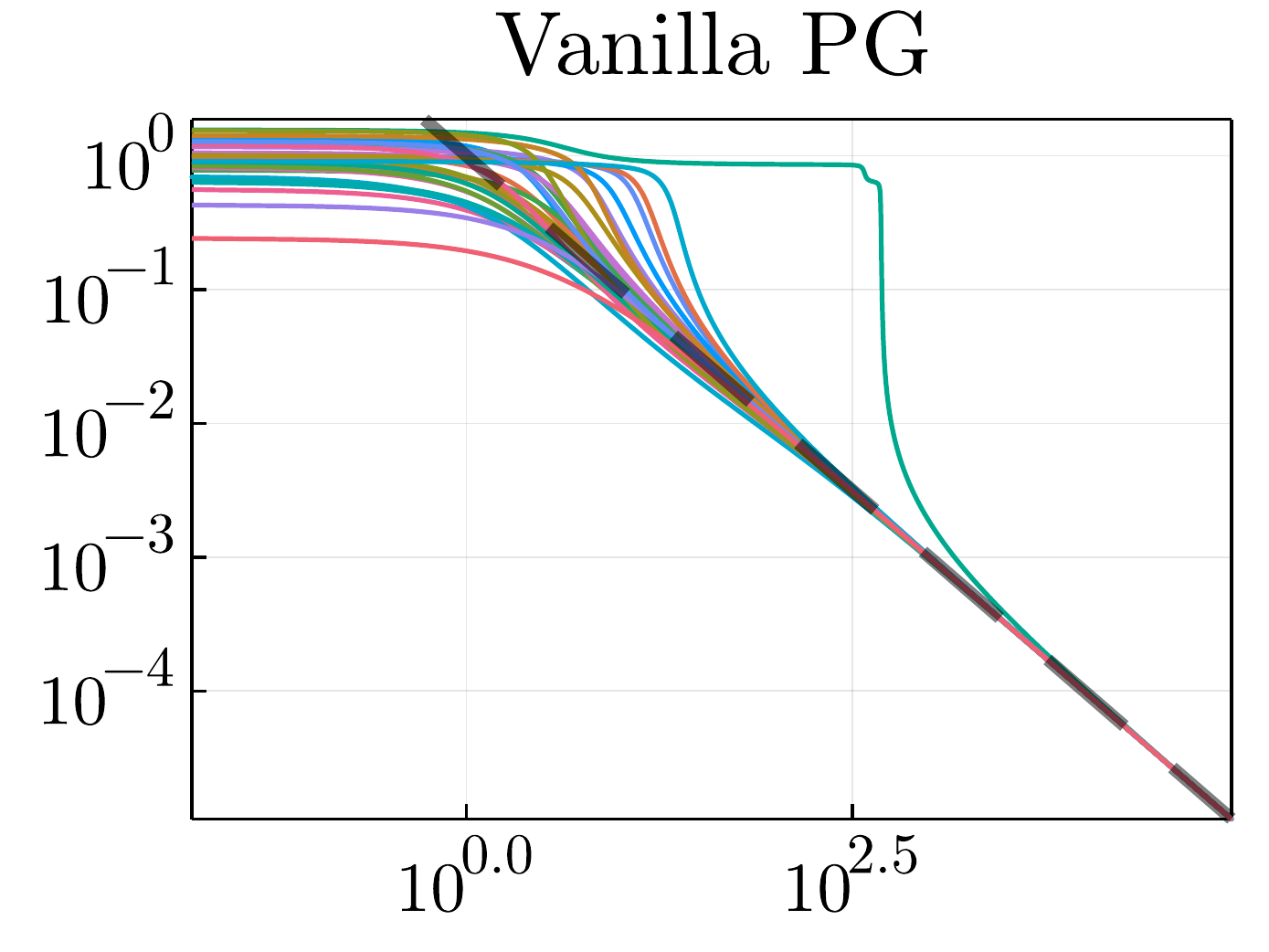}};
\node[inner sep=0pt] (r2) at (3.1,0)
    {\includegraphics[width=3cm]{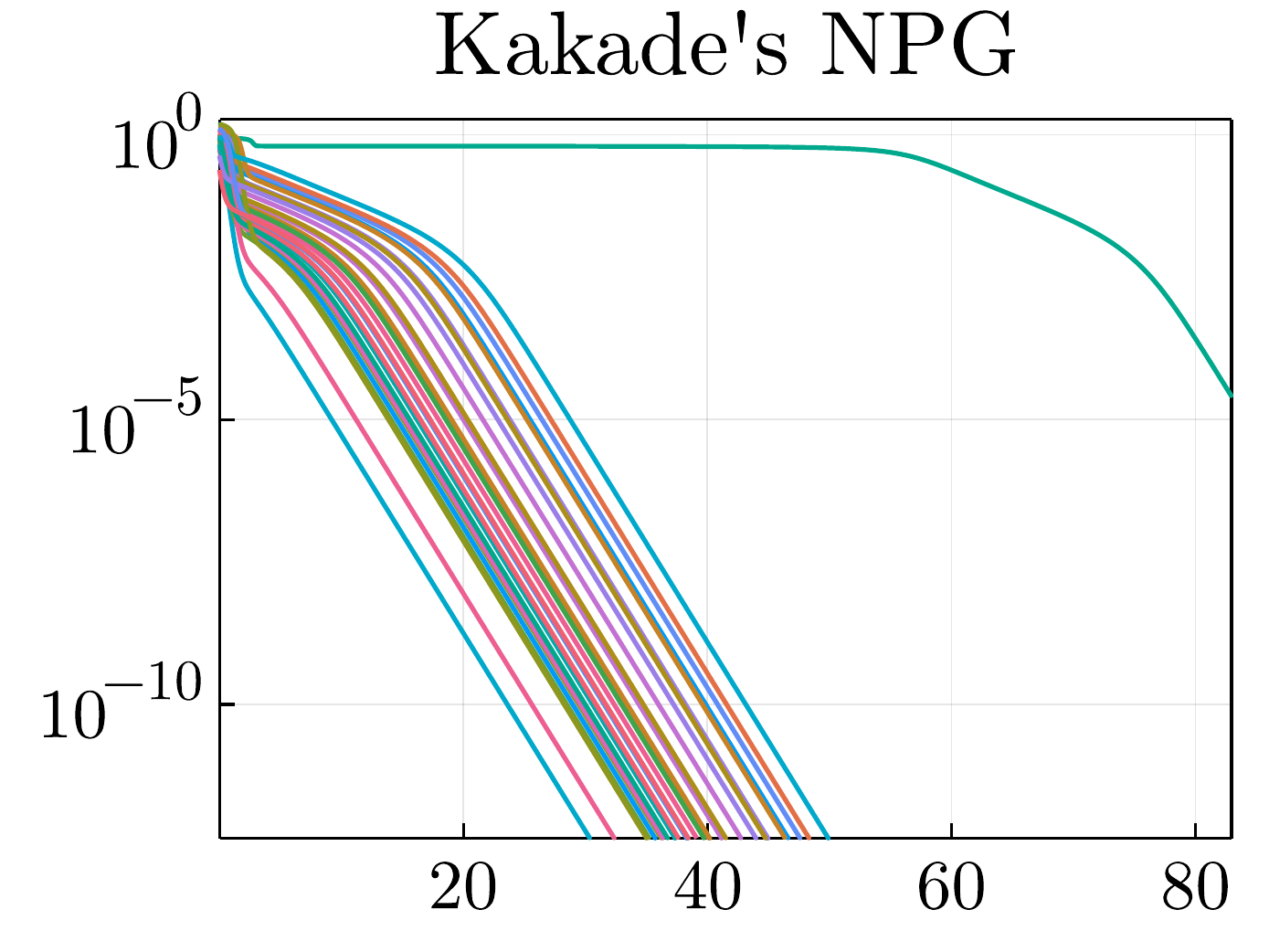}};
\node[inner sep=0pt] (r3) at (6.2,0)
    {\includegraphics[width=3cm]{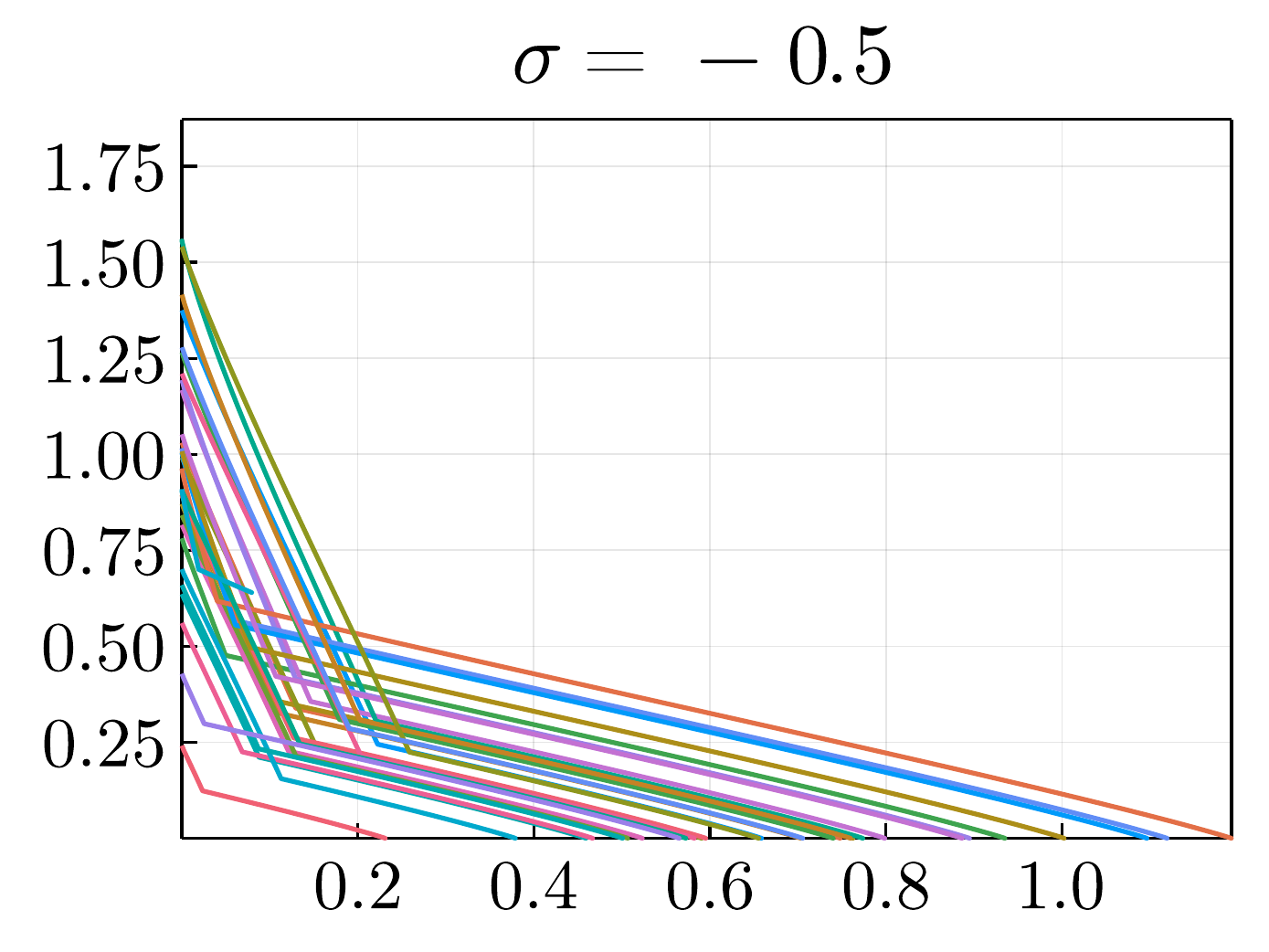}};
\node[inner sep=0pt] (r4) at (9.3,0)
    {\includegraphics[width=3cm]{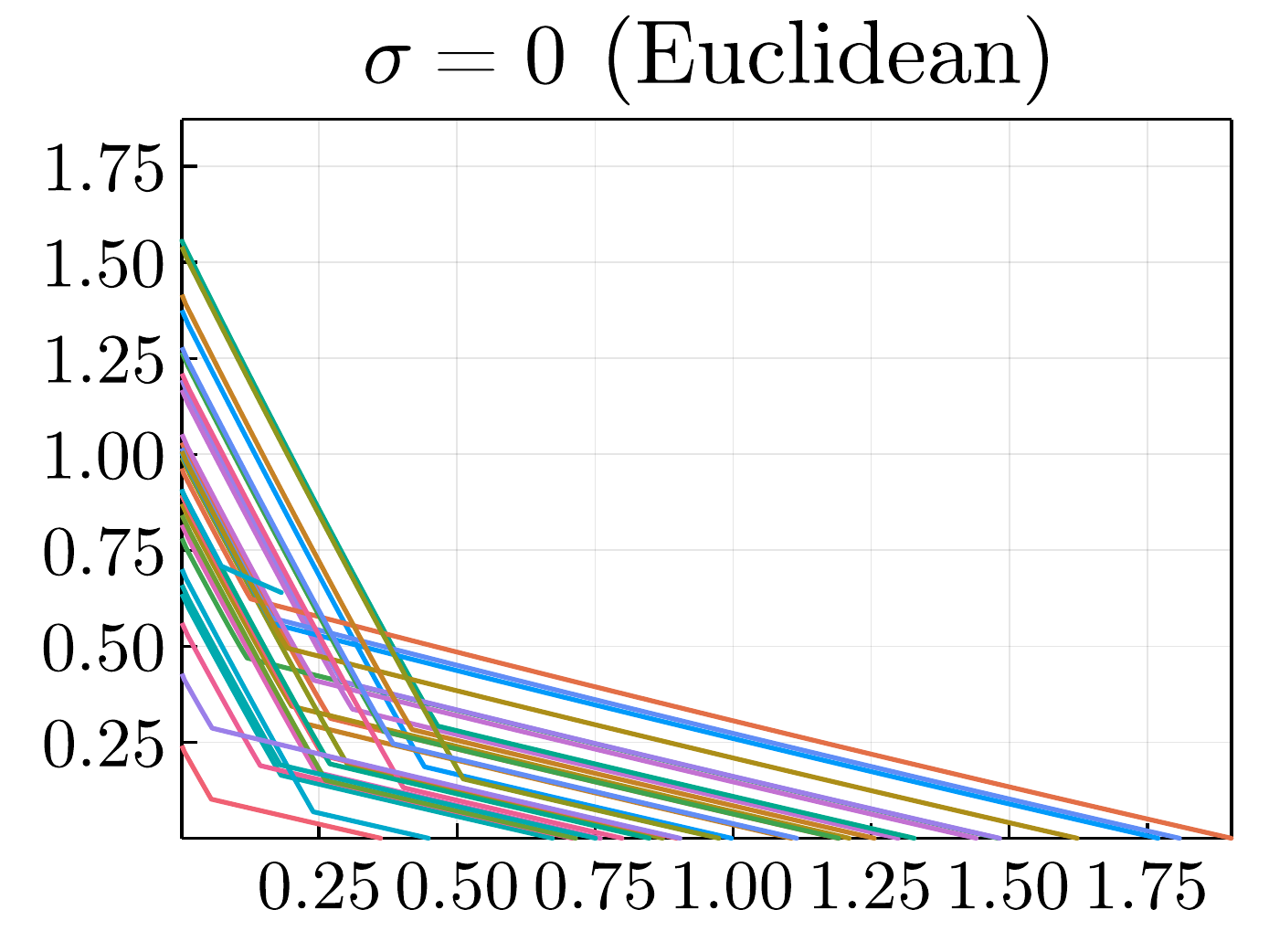}};
\node[inner sep=0pt] (r5) at (12.4,0)
    {\includegraphics[width=3cm]{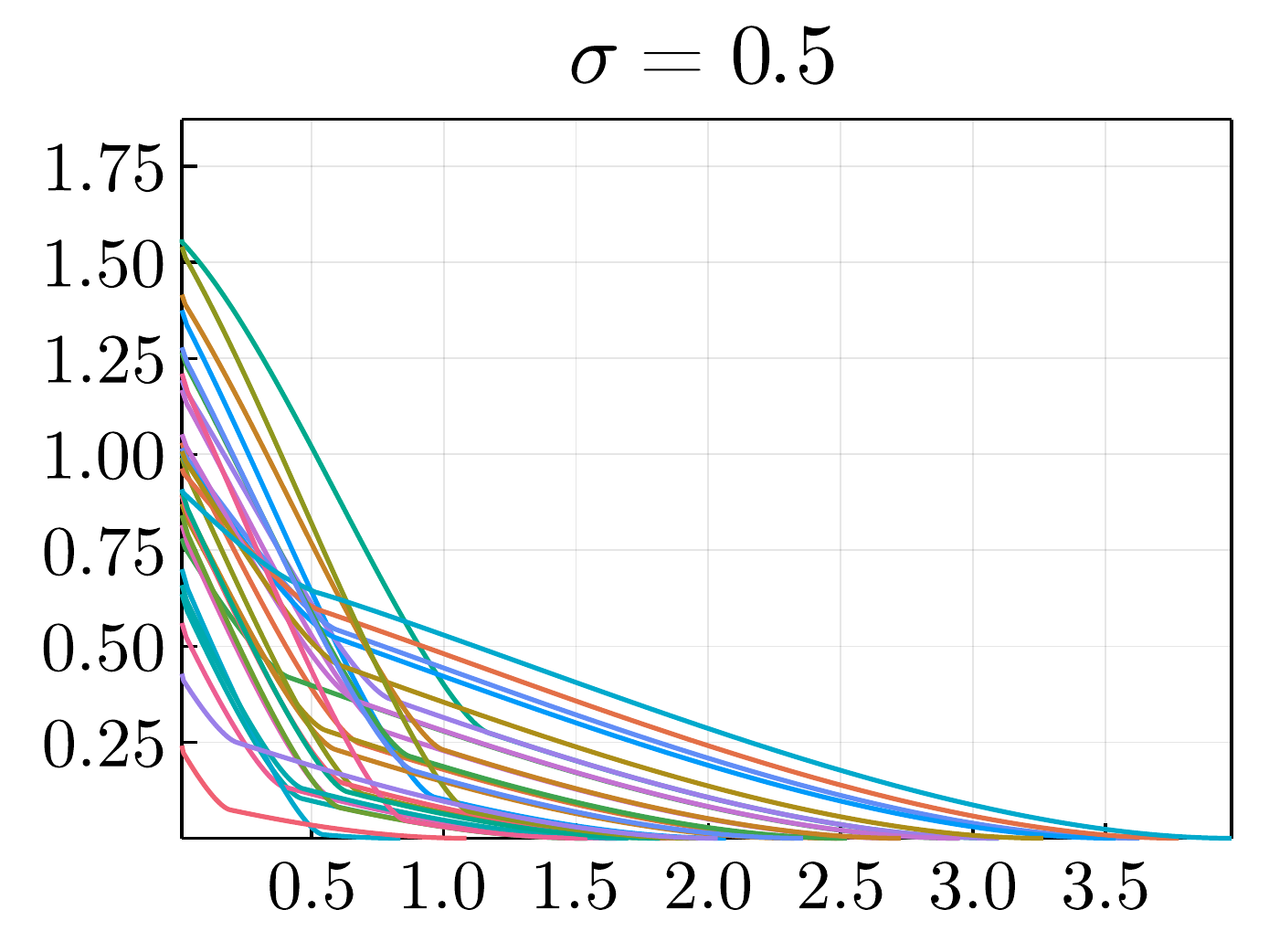}};
\node[inner sep=0pt] (r6) at (0,-2.5)
    {\includegraphics[width=3cm]{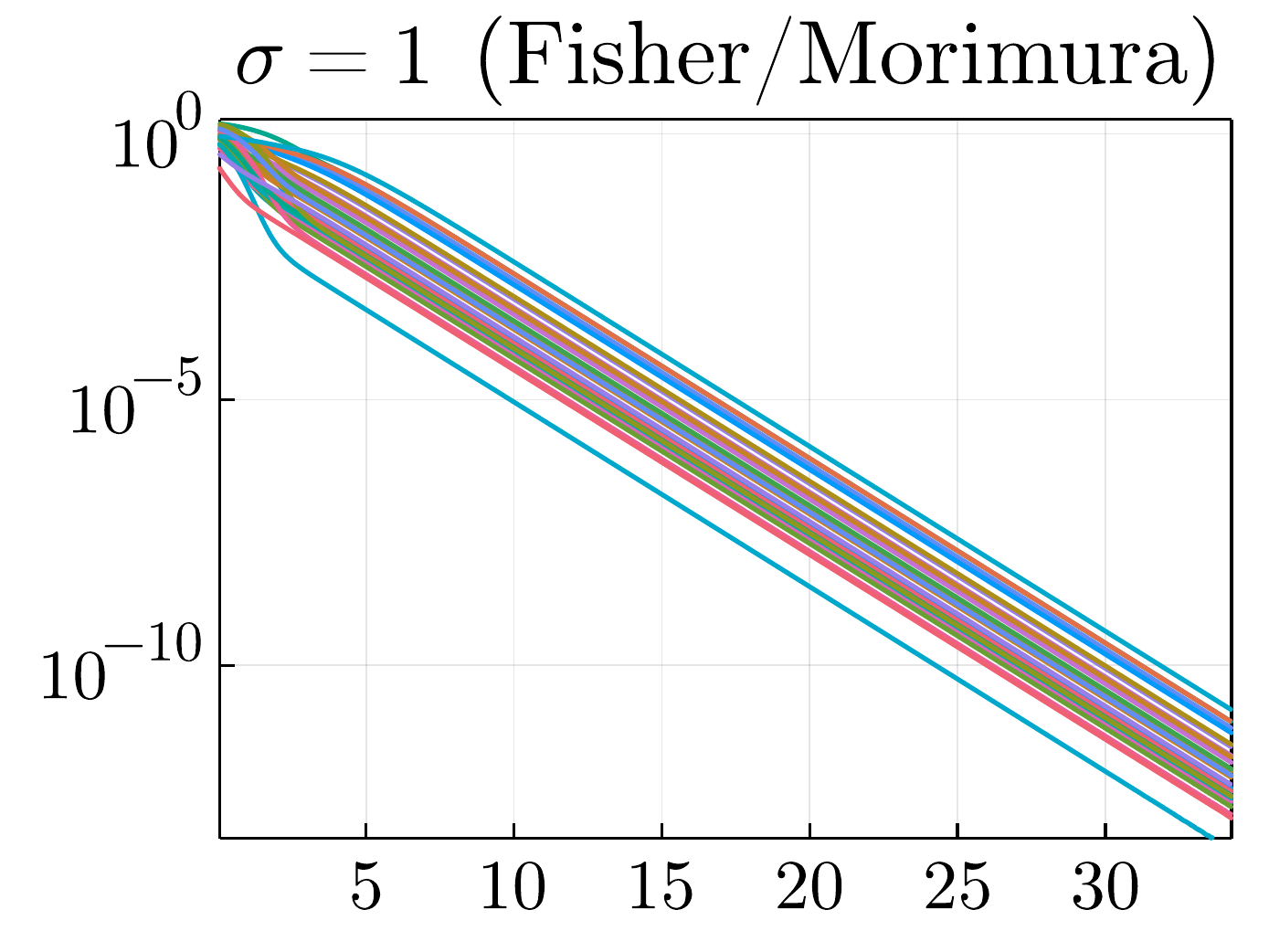}};
\node[inner sep=0pt] (r7) at (3.1,-2.5)
    {\includegraphics[width=3cm]{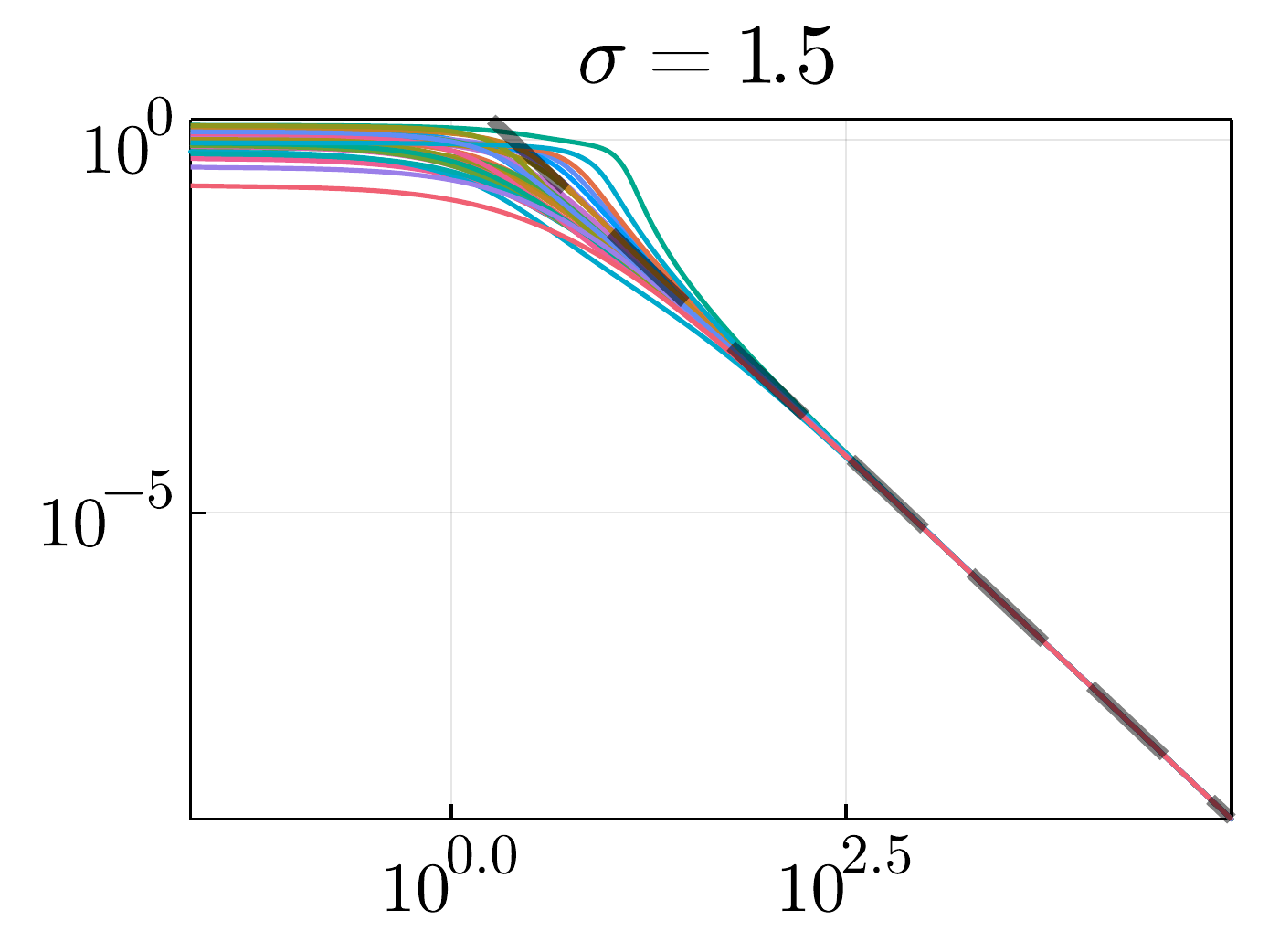}};
\node[inner sep=0pt] (r8) at (6.2,-2.5)
    {\includegraphics[width=3cm]{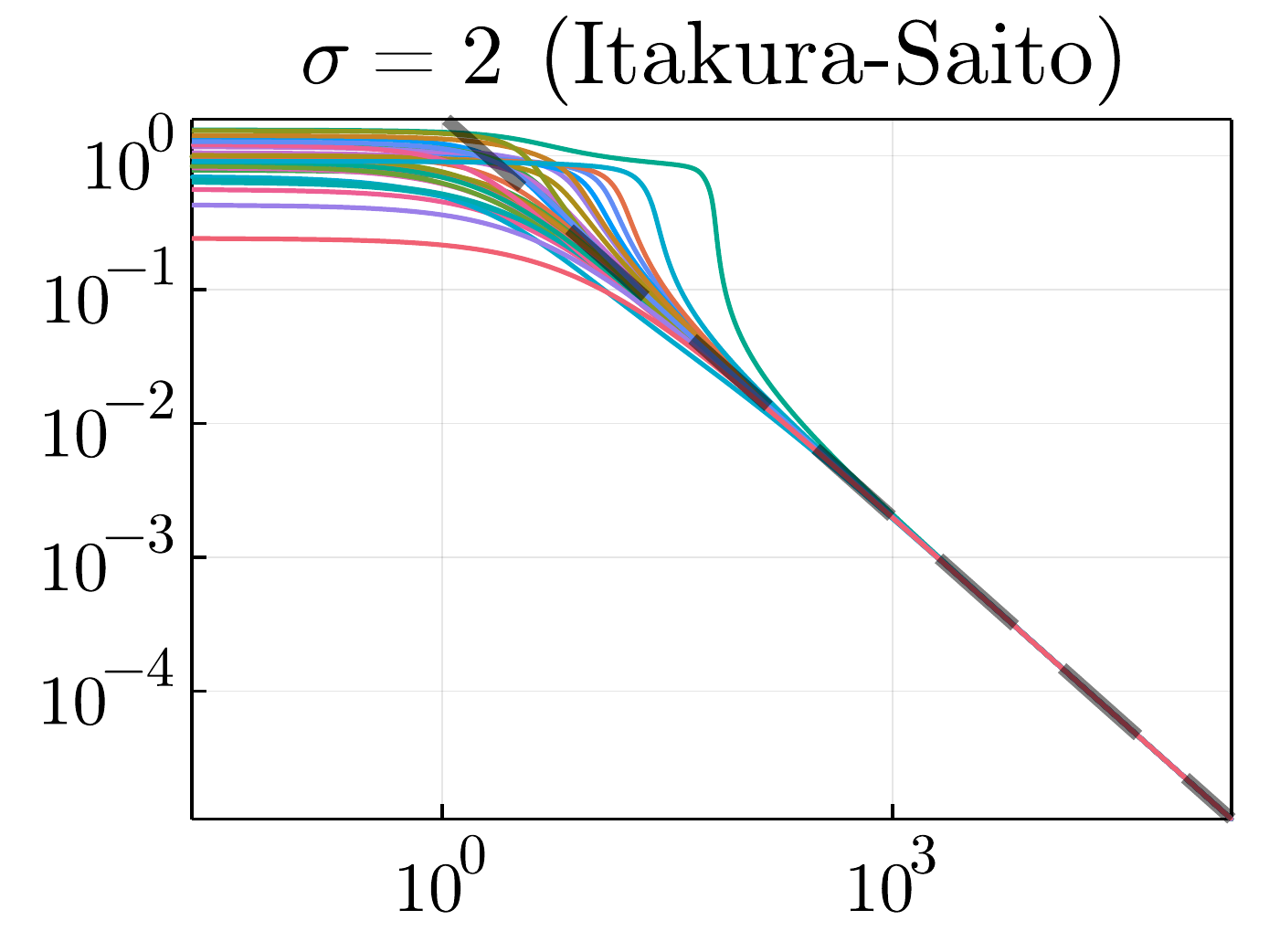}};
\node[inner sep=0pt] (r9) at (9.3,-2.5)
    {\includegraphics[width=3cm]{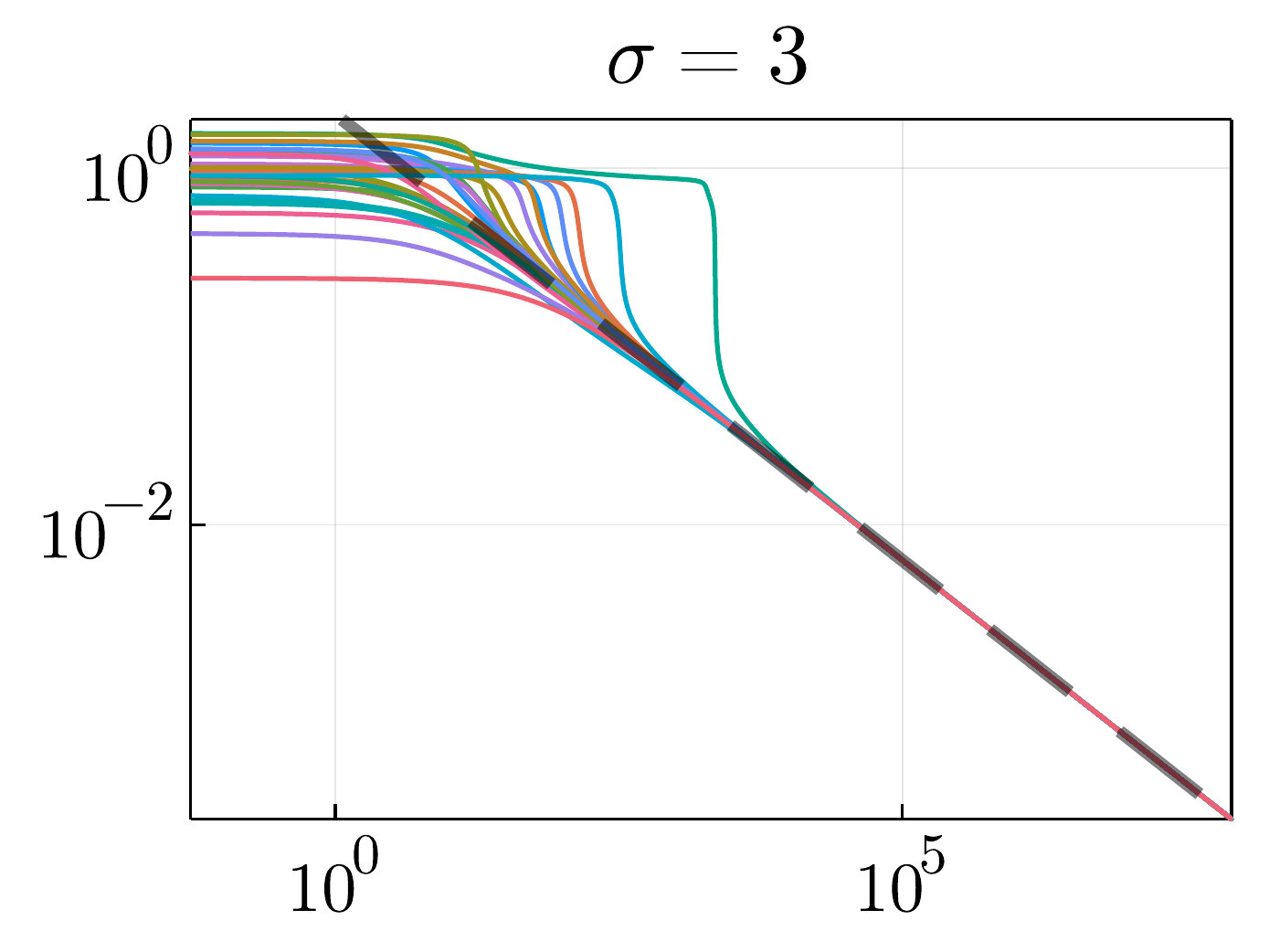}};
\node[inner sep=0pt] (r10) at (12.4,-2.5)
    {\includegraphics[width=3cm]{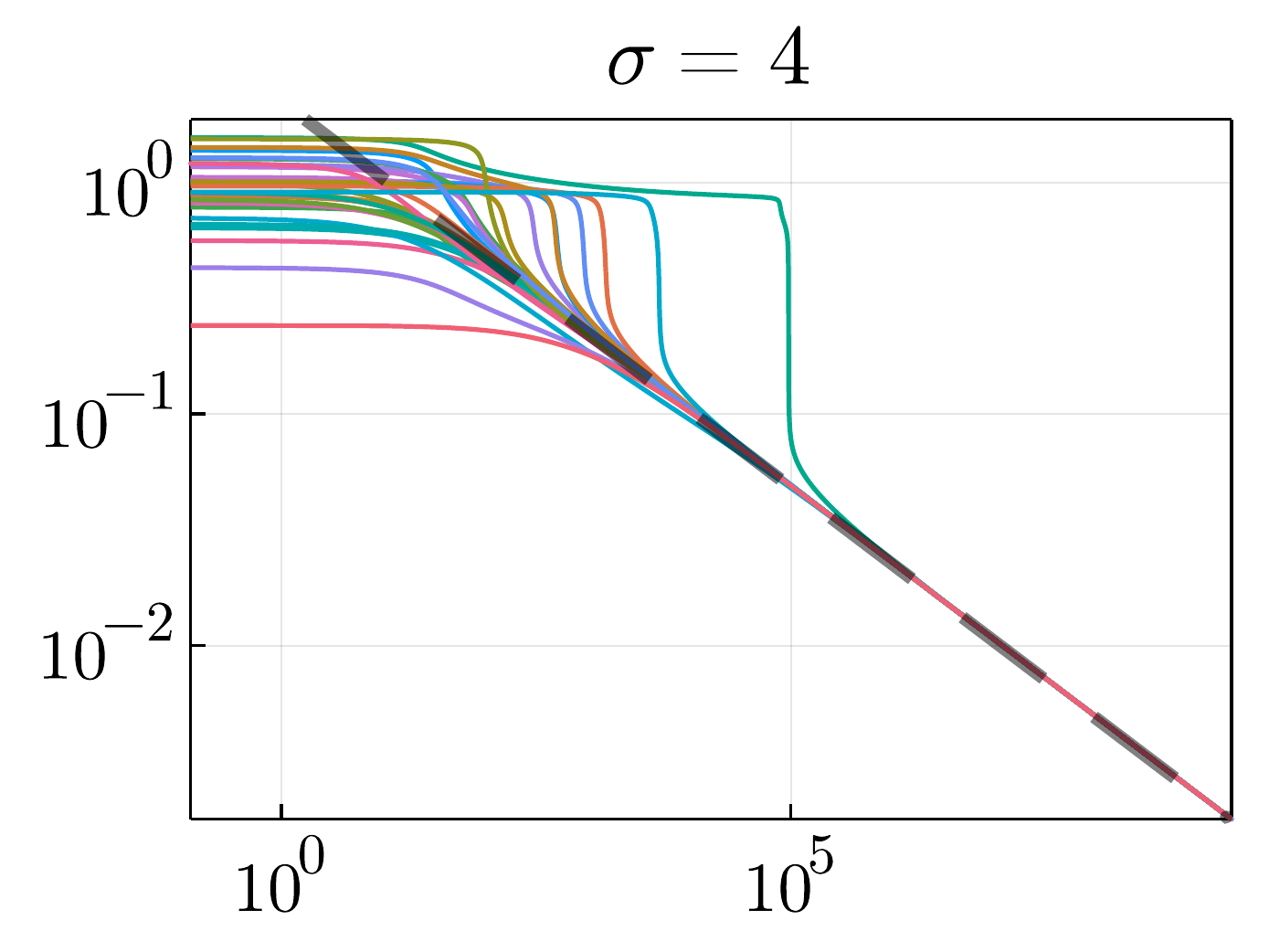}};
\node[inner sep=0pt] (label1) at (-1.5,0) {\rotatebox{90}{\tiny$R^\ast-R(\theta(t))$}};
\node[inner sep=0pt] (label2) at (-1.5,-2.4) {\rotatebox{90}{\tiny$R^\ast-R(\theta(t))$}};
\node[inner sep=0pt] (label3) at (0.2,-3.7) {\tiny$t$};
\node[inner sep=0pt] (label3) at (3.3,-3.7) {\tiny$t$};
\node[inner sep=0pt] (label3) at (6.4,-3.7) {\tiny$t$};
\node[inner sep=0pt] (label3) at (9.5,-3.7) {\tiny$t$};
\node[inner sep=0pt] (label3) at (12.6,-3.7) {\tiny$t$};
\end{tikzpicture}
}
\caption{Plot of the optimality gaps $R^\ast-R(\theta(t))$ during optimization; note that for vanilla PG and $\sigma>1$ these are log-log plots since we expect a decay like $t^{-1}$ and $t^{-1/(\sigma-1)}$ respectively, which are shown as a dashed gray line; Kakade's and Morimura's NPG are at a log plot since we expect a linear convergence; finally, for $\sigma<1$ we observe finite time convergence.}
\label{fig:convergenceRates}
\end{figure}

\subsection{Linear convergence of regularized Hessian natural policy gradient flows}

It is known that strictly convex regularization in state-action space can yield linear convergence in reward optimization for vanilla and Kakade's natural policy gradients~\cite{mei2020global, cen2021fast}. 
Using Lemma~\ref{prop:ratesTrajectories} we generalize the result for Kakade's NPG and provide a result giving the linear convergence for general Hessian NPG. 

\begin{theorem}[Linear convergence for regularized problems]\label{thm:inearConvRegularized}
Consider Setting~\ref{set:MDPconvergence} and let $\phi$ be a Legendre type function and denote the regularized reward by $\mathfrak R_\lambda(\eta) = \langle r, \eta\rangle - \lambda \phi(\eta)$ for some $\lambda>0$ and fix an $\eta_0\in\operatorname{int}(\mathcal N)$ and assume that the global maximizer $\eta^\ast_\lambda$ of $\mathfrak R_\lambda$ over $\mathcal N$ lies in the interior $\operatorname{int}(\mathcal N)$. 
Assume that $\eta\colon[0, \infty)\to\operatorname{int}(\mathcal N)$ solves the natural policy gradient flow with respect to the regularized reward $\mathfrak R_\lambda$ and the Hessian geometry induced by $\phi$. For any $c\in(0, \lambda)$ there exists a constant $K>0$ such that $D_\phi(\eta^\ast_\lambda, \eta(t)) \le K e^{-c t}$. In particular, for any $\kappa\in(\kappa_c, \infty)$ this implies $R^\ast_\lambda-\mathfrak R_\lambda(\eta(t))\le \kappa\lambda K e^{-c t}$, where $\kappa_c$ denotes the condition number of $\nabla^2\phi(\eta^\ast)$.
\end{theorem}
\begin{proof}
We first recall that by Lemma~\ref{prop:convergenceGeneral} it holds that $\mathfrak R(\eta(t))\to\mathfrak R(\eta^\ast)$ and the uniqueness of the maximizer $\eta(t)\to \eta^\ast\in\operatorname{int}(\mathcal N)$. 
By Lemma~\ref{prop:ratesTrajectories} it suffices to show that for any $\omega\in(0,1)$ it holds $\mathfrak R_\lambda(\eta^\ast) - \mathfrak R_\lambda(\eta) \ge \omega D_\phi(\eta^\ast,\eta)$ if $\eta$ in a neighborhood of $\eta^\ast$. 
Note that 
    \[ D_\phi(\eta^\ast, \eta) = \lambda^{-1} D_{\lambda \phi}(\eta^\ast, \eta) = D_{-\mathfrak R_\lambda}(\eta^\ast, \eta). \]
By Lemma~\ref{lem:Bregman} it follows that 
\begin{equation}\label{eq:coupling}
    \mathfrak R_\lambda(\eta^\ast) - \mathfrak R_\lambda(\eta) \ge \omega D_{-\mathfrak R_\lambda}(\eta^\ast,\eta) = \lambda \omega D_{\phi}(\eta^\ast,\eta),
\end{equation}
which shows the linear convergence of the trajectory in the Bregman divergence. 
For arbitrary $m, M>0$ such that $mI \prec \nabla^2\phi(\eta^\ast) \prec MI$ we can estimate
\begin{align*}
    R_\lambda^\ast - \mathfrak R_\lambda(\eta(t)) & = \mathfrak R_\lambda(\eta^\ast) - \mathfrak R_\lambda(\eta(t)) \le \frac{\lambda M}{2} \cdot \lVert \eta^\ast - \eta(t) \rVert^2 \le \frac{\lambda M}{m} \cdot D_\phi(\eta^\ast, \eta),
\end{align*}
for $\eta(t)$ close to $\eta^\ast$, where we used that $\phi$ is $m$ strongly convex in a neighborhood of $\eta^\ast$. 
\end{proof}

In the proof of the previous theorem we used the following lemma. 
\begin{lemma}\label{lem:Bregman}
Let $\phi$ be a strictly convex function defined on an open convex set $\Omega\subseteq\mathbb R^d$ with unique minimizer $\eta^\ast$. Then for any $\omega\in(0,1)$ there is a neighborhood $N_\omega$ of $x^\ast$ such that 
    \[ \phi(x) - \phi(x^\ast) \ge \omega D_\phi(x^\ast, x) \quad \text{for all } x\in N_\omega. \]
\end{lemma}
\begin{proof}
Set $f(x)\coloneqq D_\phi(x^\ast, x)$ and $g(x)\coloneqq D_\phi(x, x^\ast)$. It holds that $f(x^\ast) = g(x^\ast) = 0$ and since both functions are non-negative $\nabla f(x^\ast) = \nabla g(x^\ast) = 0$, which implies $g(x) = \phi(x) - \phi(x^\ast)$. By~\eqref{eq:hessianBregman} we have $\nabla^2 f(x^\ast) = \nabla^2 g(x^\ast) = \nabla^2\phi(x^\ast)$ and Taylor extension yields
\begin{align*}
    f(x) & = (x-x^\ast)^\top\nabla^2 \phi(x^\ast)(x-x^\ast) + o(\lVert x - x^\ast \rVert^2) \\
    & = g(x) + o(\lVert x - x^\ast\rVert^2) \\
    & = \phi(x) - \phi(x^\ast) + o(\lVert x - x^\ast \rVert^2).
\end{align*}
Hence, for any $\varepsilon>0$ there is $\delta>0$ such that for $x\in B_\delta(x^\ast)$ it holds that 
\begin{align*}
    f(x) & \le \phi(x) - \phi(x^\ast) + \varepsilon\lVert x - x^\ast \rVert^2 \le \left(1 + \frac{2\varepsilon}{m}\right)(\phi(x) - \phi(x^\ast))
\end{align*}
for any $m\in (0, \lambda_{min}(\nabla^2\phi(x^\ast))$ in a possible smaller neighborhood as $\phi$ is $m$-strongly convex in a neighborhood around $x^\ast$. Setting $\omega\coloneqq(1+2\varepsilon m^{-1})^{-1}$ yields the claim. 
\end{proof}

\begin{remark}[Location of maximizers]\label{rem:locMaximizers}
The condition that $\eta^\ast_\lambda\in\operatorname{int}(\mathcal N)$ assumed in Theorem~\ref{thm:inearConvRegularized} is satisfied if the gradient blow-up condition from Definition~\ref{def:LegendreType} is slightly strengthened. Indeed, suppose that for any $\eta\in\partial\mathcal N$ there is a direction $v$ such that $\eta+tv\in\operatorname{int}(\mathcal N)$ for small $t$ and such that $\partial_v \phi(\eta+tv) = v^\top\nabla\phi(\eta+tv)\to-\infty$ for $t\to0$. If $\phi(\eta) = \infty$, surely $\eta\ne\eta^\ast$. To argue in the case that $\phi(\eta)<+\infty$, we note that $\partial_v \mathfrak R_\lambda(\eta+tv)\to+\infty$ and choose $t_0>0$ such that $\partial_v \mathfrak R_\lambda(\eta+t_0v)>0$. 
Then by the concavity of $\mathfrak R_\lambda$ and continuity of $\mathfrak R_\lambda$ we have
    \[ 
    \mathfrak R_\lambda(\eta) \le \mathfrak R_\lambda(\eta+t_0v) - t_0 \partial_v \mathfrak R_\lambda(\eta+t_0v) < \mathfrak R_\lambda(\eta+t_0v) , 
    \]
and hence $\eta\ne\eta^\ast$. 
\end{remark}

Now we elaborate the consequences of this general convergence result given in Theorem~\ref{thm:inearConvRegularized} for Kakade and $\sigma$-NPG flows. 

\begin{corollary}[Linear convergence of regularized Kakade's NPG flow]\label{cor:kakadeLinearConvReg}
Assume that $\eta\colon[0, \infty)\to\operatorname{int}(\mathcal N)$ solves the natural policy gradient flow with respect to the regularized reward $\mathfrak R_\lambda$ and the Hessian geometry induced by $\phi$. 
For any $\omega\in(0, \lambda)$ there exists a constant $K>0$ such that $D_\phi(\eta^\ast, \eta(t)) \le K e^{-\omega t}$. In particular, for any $\kappa\in(\kappa_c, \infty)$ this implies $R^\ast_\lambda-\mathfrak R_\lambda(\eta(t))\le \kappa  K e^{-\omega t}$, where $\kappa_c$ denotes the condition number of $\nabla^2\phi_C(\eta^\ast)$.
\end{corollary}
\begin{proof}

We want to use Remark~\ref{rem:locMaximizers}. Recall that 
    \[ \phi_C(\eta) = H(\eta) - H(\rho) = \sum_{s,a} \eta(s,a)\log(\eta(s,a)) - \sum_{s} \rho(s)\log(\rho(s)), \]
where $\rho(s) = \sum_{a}  \eta(s,a)$ is the state marginal. Note that by Assumption~\ref{ass:positivity} it holds that $\rho(s)>0$. Hence, if $\eta\in\partial\mathcal N$ we can take any $v\in\mathbb R^{\SS\times\AA}$ such that $\eta_\varepsilon\coloneqq \eta+\varepsilon v\in\operatorname{int}(\mathcal N)$ for small $\varepsilon>0$. Writing $\rho_\varepsilon$ for the associated state marginal, we obtain
    \[ \partial_v \phi_C(\eta_\varepsilon) = \sum_{s,a} \log(\eta_\varepsilon(s,a)) + \lvert \SS\rvert(\lvert\AA\rvert-1) - \sum_{s}\log(\rho_\varepsilon(s)) \to - \infty \]
for $\varepsilon\to0$ since $\eta(s',a') = 0$ for some $s'\in\SS, a'\in\AA$ and $\rho_\varepsilon(s)\to\rho(s)>0$ for all $s\in\SS$. 
\end{proof}

\begin{corollary}[Linear convergence for regularized $\sigma$-NPG flow]\label{cor:sigmaLinearConvReg}
Consider Setting~\ref{set:MDPconvergence} with $\phi=\phi_\sigma$ for some $\sigma\in[1, \infty)$ and denote the regularized reward by $\mathfrak R_\lambda(\eta) = \langle r, \eta\rangle - \lambda \phi(\eta)$ and fix an element $\eta_0\in\operatorname{int}(\mathcal N)$. Assume that $\eta\colon[0, \infty)\to\operatorname{int}(\mathcal N)$ solves the natural policy gradient flow with respect to the regularized reward $\mathfrak R_\lambda$ and the Hessian geometry induced by $\phi$. For any $\omega\in(0, \lambda)$ there exists a constant $K>0$ such that $D_\phi(\eta^\ast, \eta(t)) \le K e^{-\omega t}$. In particular, for any $\kappa\in(\kappa(\eta^\ast)^\sigma, \infty)$ this implies $R^\ast_\lambda-\mathfrak R_\lambda(\eta(t))\le \kappa  K e^{-\omega t}$, where $\kappa(\eta^\ast) = \frac{\max \eta^\ast}{\min \eta^\ast}$
\end{corollary}
\begin{proof}
Again, we use Remark~\ref{rem:locMaximizers} it is straight forward to see that for the Legendre type functions $\phi_\sigma$ the unique maximizer $\eta^\ast$ of $\mathfrak R_\lambda$ lies in the interior of $\mathcal N$. Hence, it remains to compute the condition number, for which we note that $\nabla^2\phi_\sigma(\eta^\ast) = \operatorname{diag}(\eta^\ast)^{-\sigma}$, which yields the result.
\end{proof}

\begin{remark}[Extension to arbitrary regularizers]
The results above do not cover arbitrary combinations of Hessian geometries and regularizers. However, the proof of Theorem~\ref{thm:inearConvRegularized} can be adapted to this case, where the only part that requires  adjustments is~\eqref{eq:coupling} that couples the regularized reward to the Bregman divergence. In principle, this can be extended to the case of regularizers that are different from the function inducing the Hessian geometry. 
\end{remark}

\section{Locally quadratic convergence for regularized problems}\label{sec:convergenceDiscrete}

It is known that Kakade's NPG method and more generally quasi-Newton policy gradient methods with suitable regularization and step sizes converge at a locally quadratic rate~\cite{cen2021fast, li2021quasi}. 
Whereas these results regard the NPG method as an inexact Newton method in the parameter space, we regard it as an inexact Newton method in state-action space, which allows us to directly leverage results from the optimization literature and thus formulate relatively short proofs. 
Our result extends the locally quadratic convergence rate to general Hessian-NPG methods, which include in particular Kakade's and Morimura's NPG. Note that the result holds when the step size is equal to the penalization strength, which is reminiscent of Newton's method converging for step size $1$. 

\begin{theorem}[Locally quadratic convergence of regularized NPG methods]
\label{thm:localQuadraticConvergence} 
Consider a real-valued function $\phi\colon\mathbb R^{\SS\times\AA}\to\mathbb R\cup\{+\infty\}$, which we assume to be finite and twice continuously differentiable on $\mathbb R_{>0}^{\mathcal S\times\mathcal A}$ and such that $\nabla^2\phi(\eta)$ is positive definite on $T_\eta\mathcal N = T\mathcal L\subseteq\mathbb R^{\mathcal S\times\mathcal A}$ for every $\eta\in \operatorname{int}(\mathcal N)$.
Further, consider a regular policy parametrization 
and the regularized reward $R_\lambda(\theta) \coloneqq R(\theta) + \lambda\phi(\eta_\theta)$ and assume that $\eta^\ast\in\operatorname{int}(\mathcal N)$, i.e., the maximizer lies in the interior of the state-action polytope. Consider the NPG induced by the Hessian geometry of $\phi$, i.e., 
\[ 
\theta_{k+1} = \theta_k + \Delta tG(\theta_k)^{+}\nabla R_\lambda(\theta_k) , 
\] 
with step size $\Delta t = \lambda$, where $G(\theta_k)^+$ denotes the Moore-Penrose inverse. 
Assume that $R_\lambda(\theta_k)\to R_\lambda^\ast$ for $k\to\infty$. 
Then $\theta_k\to\theta^\ast$ at a (locally) quadratic rate and hence $R_\lambda(\theta_k)\to R^\ast_\lambda$ at a (locally) quadratic rate. 
\end{theorem}

The proof of this result relies on the following convergence result for inexact Newton methods. 

\begin{theorem}[Theorem 3.3 in~\cite{dembo1982inexact}]
Consider an objective function $f\in C^2(\mathbb R^d)$ with $\nabla^2f(x)\in\mathbb S^{sym}_{>0}$ for any $x\in\mathbb R^d$ and assume that $f$ admits a minimizer $x^\ast$. 
Let $(x_k)$ be inexact Newton iterates given by 
\[ 
x_{k+1} = x_k + \nabla^2 f(x_k)^{-1}\nabla f(x_k) + \varepsilon_k , 
\]
and assume that they converge towards the minimum $x^\ast$. If $\lVert \varepsilon_k \rVert = O(\lVert \nabla f(x_k)\rVert^{\omega})$, then $x_k\to x^\ast$ at rate $\omega$, i.e., $\lVert x_k-x^\ast \rVert = O(e^{-k^{\omega}})$. 
\end{theorem}

We take this approach and show that the iterates of the regularized NPG method can be interpreted as an inexact Newton method in state-action space. 
For this, we first make the form of the Newton updates in state-action space explicit. 

\begin{lemma}[Newton iteration in state-action space]
The iterates of Newton's method in state-action space are given by
\begin{equation}
    \eta_{k+1} = \eta_k + \Pi_{T\mathcal L}^{E}(\nabla^2 \mathfrak R_\lambda(\eta_k))^{-1} \Pi_{T\mathcal L}^{E}(\nabla \mathfrak R_\lambda(\eta_k)),
\end{equation}
where $\mathfrak R_\lambda(\eta) = \langle r, \eta\rangle + \lambda \phi(\eta)$ is the regularized reward and $\Pi_{T\mathcal L}^{E}$ the Euclidean projection onto the tangent space of the affine space $L$ defined in~\eqref{eq:linearSpace}.
\end{lemma}
\begin{proof}
The domain of the optimization problem is $\mathbb R^{\SS\times\AA}_{\ge0}\cap\mathcal L$ an hence, we perform Newton's method on the affine subspace $L$. Writing $L = \eta_0 + X$ for a linear subspace $X$ we can equivalently perform Newton's method on $X$ since the method is affine invariant. We denote the canonical $\iota\colon X\hookrightarrow L, x\mapsto x+\eta_0$ and set $f(x)\coloneqq \mathfrak R_\lambda(\iota x)$. Then, we obtain the Newton iterates $x_k$ and $\eta_k = \iota x_k$ by 
\[ x_{k+1} = x_k + \nabla^2f(x_k)^{-1} \nabla f(x_k). \]
Straight up computation yields $\nabla f(x) \iota^\top\nabla \mathfrak R_\lambda(\iota x)$ and $\nabla^2f(x) = \iota^\top \nabla^2\mathfrak R_\lambda(\iota x) \iota$. Hence, we obtain
\begin{align*}
    \eta_{k+1} - \eta_k & = \iota \nabla^2f(x_k)^{-1} \nabla f(x_k) = \iota \iota^+ \nabla^2 \mathfrak R_\lambda(\eta_k)^{-1}(\iota^\top)^+\iota^\top \nabla \mathfrak R_\lambda(\eta_k) \\
    & = \Pi_{T\mathcal L}^{E}(\nabla^2 \mathfrak R_\lambda(\eta_k))^{-1} \Pi_{T\mathcal L}^{E}(\nabla \mathfrak R_\lambda(\eta_k)),
\end{align*}
where we used $AA^+ = \Pi_{\operatorname{range}(A)}$ and $(A^\top)^+A^\top = \Pi_{\operatorname{ker}(A^\top)} = \Pi_{\operatorname{range}(A)}$. 
\end{proof}

\begin{lemma}
\label{lem:NPGasInexactNewton}
Let $(\theta_k)$ be the iterates of a Hessian NPG induced by a stricly convex function $\phi$ and with step size $\Delta t$, i.e,
\[ \theta_{k+1} = \theta_k + \Delta t\cdot G(\theta_k)^{+}\nabla R_\lambda(\theta_k),  \]
where the Gram matrix is given by $G(\theta) = DP(\theta)^\top \nabla^2 \phi(\eta_\theta) DP(\theta)$. Then the state-action iterates $\eta_k\coloneqq \eta_{\theta_k}$ satisfy
\begin{equation}
    \eta_{k+1} = \eta_k + \Delta t \cdot \Pi_{T\mathcal L}^{E}(\nabla^2 \phi(\eta_k)^{-1} \Pi_{T\mathcal L}^{E}(\nabla \mathfrak R_\lambda(\eta_k))) + O(\Delta t^2\lVert G(\theta_k)^+ \nabla R_\lambda(\theta_k) \rVert^2).
\end{equation}
\end{lemma}
\begin{proof}
Writing $P$ for the mapping $\theta\mapsto\eta_\theta$ and an application of Taylor's theorem implies that 
\begin{align*}
    \eta_{k+1} - \eta_k & = \Delta t \cdot  DP(\theta_k)G(\theta_k)^+\nabla R_\lambda(\theta_k) +  O(\Delta t^2\lVert G(\theta_k)^+ \nabla R_\lambda(\theta_k) \rVert^2).
\end{align*}
The first term is equal to
\[ \Delta t \cdot DP(\theta_k)DP(\theta)^+\nabla^2 \phi(\eta_k)^{-1}(DP(\theta_k)^\top)^+\nabla DP(\theta_k)^\top \nabla \mathfrak R_\lambda(\eta_k),  \]
which again is equal to
\[ \Delta t \cdot \Pi_{T\mathcal L}^{E}(\nabla^2 \phi(\eta_k)^{-1} \Pi_{T\mathcal L}^{E}(\nabla \mathfrak R_\lambda(\eta_k))) \]
since $DP(\theta_k)DP(\theta_k)^+  = (DP(\theta_k)^\top)^+DP(\theta_k)^\top = \Pi_{\operatorname{range}(DP(\theta_k))}$ like before and $\operatorname{range}(DP(\theta_k)) = T\mathcal L$.  
\end{proof}

\begin{proof}[Proof of Theorem~\ref{thm:localQuadraticConvergence}]
In our case, by the preceding two lemmata, we have
\begin{align*}
    \lVert \varepsilon_k \rVert = O(\Delta t^2 \lVert G(\theta_k)^+ \nabla R_\lambda(\theta_k) \rVert^2) = O(\lVert \Pi_{T\mathcal L}\nabla \mathfrak R_\lambda(\eta_k) \rVert^2) = O(\lVert \nabla f(x_k) \rVert^2),
\end{align*}
which proves the claim. 
\end{proof}

\begin{remark}
A benefit of regarding the iteration as an inexact Newton method in state-action space is that the problem is strongly convex in state-action space. 
In contrast, in policy space the problem is non-convex, which makes the analysis in that space more delicate. Further, the corresponding Riemannian metric might not be the Hessian metric of the regularizer in policy space (see also Remark~\ref{rem:kakadePolicyspace}). 
In the parameter $\theta$, the NPG algorithm can be perceived as a generalized Gauss-Newton method; however, the reward function is non-convex in parameter space. 
Further, for overparametrized policy models, i.e., when $\operatorname{dim}(\Theta)>\operatorname{\dim}(\Delta_\mathcal A^\mathcal S) = \lvert\mathcal S\rvert(\lvert\mathcal A\rvert-1)$ the Hessian $\nabla^2 R(\theta^\ast)$ can not be positive definite, which makes the analysis in parameter space less immediate.
Note that the tabular softmax policies in \eqref{eq:soft-max} are overparametrized since in this case $\operatorname{dim}(\Theta) = \lvert\mathcal S\rvert \lvert\mathcal A\rvert$.
\end{remark}

\section{Discussion} 

We provide a study of a general class of natural policy gradient methods arising from Hessian geometries in state-action space. This covers, in particular, the notions of NPG due to Kakade and Morimura et al., which are induced by the conditional entropy and entropy respectively. 
Leveraging results on gradient flows in Hessian geometries we obtain global convergence guarantees of NPG flows for regular policy parametrizations and show that both Kakade's and Morimura's NPG converge linearly, and obtain sublinear convergence rates for NPG associated with $\beta$-divergences. We provide experimental evidence of the tightness of these rates. 
Finally, we perceive the NPG with respect to the Hessian geometry induced by the regularizer and with step size equal to the regularization strength, as an inexact Newton method in state-action space, which allows for a very compact argument of the locally quadratic convergence of this method. 

Our convergence analysis currently does not cover the case of general parametric policy classes nor the case of partially observable MDPs, which we consider important future directions. Further, we study only the planning problem, i.e., assume to have access to exact gradients, and hence a combination of our study of NPG methods in state-action space with estimation problems would be a natural extension. 

\paragraph{Acknowledgments}
This project has been supported by ERC Starting Grant 757983 and DFG SPP 2298 Grant 464109215. 
GM has been supported by NSF CAREER Award DMS-2145630. 
JM acknowledges support from the International Max Planck Research School for Mathematics in the Sciences (IMPRS MiS) and the Evangelisches Studienwerk Villigst e.V.. 

\paragraph{Conflict of interest statement}
There is no conflict of interest.

\paragraph{Data availability statement}
A repository with computer code to reproduce the experiments will be made available. 

\bibliography{bibliography}
\bibliographystyle{plain}

\end{document}